\documentclass[11pt,a4paper]{amsart}
\usepackage{amsthm, amssymb}
\usepackage[T1]{fontenc}  
\usepackage[utf8]{inputenc}     
\usepackage{xcolor}
\usepackage{enumerate}
\usepackage[
  pdftitle={Roe algebras of coarse spaces via coarse geometric modules},
  pdfauthor={Diego Mart\'{i}nez, Federico Vigolo},
  pdfsubject={Mathematics},
  colorlinks=true, linkcolor=blue, linkbordercolor=blue, citecolor=red, citebordercolor=red, urlcolor=blue, linktocpage=true
]{hyperref}
\usepackage[capitalise, nameinlink, noabbrev, nosort]{cleveref} 
\usepackage[lite,initials]{amsrefs}

\usepackage{mathtools}          
\usepackage{microtype}          
\usepackage{bbm}          
\usepackage{bm}           
\usepackage[shortcuts]{extdash}       
\usepackage{subcaption}         
\usepackage{accents}
\usepackage{xspace}

\usepackage{xparse} 

\usepackage{tikz-cd}

\newcommand*{\cfRigidityMultiplier}{\cite{rigid}*{Corollary~9.14 and Remark~9.15}\xspace}

\usepackage[inline]{enumitem}
\setlist[enumerate,1]{font=\normalfont}
\crefname{enumi}{}{} 
\crefname{enumi}{}{} 

\newcommand{\supp}{{\rm Supp}}

\newcommand{\CHx}{{\CH_{\crse X}}}
\newcommand{\CHy}{{\CH_{\crse Y}}}
\newcommand{\CHz}{{\CH_{\crse Z}}}
\newcommand*{\sep}[1]{\mathrel{\perp_{#1}}}
\newcommand{\C}{\mathbb{C}}

\makeatletter
\def\appmap{\@ifnextchar[{\@withappmap}{\@withoutappmap}}
  \def\@withappmap[#1]#2#3#4{{{f}^{#1}_{#2, #3, #4}}}
  \def\@withoutappmap#1#2#3{{{f}_{#1, #2, #3}}}
\def\cappmap{\@ifnextchar[{\@withcappmap}{\@withoutcappmap}}
  \def\@withcappmap[#1]#2#3#4{{\crse{f}^{#1}_{#2, #3, #4}}}
  \def\@withoutcappmap#1#2#3{{\crse{f}_{#1, #2, #3}}}
\def\crseimg{\@ifnextchar[{\@withcrseimg}{\@withoutcrseimg}}
  \def\@withcrseimg[#1]#2#3{{I^{#1}_{#2, #3}}}
  \def\@withoutcrseimg#1#2{{I_{#1, #2}}}
\makeatother

\newcommand*{\cstar}{\texorpdfstring{$C^*$\nobreakdash-\hspace{0pt}}{*-}}
\newcommand*{\Star}{\(^*\)\nobreakdash-}

\newcommand*{\multiplieralg}[1]{\CM{\left(#1\right)}}

\newcommand*{\roelike}[1]{\CR_{}{\left[#1\right]}}
\newcommand*{\uroecstar}[1]{C^*_{\rm u}{\left(#1\right)}}

\newcommand*{\lccstar}[1]{C^*_{{\rm lc}}{\left(#1\right)}}
\makeatletter
  \def\roeclike{\@ifnextchar[{\@withroeclike}{\@withoutroeclike}}
    \def\@withroeclike[#1]#2{\CR^*_{#1}{\left(#2\right)}}
    \def\@withoutroeclike#1{\CR^*{\left(#1\right)}}
  \def\roeclikeone{\@ifnextchar[{\@withroeclikeone}{\@withoutroeclikeone}}
    \def\@withroeclikeone[#1]#2{\CR^*_{#1}{\left(#2\right)}}
    \def\@withoutroeclikeone#1{\CR^*_1{\left(#1\right)}}
  \def\roecliketwo{\@ifnextchar[{\@withroecliketwo}{\@withoutroecliketwo}}
    \def\@withroecliketwo[#1]#2{\CR^*_{#1}{\left(#2\right)}}
    \def\@withoutroecliketwo#1{\CR^*_2{\left(#1\right)}}
  \def\roecstar{\@ifnextchar[{\@withroecstar}{\@withoutroecstar}}
    \def\@withroecstar[#1]#2{C^*_{#1,\rm Roe}{\left(#2\right)}}
    \def\@withoutroecstar#1{C^*_{\rm Roe}{\left(#1\right)}}
  \def\roestar{\@ifnextchar[{\@withroestar}{\@withoutroestar}}
    \def\@withroestar[#1]#2{\C_{#1,\rm Roe}{\left[#2\right]}}
    \def\@withoutroestar#1{\C_{\rm Roe}{\left[#1\right]}}
  \def\classicalroecstar{\@ifnextchar[{\@withclassicalroecstar}{\@withoutclassicalroecstar}}
    \def\@withclassicalroecstar[#1]#2{C^*{\left(#2\right)}}
    \def\@withoutclassicalroecstar#1{C^*{\left(#1\right)}}
  \def\cpcstar{\@ifnextchar[{\@withcpcstar}{\@withoutcpcstar}}
    \def\@withcpcstar[#1]#2{C^*_{#1,\rm cp}{\left(#2\right)}}
    \def\@withoutcpcstar#1{C^*_{\rm cp}{\left(#1\right)}}
  \def\cpstar{\@ifnextchar[{\@withcpstar}{\@withoutcpstar}}
    \def\@withcpstar[#1]#2{\C_{#1,\rm cp}{\left[#2\right]}}
    \def\@withoutcpstar#1{\C_{\rm cp}{\left[#1\right]}}
  \def\qlcstar{\@ifnextchar[{\@withqlcstar}{\@withoutqlcstar}}
    \def\@withqlcstar[#1]#2{C^*_{#1,\rm ql}{\left(#2\right)}}
    \def\@withoutqlcstar#1{C^*_{\rm ql}{\left(#1\right)}}
\makeatother

\newcommand*{\cartansubalgroe}[2]{\ell^{\infty}_{\text{Roe}} {\left(#1; #2\right)}}
\newcommand*{\cartansubalgcp}[2]{\ell^{\infty}_{\text{cp}} {\left(#1; #2\right)}}
\newcommand*{\cartansubalgql}[2]{\ell^{\infty}_{\text{ql}} {\left(#1; #2\right)}}
\newcommand*{\cartansubalgroelike}[2]{\ell^{\infty}_{\CR} {\left(#1; #2\right)}}
\newcommand*{\condexp}[1]{E^{#1}}
\newcommand*{\condexproelike}[1]{E^{#1}_{\CR}}
\newcommand*{\condexpql}[1]{\condexp{#1}_{\rm ql}}
\newcommand*{\condexpcp}[1]{\condexp{#1}_{\rm cp}}
\newcommand*{\condexproe}[1]{\condexp{#1}_{\rm Roe}}

\newcommand*{\discdec}{\fka}

\newcommand*{\coe}[1]{{\rm CE}{\left(#1\right)}}

\newcommand*{\cuni}[1]{{\rm CtrUni}{{\left(#1\right)}}}

\newcommand*{\gr}[1]{{\rm Graph}(#1)}
\newcommand*{\grop}[1]{\op{\gr{#1}}}

\newcommand*{\charfunc}[1]{\mathbbm{1}_{#1}}                        \newcommand*{\chf}[1]{\charfunc{#1}}
\newcommand*{\charfunccomp}[2]{\charfunc{#1 \smallsetminus #2}}     
\newcommand*{\charfunccompX}[1]{\charfunccomp{X}{#1}}               \newcommand*{\chfcX}[1]{\charfunccompX{#1}}

\newcommand{\gauge}{\ensuremath{\widetilde{E}}{}}
\newcommand{\gaugex}{\ensuremath{\widetilde{E}_{\crse X}}{}}
\newcommand{\gaugey}{\ensuremath{\widetilde{F}_{\crse Y}}{}}

\newcommand*{\matrixunit}[2]{e_{#1,#2}}

\definecolor{darkgreen}{rgb}{0.0, 0.5, 0.0}
\newcounter{fcomment}

\definecolor{brown}{rgb}{0.55, 0.3, 0.25}
\newcounter{dcomment}

\definecolor{newpurple}{rgb}{0.8, 0, 0.9}

\newcommand{\CCC}{\mathbb{C}}

		\newcommand{\NN}{\mathbb{N}}
		
		\newcommand{\RR}{\mathbb{R}}

		\newcommand{\ZZ}{\mathbb{Z}}

		\newcommand{\CB}{\mathcal{B}}
		\newcommand{\CD}{\mathcal{D}}
\newcommand{\CE}{\mathcal{E}}		\newcommand{\CF}{\mathcal{F}}
\newcommand{\CG}{\mathcal{G}}		\newcommand{\CH}{\mathcal{H}}
		
\newcommand{\CK}{\mathcal{K}}		
\newcommand{\CM}{\mathcal{M}}		
\newcommand{\CO}{\mathcal{O}}		\newcommand{\CP}{\mathcal{P}}
		\newcommand{\CR}{\mathcal{R}}

		\newcommand{\CX}{\mathcal{X}}

\newcommand{\fka}{\mathfrak{a}}

\newcommand{\fkA}{\mathfrak{A}}		\newcommand{\fkB}{\mathfrak{B}}

\newcommand{\fkI}{\mathfrak{I}}

\DeclarePairedDelimiter\abs{\lvert}{\rvert}		
\DeclarePairedDelimiter\norm{\lVert}{\rVert}		
\DeclarePairedDelimiter\angles{\langle}{\rangle}	

\DeclarePairedDelimiter\paren{(}{)}			
\DeclarePairedDelimiter\braces{\{}{\}}			

	\newcommand{\bigparen}[1]{\paren[\big]{#1}}
\newcommand{\Bigpar}[1]{\paren[\Big]{#1}}	\newcommand{\Bigparen}[1]{\paren[\Big]{#1}}

	\newcommand{\bigbraces}[1]{\braces[\big]{#1}}
	
\newcommand{\bigmid}{\mathrel{\big|}}			

\newcommand{\scal}[2]{\angles{#1,#2}}			

\DeclareMathOperator{\id}{id}				

\DeclareMathOperator{\aut}{Aut}				
\DeclareMathOperator{\out}{Out}				

\DeclareMathOperator{\diam}{diam}			
\DeclareMathOperator{\image}{im}

\DeclareMathOperator{\Out}{Out}

\DeclareMathOperator{\rank}{rank}

\theoremstyle{plain}
\newtheorem{thm}{Theorem}[section]		\newtheorem{theorem}[thm]{Theorem}
		\newtheorem{proposition}[thm]{Proposition}
			\newtheorem{lemma}[thm]{Lemma}
		\newtheorem{corollary}[thm]{Corollary}

\newtheorem{warning}[thm]{Warning}

\newtheorem*{thm*}{Theorem}			\newtheorem*{theorem*}{Theorem}
\newtheorem*{prop*}{Proposition}		\newtheorem*{proposition*}{Proposition}
\newtheorem*{lem*}{Lemma}			\newtheorem*{lemma*}{Lemma}
\newtheorem*{cor*}{Corollary}			\newtheorem*{corollary*}{Corollary}
\newtheorem*{qu*}{Question}			\newtheorem*{question*}{Question}
\newtheorem*{conj*}{Conjecture}			\newtheorem*{conjecture*}{Question}
\newtheorem*{fact*}{Fact}
\newtheorem*{claim*}{Claim}
\newtheorem*{case*}{Case}
\newtheorem*{problem*}{Problem}

\numberwithin{equation}{section}

\theoremstyle{definition}
		\newtheorem{definition}[thm]{Definition}
\newtheorem{notation}[thm]{Notation}
		\newtheorem{convention}[thm]{Convention}

\newtheorem*{de*}{Definition}			\newtheorem{definition*}{Definition}
\newtheorem*{notation*}{Notation}
\newtheorem*{conv*}{Convention}			\newtheorem*{convention*}{Convention}

\theoremstyle{remark}
			\newtheorem{remark}[thm]{Remark}
\newtheorem{exmp}[thm]{Example}			\newtheorem{example}[thm]{Example}

\newtheorem*{remark*}{Remark}

\DeclareMathAlphabet{\mathbit}{OT1}{cmr}{bx}{it}

\crefname{thm}{Theorem}{Theorems}              \crefname{theorem}{Theorem}{Theorems}
\crefname{prop}{Proposition}{Propositions}     \crefname{proposition}{Proposition}{Propositions}
\crefname{lem}{Lemma}{Lemmas}                  \crefname{lemma}{Lemma}{Lemmas}
\crefname{rmk}{Remark}{Remarks}                \crefname{remark}{Remark}{Remarks}
\crefname{cor}{Corollary}{Corollaries}         \crefname{corollary}{Corollary}{Corollaries}
\crefname{qu}{Question}{Questions}             \crefname{question}{Question}{Questions}
\crefname{conj}{Conjecture}{Conjectures}       \crefname{conjecture}{Conjecture}{Conjectures}
\crefname{fact}{Fact}{Facts}
\crefname{claim}{Claim}{Claims}
\crefname{case}{Case}{Cases}
\crefname{alphthm}{Theorem}{Theorems}          \crefname{alphcor}{Corollary}{Corollaries}
\crefname{alphprop}{Proposition}{Propositions}

\ExplSyntaxOn
\cs_new_protected:Nn \egreg_embolden_command:N
 {\cs_set_eq:cN { __egreg_ \cs_to_str:N #1 : } #1
  \cs_set_protected:Npn #1 { \bm { \use:c { __egreg_ \cs_to_str:N #1 : } } }}
\cs_new_protected:Nn \egreg_embolden_char:n
 {\exp_args:Nc \mathchardef { __egreg_#1: } = \mathcode`#1 \scan_stop:
  \cs_set_protected:cn { __egreg_#1_bold: } { \bm { \use:c { __egreg_#1: } } }
  \char_set_active_eq:nc { `#1 } { __egreg_#1_bold: }
  \mathcode`#1 = "8000 \scan_stop:}
\cs_new_protected:Nn \egreg_embolden:
 {\clist_map_function:nN
   {
    A,B,C,D,E,F,G,H,I,J,K,L,M,N,O,P,Q,R,S,T,U,V,W,X,Y,Z,
    a,b,c,d,e,f,g,h,i,j,k,l,m,n,o,p,q,r,s,t,u,v,w,x,y,z,
    *,<,>,/,-,1,2,3,4,5,6,7,8,9,0
   }
   \egreg_embolden_char:n
  \clist_map_function:nN
   {
    \alpha,\beta,\gamma,\delta,\epsilon,\varepsilon,\zeta,\eta,\theta,\vartheta,\iota,\kappa,\lambda,\mu,\nu,\xi,\pi,\varpi,\rho,\varrho,\sigma,\varsigma,\tau,\upsilon,\phi,\varphi,\chi,\psi,\omega,
    \Gamma,\varGamma,\Delta,\varDelta,\Theta,\varTheta,\Lambda,\varLambda,\Xi,\varXi,\Pi,\varPi,\Sigma,\varSigma,\Upsilon,\varUpsilon,\Phi,\varPhi,\Psi,\varPsi,\Omega,\varOmega,
    \cup,\cap,
    \neq,\cong,\ncong,
    \in,
    \subset,\subseteq,\supset,\supseteq,\subsetneq,\supsetneq,\nsubseteq,\nsupseteq,
    \leq,\geq,\lneq,\gneq,\lhd,\rhd,\trianglelefteq,\trianglerighteq,\triangleright,\trianglerighteq,\circ,
   }
   \egreg_embolden_command:N}
\NewDocumentCommand{\crse}{m}
 {\group_begin:
  \egreg_embolden:
  #1
  \group_end:}
\ExplSyntaxOff

\newcommand{\cid}{\mathbf{id}}

\DeclareMathOperator{\cim}{\mathbf{im}}

\DeclareMathOperator{\dom}{dom}
\DeclareMathOperator{\cdom}{\mathbf{dom}}
\DeclareMathOperator{\csupp}{\mathbf{Supp}}

\DeclareMathOperator{\esssupp}{ess\, Supp}
\DeclareMathOperator*{\esssup}{ess\,sup}

\newcommand*{\op}[1]{#1^{\scriptscriptstyle\rm T}}

\makeatletter
\def\varcrs{\@ifnextchar[{\@withvarcrs}{\@withoutvarcrs}}
\def\@withvarcrs[#1]#2{\CE_{\rm #2}^{\rm #1}}
\def\@withoutvarcrs#1{\CE_{\rm #1}}
\def\varFcrs{\@ifnextchar[{\@withvarFcrs}{\@withoutvarFcrs}}
\def\@withvarFcrs[#1]#2{\CF_{\rm #2}^{\rm #1}}
\def\@withoutvarFcrs#1{\CF_{\rm #1}}
\makeatother

\newcommand*{\torel}[1]{\mathrel{{\mhyphen}\boxed{\vphantom{|}\,{\scriptstyle #1}\,}{\rightarrow}}}

\newcommand{\csub}{
                    \preccurlyeq}

\newcommand{\Ad}{{\rm Ad}}

\newcommand{\Cat}[1]{\ifmmode \text{\normalfont \textbf{#1}} \else {\normalfont \textbf{#1}}\fi}
\newcommand{\mhyphen}{\textnormal{-}}
\newcommand{\variable}{\,\mhyphen\,}

\begin{document}

\title[Roe algebras of coarse spaces]{Roe algebras of coarse spaces via coarse geometric modules}
\date{\today}

\author[Diego Mart\'{i}nez]{Diego Mart\'{i}nez $^{1}$}
\address{Mathematisches Institut, University of M\"{u}nster, Einsteinstr. 62, 48149 M\"{u}nster, Germany.}
\email{diego.martinez@uni-muenster.de}

\author[Federico Vigolo]{Federico Vigolo $^{2}$}
\address{Mathematisches Institut, Georg-August-Universit\"{a}t G\"{o}ttingen, Bunsenstr. 3-5, 37073 G\"{o}ttingen, Germany.}
\email{federico.vigolo@uni-goettingen.de}

\begin{abstract}
  We provide a construction of Roe (\cstar{})algebras of general coarse spaces in terms of coarse geometric modules. This extends the classical theory of Roe algebras of metric spaces and gives a unified framework to deal with either uniform or non-uniform Roe algebras, algebras of operators of controlled propagation, and algebras of quasi-local operators; both in the metric and general coarse geometric setting.
  The key new definitions are that of coarse geometric module and coarse support of operators between coarse geometric modules. These let us construct natural bridges between coarse geometry and operator algebras.
  
  We then study the general structure of Roe-like algebras, and investigate several structural properties, such as admitting Cartan subalgebras or computing their intersection with the compact operators.
  Lastly, we prove that assigning to a coarse space the K-theory groups of its Roe algebra(s) is a natural functorial operation.
\end{abstract}

\subjclass[2020]{46L80, 54E15, 46L85, 51K05, 51F30}
\keywords{Roe algebras; coarse geometry; geometric module}

\thanks{{$^{1}$} Funded by the Deutsche Forschungsgemeinschaft (DFG, German Research Foundation) under Germany’s Excellence Strategy – EXC 2044 – 390685587, Mathematics Münster – Dynamics – Geometry – Structure; the Deutsche Forschungsgemeinschaft (DFG, German Research Foundation) – Project-ID 427320536 – SFB 1442, and ERC Advanced Grant 834267 - AMAREC}

\thanks{{$^{2}$} Funded by the Deutsche Forschungsgemeinschaft (DFG) -- GRK 2491:  Fourier Analysis and Spectral Theory -- Project-ID 398436923}

\maketitle

\section{Introduction}
Roe algebras originally arose in the context of smooth (complete Riemannian) manifolds, as their K-theory groups are ideal to define the (higher) index of pseudo-differential operators on non-compact manifolds, and thus extend classical index theory beyond the compact case~\cites{RoeIndexI,RoeIndexII,roe-1993-coarse-cohom,roe_index_1996}. 
It was soon realized that these \cstar{}algebras are objects of interest also beyond their applications to index theory, and they have been thoroughly investigated in their own right (see~\cites{roe_lectures_2003,braga_rigid_unif_roe_2022,li-khukhro-vigolo-zhang-2021,guentner-et-al-2012-geometric-complex,willett-2009-some-notes-prop-a,sako-2014-proper-a-oper-norm}, among many others).
Moreover, these Roe algebras were initially defined as concretely represented algebras of operators on $L^2$-spaces of a manifold $M$ generated by operators with bounded displacement.
One key feature of Roe algebras is that their isomorphism class only depends on the \emph{coarse geometry} of the space (more on this below). As a consequence, in the construction of Roe algebras one may replace a smooth manifold $M$ and the associated Hilbert space $L^2(M,{\rm Vol})$ with an appropriately chosen discrete metric space $X$ and a Hilbert space of the form $\ell^2(X;\CH)$, where $\CH$ is some fixed Hilbert space. Both points of view have been used extensively, as they present different technical advantages.

On one hand, the smooth framework is the most natural setting for topologists. Indeed, most of the ``geometrically meaningful'' operators arise as (pseudo) differential operators on spaces of functions, differential forms or, more generally, sections of bundles on manifolds.
One of the most impressive applications of Roe algebras in this setting is concerning the interplay between analytic and homotopical/topological invariants of manifolds. 
Namely, Higson and Roe defined the \emph{coarse assembly map}, connecting $K$-homology groups with $K$-theory of Roe algebras~\cites{higson-roe-1995-coarse-bc}. The properties of these homomorphisms have strong consequences on the topology of the space. For instance, proving injectivity of the coarse assembly maps is one powerful approach to prove the strong Novikov Conjecture~\cites{higson-et-al-1997-contr-top,yu-1998-novikov-groups,yu_coarse_2000}.
Inspiration for this subject comes from the {Baum--Connes Conjecture}~\cites{baum_classifying_1994,baum_geometric_2000,yu-1995-bc-conjecture}, which is one of the main open problems of non-commutative geometry and has strong analogies with the Farrell--Jones conjecture, see~\cites{land-nikolaus-2018-k-l-theories,Luck2005}.

On the other hand, the discrete framework is much more convenient when studying structural properties of these \cstar{}algebras and their relations with various coarse geometric properties of the space.
Some of the analytic properties of Roe algebras that admit natural coarse-geometric equivalents are being \emph{quasi-diagonal}~\cites{chyuan_chung_inv_sem_2022,li-willett-2018-low-dimen}, admitting \emph{traces}~\cites{ara-lledo-martinez-2020,lledo-martinez-2021}, or being \emph{nuclear}~\cites{brodzki_property_2007,brown_c*-algebras_2008,sako-2014-proper-a-oper-norm}.
It should be pointed out that the object of interest in this analytic study is not only the Roe algebra proper, but a whole suite of algebras that can be defined in a similar spirit (the \emph{uniform Roe algebra}, the \emph{algebra of quasi-local operators}, etc.). For the rest of this introduction we will refer to them rather informally as \emph{Roe-like \cstar{}algebras} (see \cref{notation:roe-like-algs} for a precise meaning).

One very convenient feature in the study of Roe-like \cstar{}algebras is the presence of a unified framework that encompasses both the smooth and the discrete worlds. Namely, the language of \emph{geometric modules} as developed by Higson and Roe (see \cites{higson-2000-analytic-k-hom,willett_higher_2020}). In this language $L^2(M,{\rm Vol})$ and $\ell^2(X;\CH)$ can be seen as different choices of geometric modules for the same locally compact metric space. Any choice of (ample) geometric module for a fixed locally compact metric space $X$ can be used to define an associated Roe algebra, and the crucial observation is that different choices yield, in fact, isomorphic \cstar{}algebras. This point of view enables one to move seamlessly between the smooth and discrete settings.

\smallskip

As already mentioned, an extremely important property of Roe algebras is that they only depend on the \emph{coarse geometry} of the locally compact metric space.
Coarse geometry is the paradigm of studying the large scale geometry of a metric space ignoring all the local information on the space itself. Namely, the coarse geometric features of a metric space are those properties that are preserved under uniformly bounded---but arbitrarily large---perturbations of the space.
Gromov is one of the pioneers of the modern approach to large scale geometry, and his seminal work \cite{gromov-1993-invariants} is one of the birth places of Geometric Group Theory (we refer to the books \cites{dructu2018geometric,nowak-yu-book,roe_lectures_2003} for entry points to this subject). This point of view gives a handle to tackle otherwise untractable problems, and it is hard to overstate the importance of its consequences.

One fundamental observation is then that the isomorphism type of the Roe algebra is one such coarse geometric property. It is a deep and surprising rigidity phenomenon that, in fact, the converse is often true. Namely, starting with \cite{spakula_rigidity_2013}, a sequence of results has shown in more and more general settings that `tame' metric spaces having isomorphic Roe-like algebras must be coarsely equivalent \cites{braga_rigid_unif_roe_2022,bbfvw_2023_embeddings_vna,braga_farah_rig_2021,braga_gelfand_duality_2022}. In fact, one of the collateral goals of the present work is to set up the groundwork that will be used in \cite{rigid} to prove that this rigidity phenomenon applies to all locally compact (extended) metric spaces.

\smallskip

In a more recent trend, more and more attention has been given to a more general kind of coarse geometry. In order to study the large scale features of a metric space $(X,d)$ it is not necessary to make full use of the metric structure of $X$: all that is really needed is to know when two arbitrary families of points are at uniformly bounded distance from one another (equivalently, what families of subsets of $X$ have uniformly bounded diameter).
A useful axiomatization of this notion was developed by Roe in \cite{roe_lectures_2003}\footnote{\,
A different but equivalent approach was simultaneously developed by the Ukrainian school \cite{protasov2003ball}.
}: a coarse structure $\CE$ on a set $X$ is a gadget that fulfills the role of describing a notion of uniform boundedness on $X$ (this axiomatization is akin to the classical notion of \emph{uniform structure}, used to generalize the notion of uniform convergence to topological spaces without metrics).
Naturally, a metric $d$ on $X$ defines a coarse structure $\CE_d$, and different metrics are coarsely equivalent if and only if they generate the same coarse structure.
In general, a \emph{coarse space} $(X,\CE)$ is any set equipped with a coarse structure.

The need to deal with general coarse spaces arises rather naturally. Roe-like algebras on coarse structures are used both in the smooth setting \cites{higson-et-al-1997-contr-top,wright2003c0,wright2005coarse}, where they are useful intermediate constructs, and in the discrete setting \cites{chen-wang-2004,bunke-et-al-2020-k-theory,bunke-engel-2020-ass-maps,braga_farah_rig_2021,winkel2021geometric}, where they are often studied in their own right or in connection with the coarse Baum--Connes conjecture.
Additionally, coarse structures arise naturally in the study of non-compactly generated topological groups: an attempt to generalize the techniques and results of geometric group theory to this setting will likely need to follow this route \cites{coarse_groups,rosendal2021coarse}.

However, up to this date a unified approach to the smooth and discrete settings in the general coarse geometric world is still lacking. This is because the original treatment in terms of geometric modules relies heavily on the topology of the locally compact metric space, while general coarse spaces do not come equipped with a topology. One of the main goals of the present paper is to fill this gap and provide a common framework for both flavors of Roe-like algebras via \emph{coarse geometric modules}. Partial work in this direction is implicit in \cite{roe_lectures_2003} and also in \cite{winkel2021geometric}, where a notion of `geometric property (T)' for coarse spaces is considered. The approach we provide here is however much more general. In fact, one of our main contributions is identifying a somewhat minimal set of conditions necessary for the classical theory of Roe-like algebras to extend to the general coarse setting.
It should also be pointed out that there is a completely different way of providing a unified treatment in terms of C*-categories \cite{bunke2020homotopy}. We refer to \cref{sec:differences with Bunke-Engel} for a comparison between the two approaches.
Another main goal of this work is to introduce a useful set of nomenclature and conventions that is particularly well adapted to studying Roe-like algebras of coarse spaces in terms of coarse geometric modules. This is used heavily in \cite{rigid}.
We also collect a few useful but not so well known facts about Roe like algebras and coarse spaces.

Below is an outline of the paper and the main results it contains.

\subsection{Coarse geometry}
After a few general preliminaries, we start the core of our discussion with an introduction to coarse geometry on general coarse spaces in \cref{subsec:coarse-spaces}.
Before entering into detail, we should quickly recall the notation and convention developed in \cite{coarse_groups}, as they will be used heavily henceforth.
A coarse space will be denoted by $\crse{X} = (X, \CE)$, where $X$ is the space itself and $\CE$ is the coarse structure $X$ is equipped with. Since coarse geometry is concerned with the study of those properties that are preserved under uniformly bounded perturbations (recall that uniform boundedness is defined using the coarse structure), the only maps between coarse spaces that makes sense to consider are those functions $f\colon X\to Y$ preserving uniform boundedness. Such functions are called \emph{controlled}. Examples of such controlled maps are Lipschitz functions of metric spaces, or functions that are Lipschitz up to an additive constant.

In the ``coarse philosophy'' everything is only defined up to uniformly bounded perturbations. In particular, if two functions $f,f'\colon X\to Y$ are \emph{close} (that is, $f(x)$ stays uniformly close to $f'(x)$ when $x \in X$ varies) then they should be considered as representing the same map. We make this into a hard definition. Indeed, a \emph{coarse map} $\crse{f\colon X\to Y}$ is the equivalence  class of a controlled function up to closeness.\footnote{\,
This nomenclature diverges from that of \cite{roe_lectures_2003}, where coarse maps are defined as \emph{proper} controlled functions.}
This is but a particular case of a general convention: bold symbols are used throughout to denote coarse properties and constructs (\emph{i.e.}\ notions that are stable under uniformly bounded perturbations). Normal-weight symbols, on the other hand, denote pointwise defined notions, or a choice of representative for a coarsely defined object. For instance, $\crse{f\colon X\to Y}$ is a coarse map, while $f\colon X\to Y$ is a controlled map that is a representative for $\crse f$ (\emph{i.e.}\ $\crse f=[f]$, see \cref{def:coarse-map}).
A \emph{coarse equivalence} is a coarse map $\crse{f\colon X\to Y}$ with a coarse inverse $\crse{g\colon Y\to X}$, that is, such that the compositions ${f\circ g}$ and ${g\circ f}$ are close to the corresponding identity functions. If a coarse equivalence exists, $\crse X$ and $\crse Y$ are \emph{coarsely equivalent}. A property is a \emph{coarse property} if it is preserved under coarse equivalences.

This setup can be given a rigorous categorical framework: the category \Cat{Coarse} has coarse spaces as objects and coarse maps as morphisms. Coarse equivalences are precisely the isomorphisms in the coarse category, and coarse properties are, precisely, those notions that can be defined within \Cat{Coarse}.

In this work we take one step further, in that instead of considering controlled functions $f\colon X\to Y$, we define coarse maps in terms of \emph{controlled relations} $R\subseteq Y\times X$ (cf.\ \cref{def:controlled_function}). Informally, such an $R$ can be thought of as a function such that the image $R(x)$ is only determined up to uniformly bounded error, and hence will be well defined in \Cat{Coarse}.
This point of view enables us to define \emph{(partial) coarse maps} as coarse subspaces of the product $\crse{R\subseteq Y\times X}$ (cf.\ \cref{def:coarse-map}) which, in turn, yields an equivalent description of the coarse category (see \cref{lemma:coarse-relation-yields-coarse-map}). The main advantage of defining coarse maps in terms of relations will become clear when discussing the \emph{coarse support} of operators.

In \cref{subsec:coarse-spaces} we give a precise definition for all the notions mentioned above.
In particular, we discuss coarse subspaces; partially defined coarse maps, with their domain and image; compositions of (partial) coarse maps; properness and coarse embeddings. We also review some properties of coarse spaces. Namely, a coarse space $\crse X$ is \emph{coarsely locally finite} if it is coarsely equivalent to a coarse space $\crse Y$ where all the bounded sets are finite (this happens \emph{e.g.}\ if the coarse structure on $X$ is induced by a proper metric). Likewise, $\crse X$ has \emph{bounded geometry} if it is coarsely equivalent to a coarse space $\crse Y$ where uniformly bounded sets are \emph{uniformly} finite. In these cases, the coarse space $\crse Y$ is a (uniformly) locally finite model for $\crse X$ (see \cref{def:locally finite space - bounded geo}).
It is worth pointing out that these properties of $\crse X$ can be rephrased in terms of the existence of a (uniformly) locally finite partition of $X$ (as is, in fact, done in \cref{def:controlled-partition}).

\subsection{Geometric modules}
Classically, a geometric module for a locally compact metric space $(X,d)$ is a non-degenerate \Star{}representation $\rho \colon C_0(X) \to \CB(\CH_X)$ for some Hilbert space $\CH_X$ (see \cref{subsec:classical-roe-alg}). As explained, the dependence of this definition on the topology of $X$ is a major issue when dealing with arbitrary coarse spaces.
However, what one often really needs is not much the representation $\rho \colon C_0(X) \to \CB(\CH_X)$, but its canonical extension $\bar \rho$ to the algebra of \emph{Borel} bounded functions of $X$. In turn, all that is needed to describe $\bar\rho$ are the projections $\bar\rho(\chf{A})$ associated with the indicator functions of the measurable subsets $A\subseteq X$.
This latter point of view can be extended, and this is precisely what we do in \cref{sec:modules}.

A \emph{coarse geometric module for $\crse X$} or, \emph{$\crse{X}$\=/module}, (see \cref{def:coarse-module}) is a Hilbert space $\CHx$, together with a \emph{non-degenerate} unital representation $\chf{\bullet} \colon \fkA \to \CB(\CHx)$, where $\fkA$ is some Boolean algebra of subsets of $X$. In particular, $\paren{\chf{A}}_{A \in \fkA}$ is a family of commuting projections with $\chf{\emptyset}=0$ and $\chf{X}=1_\CHx$.
The non-degeneracy condition is a necessary extra requirement, which implies that the module $\CHx$ can be reconstructed from the images $\chf{A}(\CHx)$ as $A$ ranges among---suitably large---uniformly bounded sets in $\fkA$.
In particular, note that the choice of $\fkA$ is part of the definition of the module. If $X$ is seen as a discrete set, it makes sense to have $\fkA =\CP(X)$ and consider modules of the form $\ell^2(X,\CH)$. If $X$ has a topology, it is natural to take $\fkA$ to be the Borel $\sigma$-algebra and take modules of the form $L^2(X,\mu)$. 
In view of the latter example, we call the elements of $\fkA$ \emph{measurable subsets}.

We observe that a classical geometric module $\CH_X$ for a locally compact metric space $(X,d)$ naturally gives rise to a coarse geometric module $\CHx =\CH_X$ for the coarse space $\crse X = (X,\CE_d)$, where $\CE_d$ is the coarse structure induced by the metric $d$.

Non-degeneracy is the bare minimum that is needed for a representation to interact meaningfully with the coarse space $\crse X$. However, in many instances more stringent conditions are needed. For instance, it makes sense to ask every subset of $X$ to have a \emph{controlled} measurable neighborhood (cf.\ \cref{def:coarse-containments}). Such a module is called \emph{admissible} (the approach taken in \cite{winkel2021geometric} always yields admissible modules).
Other useful structural properties are \emph{ampleness} and, most importantly, \emph{discreteness} (see \cref{def:ample,def:discrete} respectively).

As it turns out, a discrete coarse geometric module is in spirit quite analogous to a module of the form $\ell^2(X,\CH)$ (see \cref{rmk:discrete modules are discrete}), hence the name.
In particular, all the results concerning Roe-like algebras of coarse spaces realized as operator algebras on $\ell^2(X,\CH)$ can be phrased in terms of discrete modules.
It is important to observe that the vast majority of modules that one may need to consider are indeed discrete. In fact, we prove that every (locally) admissible module on a coarsely locally finite coarse space has a canonical extension to a discrete module (see \cref{prop:admissible is discrete}).

\subsection{Operators on modules}\label{ssec:intro:operators on modules}
A useful feature of (classical) geometric modules is that the topology may be used to define the support of a bounded operator $t\colon \CH_X\to\CH_Y$ as a minimal closed subset of $Y\times X$ ``supporting $t$''.
In the coarse geometric setting this does not make sense, \emph{i.e.}\ if $t\colon \CHx\to \CHy$ is an operator between coarse geometric modules it is not possible to canonically define a support for $t$. Nevertheless, what one can do is to define its \emph{coarse support} $\csupp(t)$ as a coarse subspace of $\crse{Y\times X}$ (see \cref{def:coarse support,prop:coarse_support_exists}). Many of the key properties of the support of an operator apply more generally to coarse supports.

Here is where the choice of describing (partial) coarse functions as coarse subspaces of $\crse{Y\times X}$ pays off. In fact, it is now very natural to compare maps with coarse supports of operators. In particular, we say that an operator $t\colon \CHx\to \CHy$ is \emph{controlled} (see \cref{def:controlled operator}) if its coarse support is a partial coarse map $\csupp(t)\colon\crse X\to\crse Y$ (\emph{i.e.}\ its representatives are controlled relations in $Y\times X$). Controlled operators will play a key role in the sequel, as they are those operators that have the strongest geometric meaning.

Consider now operators in $\CB(\CHx)$. Here there are a special kind of controlled operators, namely those that are supported on the coarse identity function $\cid_{\crse X}\colon\crse{X\to X}$. Equivalently, these are operators whose support is contained in a controlled neighborhood of the diagonal. These operators are said to have \emph{controlled propagation} (see \cref{def: controlled propagation operator}). 
Heuristically, this means that the operators are of \emph{bounded displacement}.
It follows readily from the properties of the coarse support that the set of operators of controlled propagation is a \Star{}algebra.

The useful classical notion of local compactness for operators $t\in\CB(\CH_X)$ extends naturally to the coarse geometric setting. Namely, an operator $t\colon \CHx\to\CHy$ is \emph{locally compact} if $\chf{B}t$ and $t\chf{A}$ are compact for any choice of bounded measurable $A\subseteq X$, $B\subseteq Y$ (cf.\ \cref{def:locally-compact-ope}).
This is just another instance of our standing convention that the `local' properties are those that hold on bounded subsets.

There is also a convenient weakening of the bounded-propagation property, namely \emph{quasi-locality} (see \cref{def:ql-ope}).
In the metric setting, an operator $t\in\CB(\CHx)$ is quasi-local if for every $\varepsilon > 0$ there there is a radius $r>0$ such that $\norm{\chf{A'}t\chf{A}}\leq \varepsilon$ whenever $A',A\subset X$ are $r$-separated measurable subsets. The definition in the general coarse setting simply replaces $r$ by a controlled entourage $E\in\CE$. It is worth pointing out that quasi-locality works best on admissible coarse geometric modules.

\subsection{Roe-like (\cstar{})algebras}
All the pieces are now in place to define the main characters of this play. For a fixed coarse geometric module $\CHx$, we first note that the set of locally compact operators is a \cstar{}algebra, which we denote by $\lccstar{\CHx}$. We now define:
\begin{itemize}
  \item $\cpstar{\CHx}$ to be the algebra of operators of controlled propagation;
  \item $\roestar{\CHx}\coloneqq \cpstar{\CHx}\cap\lccstar{\CHx}$ is the \emph{Roe \Star{}algebra} of $\CHx$.
\end{itemize}
We refer to the above as \emph{Roe-like \Star{}algebras} (\cref{def:C*cp,def:roestar}). Taking their (norm) completions, we obtain \cstar{}algebras, which we denote by $\cpcstar{\CHx}$ and $\roecstar{\CHx}$ respectively.
Moreover, if $\CHx$ is admissible, the set of quasi-local operators is also a \cstar{}algebra, which we denote by $\qlcstar\CHx$ (cf.\ \cref{def:C*ql-lc}). We refer to these three \cstar{}algebras as the \emph{Roe-like \cstar{}algebras} of the coarse geometric module $\CHx$.

By design, these definitions extend the classical analogue notions. Namely, if $\CHx$ is obtained from a classical geometric module $\CH_X$, then $\roestar{\CHx}$ and $\roecstar\CHx$ are precisely the Roe \Star{} and \cstar{}algebras as defined in, \emph{e.g.}\ \cite{willett_higher_2020}.
One example of classical interest is when $X$ is a locally finite metric space and $\CHx=\ell^2(X)$. In this case $\cpcstar{\CHx}=\roecstar{\CHx}$, which is known as the \emph{uniform Roe algebra} of $X$. This \cstar{}algebra has been thoroughly studied, because it is more manageable than the usual Roe algebra defined using ample modules.
More generally, if $\crse X$ is a coarse space one may always consider the \emph{uniform module} $\CH_{u,\crse X}\coloneqq\ell^2(X)$ and $\cpcstar{\CH_{u,\crse X}}$ is then the associated uniform Roe-algebra.\footnote{\,
If $\crse X$ is not locally finite, then this is no longer equal to $\roecstar{\CH_{u,\crse X}}$. 
We find that $\cpcstar{\CH_{u,\crse X}}$ has more affinity to the classical uniform Roe-algebras, 
for instance because they are both unital, while $\roecstar{\CH_{u,\crse X}}$ needs not be.
}

\smallskip

We can then study some general structural properties of Roe-like algebras.
In the classical metric setting, it is an important feature that Roe-like \cstar{}algebras always contain the set of compact operators. This is no longer the case for, \emph{e.g.}, \emph{extended} metric spaces, because operators of finite propagation or approximations thereof can never ``mix'' parts of the space that are at infinite distance from one another.
In coarse theoretic terms, this means that some extra care is required if the coarse space is not \emph{coarsely connected} (see \cref{def: coarsely connected}). In this setting, a coarse space $\crse X$ can be decomposed as a disjoint union $\bigsqcup_{i\in I}\crse X_i$ of coarsely connected components.
If $\CHx$ is then a coarse module for $\crse X$ so that the coarse components are measurable, $\CHx$ can be decomposed as $\bigoplus_{i\in I}\CH_{\crse X_i}$, where $\CH_{\crse X_i} = \chf{X_i}(\CHx)$.
In this case $\CH_{\crse X_i}$ is itself a coarse geometric module for $\crse X_i$, and one may realize the Roe-like (\cstar)algebras of $\CHx$ as subalgebras of the product of the Roe-like (\cstar)algebras of $\CH_{\crse X_i}$ (see \cref{cor:roe algs disconnected,cor:roe algs disconnected vs sum and prod}).
It is then a simple matter to show that the intersection of any Roe-like \cstar{}algebra of $\CHx$ with the compact operators is the sum $\bigoplus_{i\in I}\CK(\CH_{\crse X_i})$ (see \cref{cor: roe algs intersect compacts}).
This is important in applications, particularly in~\cite{rigid}.

\smallskip

Observe that, by its very definition, it is clear that $\roecstar\CHx$ is always contained in the intersection $\cpcstar\CHx\cap\lccstar{\CHx}$. It is a very useful and not at all clear fact that the converse containment is often true. Namely, we prove that if $\crse X$ is coarsely locally finite and $\CHx$ is discrete then 
\begin{equation}\label{eq:intro: roe as intersection}
  \roecstar\CHx = \cpcstar\CHx\cap\lccstar{\CHx}
\end{equation}
(see \cref{thm:roe-alg:intersection}). This result is not even well-known in the classical metric setting: \cite{braga_gelfand_duality_2022} contains a proof for it in the bounded geometry case, and we were not able to locate other references for it. In particular, this fact appears to be new in the context of locally compact metric spaces.
We also refer to the discussion following \cref{thm:roe-alg:intersection} for an application of this fact.
The proof of \eqref{eq:intro: roe as intersection} is based on the construction of an approximate unit of $\cpcstar\CHx\cap\lccstar{\CHx}$ that is actually contained in $\roestar{\CHx}$. Such an approximate unit is useful in other contexts as well \cites{braga_gelfand_duality_2022,rigid}. It is also interesting to note that if $\crse X$ is also \emph{coarsely countable} then the same net of operators is also an approximate unit for the much larger algebra $\lccstar{\CHx}$ (see \cref{prop: approx unit of lccstar}).

\smallskip

Finally, \cref{subsec:cartan subalgebras} constructs natural \emph{(non-commutative) Cartan subalgebras} in Roe-like algebras (see \cref{thm:cartan-subalgs}). 
Recall that a (non-commutative) Cartan subalgebra of a \cstar{}algebra $A$ is a \cstar{}sub-algebra $B \subseteq A$ that is in some sense \emph{maximal} in $A$, while $A$ itself is \emph{minimal} with respect to $B$ (see \cref{def:cartan-subalg,def:nc-cartan-subalg}).
It was proved in~\cites{Renault2008CartanSI,exel-nc-cartan-subalgs-2011} that Cartan pairs always admit ``groupoid'' models, which yield convenient and explicit descriptions of such algebras. These descriptions have been tremendously useful, for instance, in the \emph{classification program} (see~\cite{li-classification-2020} and references therein).

In the classical setting, Cartan subalgebras of Roe algebras of uniformly locally finite metric space case are very well understood (see \cites{white_cartan_2018,braga_gelfand_duality_2022} and references therein).
In \cref{subsec:cartan subalgebras} we extend the construction to arbitrary coarse spaces and deal with both $\roecstar\variable$ and $\cpcstar\variable$.
Our interest in the existence of these Cartan subalgebras stems from the models given in \cites{Renault2008CartanSI,exel-nc-cartan-subalgs-2011}, because these models may be described in geometric terms from the coarse space $\crse{X}$. This provides a tool to prove structural properties of our Roe-like algebras, such as nuclearity (see \cref{cor:roe-models}). 

\subsection{Mappings among Roe-like algebras}
We now return to the topic of controlled operators, as discussed in \cref{ssec:intro:operators on modules}.
One crucial property is that if $t \colon \CHx \to \CHy$ is a controlled operator, then the natural mapping $\Ad(t)\colon\CB(\CHx)\to\CB(\CHy)$ obtained by conjugation $\Ad(t)(x)\coloneqq txt^*$ preserves controlled propagation properties. Namely, it restricts to mappings $\cpstar\CHx\to\cpstar\CHy$ and $\cpcstar\CHx\to\cpcstar\CHy$.
If the modules are also admissible, quasi-locality is preserved as well, whence a mapping $\qlcstar\CHx\to\qlcstar\CHy$ is obtained.
Controlledness of $t$ is not enough to deduce that $\Ad(t)$ preserves local compactness, because (partial) coarse maps are allowed to collapse arbitrarily large parts of a space into bounded sets. On the other hand, if $\csupp(t)\colon\crse{X\to Y}$ is also \emph{proper} (in the sense of \cref{def:proper}) then $\Ad(t)$ does restrict to a mapping $\roestar\CHx\to\roestar\CHy$ and $\roecstar\CHx\to\roecstar\CHy$ as well. These facts are collected in \cref{cor:coarse functions preserve roe-like algebras}.

Of course, if $t$ is an isometry (\emph{i.e.}\ $t^*t=1_{\CHx}$) then $\Ad(t)$ is a \Star{}embedding. In other words, taking Roe-like algebras and $\Ad$ provides us with functors from the category of coarse geometric modules and (proper) controlled isometries to the category of \cstar{}algebras (see \cref{rmk:coarse functions preserve roe-like functorially}).

It is an essential part of the (classical) theory of Roe-like algebras that there is a multitude of controlled operators.\footnote{\,
In fact, sufficiently many for these algebras to be faithful descriptions of the coarse geometry of coarse spaces, in a sense made precise from the rigidity theorems for Roe-like algebras mentioned before.
}
We prove in \cref{sec:covering iso} that this phenomenon holds in the general coarse setting as well, at least when considering discrete ample coarse geometric modules.

Informally, a module $\CHx$ is \emph{$\kappa$-ample} (cf.\ \cref{def:ample}) if the Hilbert spaces $\chf A(\CHx)$ have rank at least $\kappa$ for every large enough measurable bounded $A\subseteq X$ ($\kappa$ will usually be $\aleph_0$, but when working with large coarse spaces it makes sense to consider larger cardinals as well).
This notion is an extension of ampleness for classical geometric modules, and it is a condition necessary to ensure that the modules are large enough to construct interesting isometries between them.
In fact, we can then prove the following.
\begin{theorem}[cf.\ \cref{cor:isometries cover,thm:covering isometries exist}]\label{thm:intro covering isom}
   If $\CHx$ and $\CHy$ are both discrete and $\kappa$-ample of rank $\kappa$, then for every coarse map $\crse{f\colon X\to Y}$ there is a controlled isometry $t\colon\CHx\to\CHy$ with $\csupp(t)=\crse f$.
   Moreover, if $\crse f$ is a coarse equivalence, $t$ may be chosen to be a unitary.
\end{theorem}

An isometry as in \cref{thm:intro covering isom} is said to \emph{cover} the coarse map $\crse f$ (see \cref{def:covering}). We refer to \cref{thm:covering isometries exist} for a more refined statement about the existence of covering (partial) isometries under sharper conditions.
It is however important to note that the construction of such $t$ is highly non-canonical (and \emph{not} functorial).
Nevertheless, the choices involved in the construction of the isometry $t$ in \cref{thm:intro covering isom} only differ by uniformly bounded error (cf.\ \cref{lem: operators with same coarse support are close}).
In particular, in the case where $\crse{X} = \crse{Y}$ and $t$ is a unitary covering a coarse equivalence, these choices are made canonical if one quotients out by the normal subgroup of unitary operators with controlled propagation (see \cref{thm: coarse equivalences iso to cuni}).
This behavior implies that, as long as $\CHx$ is ample enough, the group $\coe{\crse X}$ of coarse equivalences of $\crse{X}$ is naturally mapped into the groups of \emph{outer} automorphisms of the Roe-like \cstar{}algebras. Moreover, these homomorphisms are often embeddings (see \cref{cor:injection-from-ce-to-out}).

\begin{remark}\label{rmk:intro: local ampleness}
  If $\crse f$ is a \emph{proper} coarse map, then the ampleness requirements of \cref{thm:intro covering isom} can be weakened. Namely, to guarantee the existence of an isometry covering $\crse f$ it is enough to assume that $\CHx$ and $\CHy$ are $\kappa$-ample of \emph{local} rank $\kappa$ (see \cref{rem:covering isos}\cref{rem:covering isos:coe-covered-loc-fin}).
\end{remark}

\subsection{Roe-like algebras of coarse spaces}
Finally, we discuss Roe-like algebras of \emph{spaces}, as opposed to modules. Given the theory that we developed, this is now a simple matter. In fact, for a fixed cardinal $\kappa$ we define the \emph{rank-$\kappa$ Roe-like (\cstar)algebras} of a coarse space $\crse X$ to be the Roe-like (\cstar)algebras of any $\kappa$-ample discrete $\crse X$-module of rank $\kappa$, if such a module exists. It then follows from \cref{thm:intro covering isom} that this definition is well-posed, as it does not depend on the choice of rank-$\kappa$ module.

Here a remark is in order: in the classical theory one only works with separable modules. This can be done because Roe-like algebras are usually only defined for locally compact metric spaces, and such a space always admits ample separable modules. On the other hand, separable modules are not large enough to deal with arbitrary coarse spaces: in fact, we show that a coarse space $\crse X$ admits a discrete ample coarse module of rank $\kappa$ if and only if the \emph{coarse cardinality} (cf.\ \cref{def:coarse cardinality}) of $\crse X$ is at most $\kappa$ (see \cref{cor: roe-like algebras vs coarse cardinality}).
In view of \cref{rmk:intro: local ampleness}, if one only wishes to work with proper coarse maps, one may modify the definition of Roe-like \cstar{}algebras to include modules that are $\kappa$-ample of local rank $\kappa$.\footnote{\,
One example of such a module that one may wish to consider is an ample and locally separable geometric module for a non-separable locally compact extended metric space.}
In this work we preferred to keep the freedom to use non-proper coarse maps, as we find this to be a more natural coarse geometric setup.

\smallskip

Finally, in \cref{sec:k-theory} we prove that, for every fixed $\kappa$, assigning to a coarse space $\crse{X}$ of coarse cardinality at most $\kappa$ the $K$-theory $K_*(\roecstar[\kappa]{\CH})$ yields a natural functor (see \cref{thm:k-theory-functor}). In particular, this fixes the aforementioned lack of functoriality in the construction of covering isometries. To be precise, in \cref{thm:k-theory-functor} we show that for every controlled proper isometry $t$ the homomorphisms 
\[
  (\Ad(t))_\ast\colon K_*(\roecstar[\kappa]\CHx)\to K_*(\roecstar[\kappa]\CHy)
\]
only depend on $\csupp(t)\colon \crse{X\to Y}$. This shows that the assignment of homomorphisms in $K$-theory to proper coarse maps is canonical and defines a functor into the category of graded abelian groups.
This is, in fact, the main reason Roe introduced classical Roe algebras back in~\cites{roe-1993-coarse-cohom,roe_index_1996}, and thus completes this circle of ideas.
\begin{notation}\label{general notation}
  We use the following conventions. Bold symbols denote coarse notions, while the non-bold versions are a choice of representatives. Specifically, $\crse{X} = (X, \CE)$ and $\crse{Y} = (Y, \CF)$ denote coarse spaces.
  Likewise, $\CHx$ and $\CHy$ are coarse geometric modules associated to $\crse X$ and $\crse Y$ respectively.

  Controlled entourages are usually denoted as $E \in \CE$ and $F \in \CF$. Measurable sets will be denoted by $A \subseteq X$ and $B \subseteq Y$. These will often (but not always) be bounded as well.
  $R$ and $S$ will usually be binary relations (subsets of $Y\times X$ or $X\times X$).
  
  $\CB(\CH)$ denotes the algebra of bounded linear operators on $\CH$, and $\CK(\CH)$ denotes the compact ones.
  The \emph{rank} of a Hilbert space $\CH$ is the cardinality of an orthonormal basis.
\end{notation}

\subsection*{Acknowledgements}
We are grateful for several fruitful conversations about these topics with  Ilijas Farah, Ralf Meyer, Alessandro Vignati and Wilhelm Winter.

\section{General preliminaries} \label{sec:pre}
This section introduces necessary notation and background information useful in the rest of the text.

\subsection{Relations} \label{subsec:relations}
A \emph{relation from $X$ to $Y$} (short for ``binary relation'') is any subset $R\subseteq Y\times X$.\footnote{%
\, The choice of taking $Y\times X$ rather than $X\times Y$ is justified by the unfortunate habit of composing to the left successive functions.}
We write $yRx$ for $(y,x)\in R$ (in \cite{coarse_groups} the symbol $y\torel{R}x$ is used).
A relation \emph{on} $X$ is a relation from $X$ to $X$. For every $A\subseteq X$, we let $\Delta_A\coloneqq \{(a,a)\mid a\in A\}\subseteq X\times X$ denote the \emph{diagonal over $A$}.

Given $A\subseteq X$, we let $R(A) \coloneqq \{y\mid \exists a\in A, yRa\}$ be the \emph{$R$-image} of $A$.
If $A=\{x\}$ is a singleton, we may write $R_x$ or $R(x)$ instead of $R(\{x\})$. These are also called the \emph{sections} of $R$.
Given $A\subseteq X$ and $B\subseteq Y$, we say \emph{$B$ is $R$-separated from $A$} (denoted by $B\sep{R}A$) if $B\cap R(A)=\emptyset$. Note that this is a symmetric relation whenever $R$ is symmetric.

We denote by $\op{R}\coloneqq\braces{(x,y)\mid yRx}$ the \emph{transposition} of $R$.
The \emph{composition} of two relations $R_2\subseteq Z\times Y$ and $R_1\subseteq Y\times X$ is the relation
\[
  R_2\circ R_1\coloneqq\braces{(z,x)\mid\exists y\in Y \;\; \text{such that} \;\; zR_2yR_1x}.
\]
Observe that $(R_2\circ R_1)(A)=R_2(R_1(A))$ for every $A\subseteq X$.
If $f\colon X\to Y$ is a function, then the transposition of the graph $\gr{f}$ is a relation from $X$ to $Y$.
The composition rule yields $\grop{g\circ f}=\grop{g}\circ \grop{f}$.

\

Given $E\subseteq Y\times X$ and $E'\subseteq Y'\times X'$, we denote by
\[
 E\otimes E'\coloneqq \braces{\paren{(y,y'),(x,x')}\mid yEx \; \text{and} \; y'E'x'}
\]
the \emph{product relation from $X\times X'$ to $Y\times Y'$} (we use $\otimes$ to differentiate it from the Cartesian product of sets $E\times E'\subseteq Y\times X\times Y'\times X'$). Observe that $\grop{f\times f'}=\grop{f}\otimes\grop{f'}$.
Note the product is well behaved with compositions, in the sense that
\begin{equation}\label{eq:product relation composition}
 (E\otimes E')\circ(F\otimes F')=(E\circ F)\otimes(E'\circ F').
\end{equation}
Moreover, if $D\subseteq X\times X'$ is a relation on $X\times X'$, then its \emph{image under $E \otimes E'$} is a relation on $Y\times Y'$ which can also be expressed as
\begin{equation}\label{eq:product relation image}
 (E\otimes E')(D)=E\circ D \circ \op{(E')}.
\end{equation}
Both \cref{eq:product relation composition,eq:product relation image} follow directly from the definitions.

\subsection{Operators, boolean algebras and joins} \label{subsec:operator-boolalg}
By \emph{projection} we systematically mean \emph{self-adjoint idempotent}.

\smallskip

A \emph{Boolean algebra of subsets of $X$} is a family $\fkA \leq \CP(X)$ that is closed under finite unions, intersections and complements. We say that $\fkA$ is \emph{unital} if it contains $X$.
Given a Hilbert space $\CH$, a \emph{representation of $\fkA$ on $\CH$} is a map $\chf{\bullet}\colon\fkA\to \CB(\CH)$ such that $\chf{\emptyset}=0$, $\chf{A\cap B}=\chf A\chf B$ and $\chf{A\cup B}=\chf A+\chf B -\chf A\chf B$. Observe that this yields that $\paren{\chf{A}}_{A \in \fkA}$ is a family of commuting projections. We say $\chf{\bullet}$ is \emph{unital} if $\chf{X}=1$. Note that this then forces $\chfcX{A}=1 -\chf{A}$.

Let $t_i\in\CB(\CH)$ with $i\in I$ be a family of bounded operators. Recall that their sum converges in the \emph{strong operator topology to an operator $t\in \CB(\CH)$} if for every $v\in \CH$ and $\varepsilon>0$ there is a finite $J\subseteq I$ such that for any finite set $K\subseteq I\smallsetminus J$
\[
  \norm{t(v) - \sum_{i\in K}t_i(v)} < \varepsilon.
\]
In the sequel we will often make use of \emph{SOT-sums} (or \emph{strongly convergent sums}) without further notice.

\smallskip

Two operators $s,t$ are \emph{orthogonal} if $st^*=s^*t=0$. If $\paren{t_i}_{i\in I}$ is a family of uniformly bounded pairwise orthogonal operators, then the SOT-sum $\sum_{i\in I} t_i$ exists and it is an operator of norm $\sup_{i\in I}\norm{t_i}$, which is bounded by assumption.

\smallskip

In a few places we will use the notation $\matrixunit{v'}{v}$ to denote the operator given by $\matrixunit{v'}{v} (h) \coloneqq \scal{v}{h} \, v'$. Observe that every rank-one operator $\CH\to\CH'$ can be expressed as $\matrixunit{v'}{v}$ for some $v\in \CH$ and $v'\in\CH'$.

Given a bounded operator $s\colon \CH\to\CH'$, we may consider its action by conjugation, that is, $\Ad(s)\colon\CB(\CH)\to\CB(\CH')$, which is defined by $\Ad(s)(t)\coloneqq sts^*$. Elementary computations show that $\Ad(s)\matrixunit{v'}{v}=\matrixunit{s(v')}{s(v)}.$ 

\smallskip

Recall that if $A\leq\CB(\CH)$ is a concretely represented \cstar{}algebra, its \emph{multiplier algebra} is
\[
  \multiplieralg{A}\coloneqq\braces{t\in\CB(\CH)\mid tA,At\subseteq A} 
\]
(that is, $\multiplieralg{A}$ is the largest idealizer of $A$ in $\CB(\CH)$). It is classical that $\multiplieralg{A}$ only depends on $A$, and not the representation of $A$ into $\CB(\CH)$.\footnote{\, $\multiplieralg{A}$ may also be defined independently of choosing a representation $A \subseteq \CB(\CH)$, but since all Roe-like algebras are concretely represented this approach is much better suited for the paper.}

Moreover, note that every unitary $u\in U(\multiplieralg{A})$ acts on $A$ by conjugation, \emph{i.e.}\ $\Ad$ defines a homomorphism $U(\multiplieralg{A})\to\aut(A)$. Such automorphisms are said to be \emph{inner}. 
The group of inner automorphisms is normal in $\aut(A)$, and the quotient 
\[
  \Out(A)\coloneqq \aut(A)/\{\text{inner}\}
\] 
is the group of \emph{outer automorphisms} of the \cstar{}algebra $A$.

\subsection{Classical Roe algebras of metric spaces} \label{subsec:classical-roe-alg}
Classically, Roe algebras are only defined for locally compact metric spaces. Here we shall give a very brief overview of the facts we will need and extend in the sequel. All the upcoming notions are taken from \cite{willett_higher_2020}*{Part~2}, to where we refer the reader for a comprehensive treatise.

Given a locally compact metric space $(X,d)$ a \emph{geometric module for $X$} (or \emph{$X$-module} for short) is a non-degenerate \Star{}representation $\rho\colon C_0(X)\to\CB(\CH_X)$ for some separable Hilbert space $\CH_X$.
For the purposes of this text, it is crucial that such a $\rho$ can always be extended to a \Star{}representation $\bar\rho\colon{\rm{Borel}}(X)\to \CB(\CH_X)$ of the algebra of Borel bounded functions. This follows from \cite{willett_higher_2020}*{Proposition~1.6.11}, or simply from the fact that any \Star{}representation of a \cstar{}algebra uniquely extends to a weakly continuous \Star{}representation of its double dual.

Given an operator $t\colon\CH_X\to\CH_Y$ between geometric modules on locally compact metric spaces $X$ and $Y$, we define its \emph{support} to be the set
\[
  \supp(f) \coloneqq 
  \left\{\left(y,x\right) \, \mid \, \bar\rho(\chf{V})t\bar\rho(\chf{U})\neq 0 \;\, \forall \, y \in V, x \in U \; \text{open}\right\} \subseteq Y \times X,
\]
where $\chf{V}$ and $\chf{U}$ denote the indicator functions.
An operator $t\in\CB(\CH_X)$ has \emph{finite propagation} if its support $\supp(t)\subseteq X \times X$ lies within finite distance from the diagonal $\Delta_X$. Likewise, $t$ is \emph{locally compact} if $\rho(f)t$ and $t\rho(f)$ are compact for every compactly supported $f\in C_0(X)$ or, equivalently, $\bar\rho(\chf{K})t$ and $t\bar\rho(\chf{K})$ are compact for every compact $K\subseteq X$.

An $X$-module is \emph{ample} if $\rho(f)$ is not compact for any $f\in C_0(X)\smallsetminus\{0\}$.\footnote{\,
Observe that an operator algebraist would usually call an ample representation \emph{essential}, as it does not intersect the compact operators. The naming \emph{ample} is coherent with \cite{willett_higher_2020}*{Definition~4.1.1}, and is motivated by its origin in the context of higher index theory.}
Every locally compact metric space admits an ample $X$-module, as it can be seen by, \emph{e.g.}\ considering $\ell^2(Z;\CH)$, where $Z \subseteq X$ is countable and dense and $\CH$ is an infinite-dimensional separable Hilbert space.
Given an ample $X$-module $\CH_X$, the \emph{Roe algebra} $\roecstar X\subseteq\CB(\CH_X)$ of $X$ (classically denoted by $\classicalroecstar{X}$) is the norm-closure of the set of locally compact operators of finite propagation. Note that these operators form a \Star{}algebra, and hence their closure is a \cstar{}algebra.

Lastly, given a ``controlled, proper'' map $f\colon X\to Y$ (see below for definitions) and ample modules $\CH_X$, $\CH_Y$, it is always possible to construct an isometry $t\colon\CH_X\to\CH_Y$ that ``covers f'' (see \cite{willett_higher_2020}*{Proposition~4.3.4}). Roughly speaking, this means that the support of $t$ is within finite distance from the graph of $f$ (see \cite{willett_higher_2020}*{Definition~4.3.3}).
The conjugation map $\Ad(t)\colon\CB(\CH_X)\to \CB(\CH_Y)$ then restricts to a homomorphism between the Roe algebras. Applying this fact to the identity on $X$ one sees that the Roe algebra is well-defined up to (non-canonical) \Star{}isomorphism, \emph{i.e.} the construction does not depend on the choice of ample $X$-module (see \cite{willett_higher_2020}*{Remark~5.1.13}).

\section{Coarse geometric preliminaries}\label{subsec:coarse-spaces}
The ideas here presented are by now rather standard \cites{roe_lectures_2003,protasov2003ball,dydak_alternative_2008,bunke2020homotopy,coarse_groups}. However, the applications we have in mind (particularly~\cite{rigid}) are best explained using non-standard notions and conventions, which we now introduce.

\subsection{Coarse spaces and subspaces}
We mostly use the notation and conventions used in \cite{coarse_groups}. We refer to it for more detailed explanations.
\begin{definition} \label{def:coarse-space}
 A \emph{coarse structure on a set $X$} is an ideal of subsets $\CE \subseteq \CP(X\times X)$---that is, $\CE$ is a non-empty family closed under taking subsets and finite unions---such that
 \begin{itemize}
   \item $\Delta_X\in\CE$;
   \item $E\in\CE\implies \op{E}\in \CE$;
   \item $E,F\in\CE\implies E\circ F\in \CE$.
 \end{itemize}
 A \emph{coarse space} $\crse X = (X,\CE)$ is a set $X$ equipped with a coarse structure $\CE$.
\end{definition}
\begin{remark}
  The elements $E \in \CE$ of a coarse structure are often called \emph{(controlled) entourages}. Observe, moreover, that $\CE$ is a directed set with respect to inclusion. In particular, it makes sense to talk about \emph{cofinal} families, that is, $(E_i)_{i \in I} \subseteq \CE$ such that for all $E \in \CE$ there is some $i \in I$ such that $E \subseteq E_i$.
\end{remark}
\begin{example} \label{ex:metric-as-coarse-1}
  The prototypical example of coarse space is an extended metric space $(X,d)$ (\emph{i.e.}\ where the distance function $d$ is allowed to take the value $+\infty$). Its canonical coarse structure is given by
  \[
   \CE_d\coloneqq\braces{E\subseteq X \, \mid \, \exists \, r<\infty \;\; \text{such that} \;\; d(x,y)<r\ \forall xEy}.
  \]
\end{example}

Let $\crse X=(X,\CE)$ be a coarse space. A subset $A\subseteq X$ is \emph{bounded} (or \emph{controlled}) if $A\times A\in \CE$. Likewise, a family $(A_i)_{i\in I}$ of subsets of $X$ is \emph{$E$\=/controlled} if $A_i\times A_i\subseteq E$ for every $i\in I$. We say $(A_i)_{i \in I}$ is \emph{controlled} if it is $E$\=/controlled for some $E\in \CE$. Roughly speaking, such a family is ``uniformly bounded''.

The following definition slightly differs from certain classical literature (\emph{e.g.}\ \cites{willett_higher_2020}) around Roe algebras and coarse spaces (see \cref{rem:connected-vs-us} as well).
\begin{definition}\label{def: coarsely connected}
  A coarse space $\crse X$ is \emph{(coarsely) connected} if every finite subset of $X$ is bounded. Likewise, two subsets $A,B\subseteq X$ are \emph{coarsely disjoint} if there is no bounded set intersecting both of them.
\end{definition}

Every coarse space $\crse X$ can be decomposed as the disjoint union of its \emph{coarsely connected components}. That is, there is a unique partition $X=\bigsqcup_{i\in I}X_i$ such that the $X_i$ are pairwise coarsely disjoint and the restriction of $\CE$ to each $X_i$ makes it into a connected coarse space. We shall denote this decomposition as $\crse X =\bigsqcup_{i\in I}\crse X_i$.
\begin{warning}
  We are, however, \emph{not} claiming that for a given family of coarse spaces $\crse X_i=(X_i,\CE_i)$ there is a uniquely defined coarse disjoint union $\bigsqcup_{i\in I}\crse X_i$. Namely, if the set $I$ is infinite then there can be multiple coarse structures on the set $\bigsqcup_{i\in I} X_i$ whose restriction to $X_i$ coincides with $\CE_i$ and so that the sets $X_i$ are pairwise coarsely disjoint. There is one canonical minimal choice---called the \emph{disconnected union} in \cite{coarse_groups}. However, it is not the case that if $\crse X =\bigsqcup_{i\in I}\crse X_i$ is the decomposition of a coarse space in its connected components then $\bigsqcup_{i\in I}\crse X_i$ represents the disconnected union of the $\crse X_i$.
\end{warning}

\begin{remark} \label{rem:connected-vs-us}
  Classically, the \emph{coarse disjoint union} of a sequence $(X_n)_{n \in \NN}$ of metric spaces of finite diameter as a metric space obtained by declaring that $d(X_n,X_m) \gg \diam(X_n)+\diam(X_m)$. This is \emph{not} a coarse disjoint union as we mean it in this text.
\end{remark}

\begin{example} \label{ex:metric-as-coarse-2}
  Recall \cref{ex:metric-as-coarse-1}. Clearly, every metric coarse space $(X, d)$ is coarsely connected. On the other hand, if $d$ is an \emph{extended} metric then the decomposition in connected components of $(X,\CE_d)$ coincides with the partition of $X$ into sets $X_i$ of points at finite distance from one another. In such case, $d(X_i,X_j)=+\infty$ for every $i\neq j$.
\end{example}

The coarse structure \emph{generated by a family $(E_i)_{i\in I}$ of relations on $X$} is the smallest coarse structure containing them all, which we denote by $\angles{E_i\mid i\in I}$.
For a given cardinal $\kappa$, a coarse structure is \emph{$\kappa$\=/generated} if there exists a family of at most $\kappa$ relations that generate it. The following simple criterion is classical (see, \emph{e.g.}\ \cite{roe_lectures_2003}*{Theorem 2.55}, \cite{protasov2003ball}*{Theorem 9.1} or \cite{coarse_groups}*{Lemma 8.2.1}).
\begin{proposition}\label{prop:metrizable coarse structure iff ctably gen}
  Let $\CE$ be a coarse structure. The following are equivalent:
  \begin{enumerate} [label=(\roman*)]
   \item $\CE=\CE_d$ for some extended metric $d$;
   \item $\CE$ is countably generated;
   \item $\CE$ contains a countable cofinal sequence of entourages.
  \end{enumerate}
  When the above hold, $\CE$ is connected if and only if $d$ takes only finite values.
\end{proposition}

The following will be useful when discussing coarse maps (see \cref{subsec:coarse-maps,lem:coarse composition is well-defined}, for instance).
\begin{definition} \label{def:coarse-prod}
  Let $\crse X=(X,\CE)$ and $\crse Y=(Y,\CF)$ be coarse spaces. The \emph{product $\crse{Y\times X}$} is the coarse space $(Y\times X, \CF\otimes\CE)$, where
  \[
   \CF\otimes\CE\coloneqq \braces{D \, \mid \, \exists \, F\in \CF \; \text{and} \; E\in \CE \;\; \text{such that} \; D\subseteq F\otimes E}.
  \]
  Equivalently, $\CF\otimes\CE \coloneqq \angles{F\otimes E\mid F\in \CF,\ E\in \CE}$.
\end{definition}

We need a few extra definitions to introduce the notion of \emph{coarse subspace} (see \cref{def:coarse-subspace}), which plays a key role this work. More details and motivation can be found in \cite{coarse_groups}*{Section 3.3}.
\begin{definition} \label{def:thickenings}
  Let $\crse X= (X,\CE)$ be a coarse space and $A \subseteq X$.
  \begin{enumerate}[label=(\roman*)]
    \item An \emph{$E$-controlled thickening of $A$} is any $B \subseteq X$ such that $A \subseteq B \subseteq E(A)$.
    \item A \emph{controlled thickening of $A$} is any $E$-controlled thickening of $A$ for some $E\in \CE$.
  \end{enumerate}
\end{definition}

\begin{definition} \label{def:coarse-containments}
  Let $\crse X = (X, \CE)$ be a coarse space and $A, B \subseteq X$.
  \begin{enumerate}[label=(\roman*)]
    \item We write $A\csub B$ if $A$ is contained in a controlled thickening of $B$. In such case we say $A$ is \emph{coarsely contained} in $B$.
    \item If $A\csub B$ and $B\csub A$ then $A$ and $B$ are \emph{asymptotic} (denoted $A \asymp B$).
  \end{enumerate}
\end{definition}

The relation $\csub$ in \cref{def:coarse-containments} is easily shown to be a partial order on the set of subsets of $X$. Thus, $\asymp$ is an equivalence relation, which allows us to define the following.
\begin{definition} \label{def:coarse-subspace}
  A \emph{coarse subspace $\crse Y$ of $\crse X$} (denoted $\crse{Y\subseteq X}$) is a $\asymp$\=/equivalence class $[Y]$ for some $Y \subseteq X$. If $\crse Y=[Y]$ and $\crse Z=[Z]$ are coarse subspaces of $\crse X$, we say that $\crse Y$ is \emph{coarsely contained} in $\crse Z$ (denoted $\crse{Y\subseteq Z}$) if $Y\csub Z$.
A subset $Y\subseteq X$ is \emph{coarsely dense} if $X\csub Y$ (or, equivalently, $\crse Y=\crse X$).
\end{definition}

\begin{remark}
  Observe that the notation $\crse Y$ in \cref{def:coarse-subspace} is justified. Indeed, $\crse Y=(Y,\CE|_Y)$ is a coarse space and, up to canonical coarse equivalence, it is independent from the choice of $\asymp$-representative $Y$.
\end{remark}

Lastly, we define the intersection of coarse subspaces in the following way.
\begin{definition} \label{def:coarse-intersection}
  Given coarse spaces $\crse Y_{1},\crse Y_{2}, \crse{Z\subseteq X}$, we say that $\crse Z$ is the \emph{coarse intersection of $\crse Y_1 $ and $\crse Y_2$} if $\crse Z$ is coarsely contained in both $\crse Y_1$ and $\crse Y_2$ and is largest with this property, \emph{i.e.}\ if $\crse{Z}'\crse\subseteq \crse Y_1$ and $\crse{ Z}'\crse \subseteq \crse Y_2$ then $\crse Z'\crse \subseteq \crse Z$.
\end{definition}

\begin{remark} \label{rem:coarse-intersection}
  The coarse intersection, as defined in \cref{def:coarse-intersection}, needs not always exist (see \cite{coarse_groups}*{Example~3.4.7}). Nevertheless, if it does exist then it is unique.
  In such case we denote it by $\crse Y_1\crse \cap \crse Y_2$. Observe that if $\crse{Y\subseteq Z}$ then $\crse{Y\cap Z=Y}$ (in particular the intersection then exists).
\end{remark}

\subsection{Coarse maps} \label{subsec:coarse-maps}
It will be convenient to define coarse maps in terms of relations (as opposed to functions, as is done in~\cite{coarse_groups}). This difference is purely formal and does not affect the results in~\cite{coarse_groups}. In fact, see \cref{cor:coarse composition is well-defined,lemma:coarse-relation-yields-coarse-map} below for the interplay between these notions.
\begin{notation}
  We denote by $\pi_X \coloneqq Y \times X \to X$ and $\pi_Y \coloneqq Y \times X \to Y$ the usual projections onto the $X$ and $Y$-coordinates respectively.
\end{notation}
\begin{definition}\label{def:controlled_function}
  Let $\crse X=(X,\CE)$ and $\crse Y=(Y,\CF)$ be coarse spaces. A relation $R$ from $X$ to $Y$ is \emph{controlled} if
  \[
    \left(R\otimes R\right)\left(\CE\right) \subseteq \CF.
  \]
  Moreover, $R$ is \emph{coarsely everywhere defined} if $\pi_X(R)$ is coarsely dense in $X$.
\end{definition}

\begin{remark}
  Note that, by \cref{eq:product relation image}, the condition $(R\otimes R)(\CE) \subseteq \CF$ in \cref{def:controlled_function}
  may be rephrased as $R\circ E\circ \op{R}\in \CF$ for all $E \in \CE$.
\end{remark}

\begin{remark} \label{rem:map-controlled-graph-controlled}
 If $f\colon X\to Y$ is any map, then the relation $\grop{f}$ is controlled exactly when $f$ is \emph{controlled} (see~\cite{coarse_groups}, or \emph{bornologous} in~\cite{roe_lectures_2003}).
\end{remark}

\begin{remark}
  Controlled relations should be thought of as controlled partial functions that are only coarsely well-defined (see \cref{cor:coarse composition is well-defined}). One may more generally define \emph{coarsely well-defined} partial functions as relations $R$ with $(R\otimes R)(\Delta_X) \in \CF$. However, we will not need such a notion here.
\end{remark}

For the following, recall \cref{def:coarse-containments,def:coarse-subspace}.
\begin{definition}\label{def:close_relations}
  Two controlled relations $R,R'$ from $\crse X$ to $\crse Y$ are \emph{close} if $R\asymp R'$ in $\crse{Y\times X}$, that is, if they define the same coarse subspace.
\end{definition}

By the definition of $\CF\otimes\CE$, it follows that a controlled thickening of $R$ in $\crse{Y\times X}$ is always contained in a thickening of the form $(F\otimes E)(R)$ for some $F \in \CF$ and $E \in \CE$ which, by \cref{eq:product relation image}, can be written as $F\circ R\circ \op{E}$. It follows that
\begin{equation}\label{eq:coarsely contained relation}
  R'\csub R \;\; \Longleftrightarrow \;\;  R'\subseteq F\circ R\circ E \;\; \text{for some} \,\; F \in \CF \,\, \text{and} \,\, E\in \CE.
\end{equation}

The following result implies that \cref{def:close_relations} is an extension of the standard definition of closeness for controlled maps.
\begin{lemma}\label{lem:close_relations_equivalent}
  Let $R,R'$ be controlled relations from $\crse X$ to $\crse Y$.
  \begin{enumerate}[label=(\roman*)]
   \item If they are close then $(R'\otimes R)(\Delta_X)\in \CF$ (equivalently, $R'\circ \op R\in \CF$).
   \item If they are coarsely everywhere defined, the converse is also true.
  \end{enumerate}
\end{lemma}
\begin{proof}
  Suppose $R'$ and $R$ are close. Without loss of generality, we may assume that $R \subseteq R'\subseteq F\circ R\circ E$ for some $F\in\CF$ and $E\in \CE$. By \cref{eq:product relation composition,eq:product relation image} we have:
  \begin{align*}
    (R'\otimes R)(\Delta_X)
    &\subseteq \paren{(F\circ R\circ E)\otimes R}(\Delta_X) \\
    &= (F\circ R\circ E)\circ \Delta_X\circ\op R \\
    &= F\circ R\circ E\circ \op R.
  \end{align*}
  Since $R$ is controlled, $R\circ E\circ \op R\in\CF$, proving the first claim.

  For the converse, let us now assume that $R'$ and $R$ are defined coarsely everywhere and $R'\circ \op R\in \CF$. For any $E\in \CE$ the following relation is (contained in) a controlled thickening of $R$:
  \begin{equation}\label{eq:thickening_of_R}
   \paren{(R'\circ \op R)\otimes E}(R)= R'\circ \op R\circ R\circ \op E.
  \end{equation}
  Observe that $\op R\circ R$ contains the diagonal $\Delta_{\pi_X(R)}$. Since $\pi_X(R)$ is coarsely dense, we may then choose $E \in \CE$ large enough so that $\Delta_X\subseteq \Delta_{\pi_X(R)}\circ \op E$.
  For such an $E$, \eqref{eq:thickening_of_R} proves that the relation $R'$ is contained in a controlled thickening of $R$. A symmetric argument concludes the proof.
\end{proof}

\begin{remark} \label{rem:close-preserves-stuff}
  Arguing as in \cref{lem:close_relations_equivalent}, it can be shown that if $R$ is a controlled relation
  and $R'\asymp R$, then $R'$ is also controlled. The property of being coarsely everywhere defined is preserved under closeness as well.
\end{remark}

The following is well posed because of the above remark.
\begin{definition} \label{def:coarse-map}
  Let $\crse X$ and $\crse Y$ be coarse spaces.
  \begin{enumerate}[label=(\roman*)]
    \item A \emph{partial coarse map $\crse R$ from $\crse X$ to $\crse Y$} is the $\asymp$-equivalence class of a controlled relation $R \subseteq Y \times X$.
    \item A \emph{coarse map $\crse R$ from $\crse X$ to $\crse Y$} is a partial coarse map that is coarsely everywhere defined.
  \end{enumerate}
\end{definition}

\begin{example} \label{ex:canonical-proj-coarse-maps}
  Observe that the canonical projections $\pi_{X}$, $\pi_{Y}$ from $\crse{Y\times X}$ to $\crse X$ and $\crse Y$ are controlled. Thus, they define coarse maps $\crse{\pi_{X}}$ and $\crse{\pi_{Y}}$. Moreover, $\crse{\pi_{X}(A)\subseteq X}$ and $\crse{\pi_{Y}(A)\subseteq Y}$ are well-defined coarse subspaces for any coarse subspace $\crse{A\subseteq Y\times X}$.
  In fact, if $A'\subseteq F\circ A\circ E$ and $x'\in \pi_X(A')$ then there are $y, y' \in Y$ and $x \in X$ with $y'FyAxEx'$. It follows that $\pi_X(A')\subseteq \op{E}(\pi_X(A))$, and, analogously, that $\pi_Y(A')\subseteq F(\pi_Y(A))$.
\end{example}

\begin{definition}
  Let $\crse{R\subseteq Y\times X}$ be a partial coarse map from $\crse X$ to $\crse Y$. The \emph{coarse domain} and the \emph{coarse image} of $\crse R$ are the subspaces
  \[
    \cdom(\crse R) \coloneqq \crse{\pi_{X}\left(R\right)}
    \quad\text{and}\quad
    \cim(\crse R) \coloneqq \crse{\pi_{Y}\left(R\right)}.
  \]
\end{definition}
Observe that $\cdom(\crse R)$ and $ \cim(\crse R)$ are well defined (\emph{i.e.} they do not depend on the choice of representative $R \subseteq Y \times X$). Note, as well, that $\crse R$ is coarsely everywhere defined (and hence a coarse map) if and only if $\cdom(\crse R)=\crse X$.
Analogously, we say that $\crse R$ is \emph{coarsely surjective} if $\cim(\crse R)=\crse Y$.

Taking compositions of partial coarse maps requires a little more thought. 
\begin{definition}\label{def: coarse composition}
  Given coarse subspaces $\crse{R\subseteq Z\times Y}$ and $\crse{S\subseteq Y\times X}$, we say that $\crse{T\subseteq Z\times X}$ is the \emph{coarse composition} of $\crse R$ and $\crse S$ (denoted $\crse{T =R\circ S}$) if it is smallest with the property that for every choice of relations $R\csub\crse R$ and $S\csub\crse S$ the composition $R\circ S$ is coarsely contained in $\crse T$.
\end{definition}

As for coarse intersections, the coarse composition need not exist in general (not even if $\crse R$ and $\crse S$ are assumed to be partial coarse maps). However, we have the following.
\begin{lemma}\label{lem:coarse composition is well-defined}
  Let $\crse X =(X,\CE)$, $\crse Y=(Y,\CF)$, $\crse Z=(Z,\CD)$ be coarse spaces, $\crse{S\subseteq Y\times X}$ a coarse subspace and $\crse R$ a partial coarse map from $\crse Y$ to $\crse Z$. 
  If $\crse{\pi_{Y}(S)\subseteq}\cdom(\crse R)$, then the coarse composition $\crse{R\circ S}$ exists.

  Specifically, if $R$ and $S$ are representatives such that $\pi_{Y}(S)\subseteq\pi_Y(R)$, then $\crse{R\circ S}\coloneqq [R\circ S]$.
\end{lemma}
\begin{proof}
  Observe that minimality is clear: since $R\csub\crse R$ and $S\csub\crse S$ any candidate coarse intersection must contain $[R\circ S]$. Therefore, we only need to show that $[R\circ S]$ fulfills the other requirement in the definition of coarse intersection.

  Given $R'\csub\crse R$ and $S'\csub \crse S$, there are $D\in\CD$, $F_1,F_2\in\CF$ and $E\in\CE$ such that $R'\subseteq D\circ R\circ F_2$ and  $S'\subseteq F_1\circ S\circ E$.
  Then
  \begin{equation}\label{eq:coarse image is well defined}
    R'\circ S' \subseteq D\circ R\circ F_2\circ F_1\circ S\circ E.
  \end{equation}
  Since $\pi_{Y}(S)\subseteq\pi_Y(R)$, we have $S=\Delta_{\pi_Y(R)}\circ S\subseteq \op{R}\circ R\circ S$. Substituting into \cref{eq:coarse image is well defined}, we obtain:
  \[
    R'\circ S' \subseteq D \circ \paren{R\circ F_2 \circ F_1 \circ\op{R}}\circ R\circ S\circ E = D'\circ R\circ S\circ E
  \]
  for some other $D'\in\CD$, where for the last equality we used that $R$ is a controlled relation.  
\end{proof}

Observe that coarse subspaces of $\crse X$ are the same as coarse subspaces of $\crse{X}\times \{\rm pt\}$, and that $R\circ (A\times\{{\rm pt}\} ) = R(A)\times \{{\rm pt}\}$ for every relation $R\subseteq Y\times X$ and subset $A\subseteq X$. We may thus apply \cref{lem:coarse composition is well-defined} to define coarse images, yielding the following.
\begin{corollary}\label{cor:coarse composition is well-defined}
  Let $\crse R$ be a partial coarse map from $\crse X$ to $\crse Y$, and $\crse {A\subseteq} \cdom\paren{\crse R}$ be a coarse subspace. Then the coarse image $\crse{R(A)\subseteq Y}$ is a well-defined coarse subspace of $\crse Y$. In particular, if $\crse R$ is a coarse map then the coarse image of every coarse subspace of $\crse X$ is well-defined.
\end{corollary}

\begin{remark}
  \begin{enumerate}[label=(\roman*)]
    \item One problem with the coarse composition of partial coarse maps $\crse{R\subseteq Z\times Y}$ and $\crse {S\subseteq Y\times Z}$ is that the coarse intersection of $\cdom(\crse R)$ and $\cim(\crse S)$ needs not exist. If this intersection does exist, and so does its preimage under $\crse S$---call it $\crse Z'$---then the coarse composition $\crse{R\circ S}$ is equal to the coarse composition of the restriction $\crse{R\circ (S|_{Z'})}$, which exists by \cref{lem:coarse composition is well-defined}.
    \item Using the notation of \cref{cor:coarse composition is well-defined}, the coarse image of a partial coarse map can be written as $\cim(\crse R)= \crse R(\cdom(\crse R))$. We may also write $\cim(\crse R)=\crse{R(X)}$: this is justified from the point above, because $\cdom(\crse R)\crse{\cap X}$ exists and it is simply equal to $\cdom(\crse R)$.
  \end{enumerate}
\end{remark}

From now on we will stay true to our convention that a bold symbol denotes a coarsely defined object and the relevant non-bold symbol denotes a choice of a representative for it (cf.\ \cref{general notation}).
The following simple observation is notationally useful and it makes explicit the ongoing idea that controlled relations from $\crse X$ to $\crse Y$ are in correspondence with controlled partial functions.
\begin{lemma} \label{lemma:coarse-relation-yields-coarse-map}
  Let $\crse R$ be a partial coarse map from $\crse X$ to $\crse Y$. There is a partial controlled function $f\colon \dom(f)\to Y$ such that $\grop{f}\asymp R$. Moreover, if $\crse R$ is coarsely everywhere defined, then $f$ may be chosen so that $\dom(f) = X$.
\end{lemma}
\begin{proof}
  Fix a representative $R$. For every $x\in \pi_X(R)=\dom(R)$ arbitrarily choose $f(x) \in R(x)$. This defines a function $f\colon \dom(R)\to Y$ with $\grop{f}\subseteq R$. Observe that
  \[
    R = R\circ \Delta_{\dom(R)}\subseteq R\circ \op R\circ \grop{f},
  \]
  which is a controlled thickening of $\grop{f}$, because $R\circ \op R\in \CF$.

  Suppose now that $R$ is coarsely everywhere defined. Up to replacing $R$ with a thickening $R\circ E$, we may then assume that $\pi_X(R)=X$, so that a partial function $f$ constructed as above is indeed a function $f\colon X\to Y$ defined everywhere.
\end{proof}

The previous results and discussions allow us to use the following notation.
\begin{notation}
  We denote partial coarse maps from $\crse X$ to $\crse Y$ as functions
  \[
    \crse f\colon \cdom(\crse f)\to\crse Y.
  \]
\end{notation}
The identity function $\id_X$ defines the identity coarse map $\cid_{\crse X}\colon\crse X\to\crse X$ (in terms of relations, $\cid_{\crse X}=[\Delta_X]$).
The following generalizes the usual notion of coarse equivalence between metric spaces (see, \emph{e.g.} \cites{nowak-yu-book}) to the setting of coarse spaces.
\begin{definition}
  The coarse spaces $\crse X$ and $\crse Y$ are \emph{coarsely equivalent} if there are coarse maps $\crse{ f\colon X\to Y}$ and $ \crse{g\colon Y\to X}$ such that $\crse{f\circ g}=\cid_{\crse Y}$ and $\crse{g\circ f}=\cid_{\crse X}$. Such coarse maps are \emph{coarse equivalences} and are said to be \emph{coarse inverse} to one another. The coarse inverse $\crse g$ is also denoted by $\crse f^{-1}$.
\end{definition}

\begin{remark}
  Note that $\crse f$ is a coarse equivalence if and only if $\op{\crse f}$ is its coarse inverse.
\end{remark}

It is routine to verify that the following definition is well-posed, that is, it does not depend on the choice of representative.
\begin{definition}\label{def:proper}
  A coarse subset $\crse{R\subseteq Y\times X}$ is \emph{proper} if $\op R(B)$ is bounded for every bounded set $B\subseteq Y$.
  In particular, a (partial) coarse map $\crse{f\colon X\to Y}$ is \emph{proper} if it is proper as a coarse subspace of $\crse{Y\times X}$.
\end{definition}

\begin{remark}
  For functions, \cref{def:proper} just means that the preimage of bounded sets is bounded.
  Given $\crse{f\colon X\to Y}$ and $B\subseteq Y$, the usual $f^{-1}(B)$ denotes the preimage under the representative $f$. If the chosen representative is a relation $R$, then $R^{-1}(B)$ is simply $\op R(B)$, which is the reason this appears in \cref{def:proper}.
\end{remark}

\begin{remark} \label{rem:proper-preimage-may-not-exist}
  Note that, similarly to \cref{rem:coarse-intersection}, the coarse preimage of a coarse subspace may not exist in general (see \cite{coarse_groups}*{Definition~3.4.8}). Properness for (partial) coarse maps $\crse{f\colon X\to Y}$ in \cref{def:proper} may be rephrased as saying that for every coarse point $\crse{y\in Y}$ the coarse preimage $\crse{f^{-1}(y)\subseteq X}$ does exist and it is either empty or a coarse point $\crse{x\in X}$.
\end{remark}

The following is a natural strengthening of properness.
\begin{definition} \label{def:coarse embedding}
  A partial coarse map $\crse{f\colon X\to Y}$ is a \emph{partial coarse embedding} if for every $F\in \CF$ there is $E\in \CE$ such that $f^{-1}(B)$ is $E$-bounded whenever $B\subseteq Y$ is $F$-bounded. If $\crse{f}$ is moreover coarsely everywhere defined, then we say it is a \emph{coarse embedding}.
\end{definition}

\begin{remark}\label{rmk: alternative def coarse embedding}
  Equivalently, we may also say that a controlled relation $\crse R$ is a partial coarse embedding in the sense of \cref{def:coarse embedding} if ${(\op R\otimes\op R)}(\CF) \subseteq \CE$. This approach is more akin to the definition of controlledness we used above (cf.\ \cref{def:controlled_function}).
\end{remark}

By \cref{ex:canonical-proj-coarse-maps}, the projection $\crse{\pi_{X}\colon R\to X}$ is a coarse map for every coarse subspace $\crse{R\subseteq  Y\times X}$.
Unravelling the definitions yields the following.
\begin{lemma}\label{lem:coarse map iff projection is embedding}
   A coarse subspace $\crse{R\subseteq Y\times X}$ is a partial coarse map if and only if the restriction of $\crse{\pi_{X}}$ to $\crse R$ is a coarse embedding.
\end{lemma}
\begin{proof}
  Fix a representative $R$ of $\crse R$.
  Suppose $\crse{\pi_{X}}$ restricts to a coarse embedding. By definition, this means that for every $E\in\CE$ there is an $F'\otimes E'\in\CF\otimes\CE$ such that if $(y_1,x_1)$ and $(y_2,x_2)$ are in $R$ and $x_1Ex_2$, then $(y_1,x_1)(F'\otimes E')(y_2,x_2)$. This implies that $y_1F'y_2$ or, in other words, that $R\otimes R(E)\subseteq F'$. This shows that $\crse R$ is a coarse map.

  Likewise, if $\crse R$ is a coarse map and $E\in\CE$ then $F \coloneqq (R\otimes R)(E)\in\CF$. But then if we are given $x_1Ex_2$ and  $(y_1,x_1),(y_2,x_2)\in R$, we see that $(y_1,x_1)(F\otimes E)(y_2,x_2)$. This shows that $\pi_X|_R$ is a coarse embedding.
\end{proof}

The use of the word \emph{embedding} in \cref{def:coarse embedding,lem:coarse map iff projection is embedding} is justified by observing that if $\crse{f\colon X\to Y}$ is a partial coarse embedding, then the restriction $\crse f\colon \cdom(\crse f)\to\crse{f(X)}$ is a coarse equivalence. The following are immediate consequences.
\begin{corollary} \label{cor:pcoarse-emb-implies-coarse-containment}
  If $\crse{f\colon X\to Y}$ is a partial coarse embedding and $\crse{Z, Z'\subseteq}\cdom(\crse X)$ are coarse subspaces such that $\crse{f(Z)\subseteq f(Z')}$, then $\crse{Z\subseteq Z'}$.
\end{corollary}

\begin{corollary} \label{cor:restriction with same domain coincides with f}
  Let $\crse{f\colon X\to Y}$ be a partial coarse map. Suppose $\crse{f'\subseteq f}$ is a coarse subspace such that $\cdom(\crse f)\crse\subseteq\cdom(\crse{f'})$. Then $\crse{f=f'}$.
\end{corollary}

\cref{cor:restriction with same domain coincides with f} will be useful in the sequel, as it is notationally more convenient to verify $\cdom(\crse f)\crse\subseteq\cdom(\crse{f'})$ rather than $\crse{f\subseteq f'}$.

\begin{remark}
  Observe that \cref{cor:pcoarse-emb-implies-coarse-containment} fails without the coarse embedding assumption.
\end{remark}

Observe that \cref{rmk: alternative def coarse embedding} can be rephrased as saying that a partial coarse map $\crse{f\colon X\to Y}$ is a partial coarse embedding if and only if the symmetric coarse subspace $\op{\crse f}\crse{\subseteq X\times Y}$ is a partial coarse map.
Following this train of thoughts further yields the following.
\begin{proposition}\label{prop:transpose is coarse inverse}
  Let $\crse f$ be a partial coarse embedding. Then $\op{\crse f}$ is a partial coarse embedding as well, and the compositions $\op{\crse{f}}\crse \circ \crse{f}$ and $\crse{f}\crse\circ\op{\crse{f}}$ are well-defined and are contained in $\cid_{\crse X}$ and $\cid_{\crse Y}$ respectively. 
  Moreover, $\crse{f}$ is coarsely everywhere defined if and only if $\op{\crse{f}}$ is coarsely surjective (and vice versa).
  Lastly, the following are equivalent:
  \begin{enumerate}[label=(\roman*)]
    \item\label{item:prop:transpose is coarse inverse-ce} $\crse{f}$ is a coarse equivalence;
    \item\label{item:prop:transpose is coarse inverse-*ce} $\op{\crse{f}}$ is a coarse equivalence;
    \item\label{item:prop:transpose is coarse inverse-cinv} $\crse{f}$ and $\op{\crse{f}}$ are coarse inverses of one another;
    \item\label{item:prop:transpose is coarse inverse-csur} $\crse{f}$ and $\op{\crse{f}}$ are coarsely surjective;
    \item\label{item:prop:transpose is coarse inverse-cdef} $\crse{f}$ and $\op{\crse{f}}$ are coarsely everywhere defined.
  \end{enumerate}  
\end{proposition}
\begin{proof}
  It is clear that $\op{\crse f}$ is a partial coarse embedding, because $\op{(\op{\crse{f}})}=\crse f$.
  The existence of the coarse compositions follows from \cref{lem:coarse composition is well-defined} upon noting that
  \begin{equation}\label{eq:transpose is inverse}
    \cdom(\crse{f})=\cim(\op{\crse{f}})
    \;\;\text{and}\;\;
    \cdom(\op{\crse{f}})=\cim(\crse{f}).
  \end{equation}
  It is also clear that these coarse compositions are coarsely contained in the coarse identities, because \cref{eq:product relation image} shows that if $R$ is a controlled relation from $\crse X$ to $\crse Y$ then $[R\circ \op{{R}}]\crse{\subseteq} \cid_{\crse Y}$. The duality between coarse surjectivity and coarse density of the domain follows immediately from \cref{eq:transpose is inverse}.

  All that remains to do is to verify the claimed equivalences. For every relation $R\subseteq Y\times X$ the diagonals $\Delta_{\pi_X(R)}$ and $\Delta_{\pi_Y(R)}$ are contained in $\op R\circ R$ and $R\circ \op R$ respectively. It is then clear that 
  \cref{item:prop:transpose is coarse inverse-cdef}$\Rightarrow$ \cref{item:prop:transpose is coarse inverse-ce,item:prop:transpose is coarse inverse-*ce}.  
  The implication \cref{item:prop:transpose is coarse inverse-csur} $\Leftrightarrow$ \cref{item:prop:transpose is coarse inverse-cdef} is still a consequence of  \cref{eq:transpose is inverse}, and \cref{item:prop:transpose is coarse inverse-ce,item:prop:transpose is coarse inverse-*ce} obviously imply \cref{item:prop:transpose is coarse inverse-csur} and \cref{item:prop:transpose is coarse inverse-cdef}.

  The implication \cref{item:prop:transpose is coarse inverse-cdef}+\cref{item:prop:transpose is coarse inverse-ce}$\Rightarrow$ \cref{item:prop:transpose is coarse inverse-cinv} is essentially the usual proof of the uniqueness of inverse:
  \[
    \op{\crse f} = (\crse{f}^{-1}\crse \circ\crse f) \crse\circ \op{\crse f} \crse
    = \crse{f}^{-1}\crse \circ (\crse f\crse\circ \op{\crse f}) \crse \subseteq  \crse{f}^{-1}
  \]
  paired with the observation that if $\op{\crse f}$ is coarsely everywhere defined and $\op{\crse f} \crse\subseteq  \crse{f}^{-1}$ then equality must hold by \cref{cor:restriction with same domain coincides with f}.
  The converse is trivial.
\end{proof}

\begin{corollary}[cf.\ \cite{coarse_groups}*{Lemma~2.4.9}]
  A coarse map is a coarse equivalence if and only if it is a coarsely surjective coarse embedding.
\end{corollary}

\subsection{Useful coarse geometric properties}
In this subsection we further develop the glossary of useful terms and notions in coarse geometry.
In our naming conventions, a property is ``local'' if it is defined in terms of bounded sets, while it is ``controlled'' if it is defined in terms of entourages (or, equivalently, uniformly bounded sets).
\begin{definition} \label{def:controlled-partition}
 A \emph{controlled partition} of $\crse X = (X, \CE)$ is a partition  $X=\bigsqcup_{i\in I} C_i$, where there is some $E \in \CE$ such that $C_i$ is $E$\=/bounded for all $i \in I$.
\end{definition}

\begin{remark}\label{rmk:controlled partitions give crs eq}
  If $X=\bigsqcup_{i\in I}C_i$ is a controlled partition of a coarse space $\crse X=(X,\CE)$, then we may generate a coarse structure $\CE_I$ on $I$ by declaring
  \[
    \CE_I \coloneqq \angles{\braces{(i,j)\mid C_i\cup C_j \, \text{ is $E$-bounded}} \mid {E \in \CE}} \subseteq I \times I.
  \]
  The map sending a point $x\in X$ to the (uniquely determined) index $i \in I$ such that $x \in C_i$ can then be shown to be a coarse equivalence $\crse X\to(I,\CE_I)$.
  This coarse equivalence is more concisely defined via the binary relation
  \[
    \bigcup\braces{\braces{i}\times C_i\mid i\in I}\subseteq I\times X.
  \]
\end{remark}

\begin{definition} \label{def:locally finite family}
  Let $\{C_i\}_{i\in I}$ be a family where $C_i \subseteq X$. We say
  \begin{enumerate}[label=(\roman*)]
    \item $\{C_i\}_{i\in I}$ is \emph{locally finite} if for every bounded $A \subseteq X$ there are at most finitely many $i \in I$ so that $C_i \cap A \neq \emptyset$;
    \item $\{C_i\}_{i\in I}$ is \emph{uniformly} (or \emph{controlledly}) \emph{locally finite} if for every $E \in \CE$
      \[
       \sup\braces{\#\braces{i\in I\mid C_i \cap A \neq \emptyset} \mid A \subseteq X \text{ is $E$-controlled}} < \infty.
      \]
  \end{enumerate}
\end{definition}

With \cref{def:locally finite family} at hand, the following should be no surprise.
\begin{definition} \label{def:locally finite space - bounded geo}
 Let $\crse X = (X, \CE)$ be a coarse space. We say
\begin{enumerate}[label=(\roman*)]
  \item $\crse X$ is \emph{coarsely locally finite} if it admits a locally finite controlled partition;
  \item $\crse X$ has \emph{bounded (coarse) geometry} if it admits a uniformly locally finite controlled partition.
\end{enumerate}
\end{definition}

\begin{example}
  Recall \cref{ex:metric-as-coarse-1,ex:metric-as-coarse-2}. If $(X,d)$ is a locally compact (extended) metric space, then $(X,\CE_d)$ is coarsely locally finite. There is also a metric version of bounded (coarse) geometry, and it is easily seen that the metric space $(X,d)$ has bounded geometry if and only if the metric coarse space $(X,\CE_d)$ has bounded geometry in the sense of \cref{def:locally finite space - bounded geo}.
\end{example}

\begin{remark} \label{rem:coarse-bdd-vs-ulf}
  A coarse space is \emph{uniformly locally finite} if $\sup_{x \in X} \abs{E(x)} < + \infty$ for all $E \in \CE$. In the literature, the term \emph{bounded geometry} is often used as a synonym for uniformly locally finite~\cites{braga_rigid_unif_roe_2022,bbfvw_2023_embeddings_vna,braga_farah_rig_2021,braga_gelfand_duality_2022}. However, we will keep them as separate concepts. In particular, our bounded geometry coarse spaces need not be discrete.
\end{remark}

\cref{rmk:controlled partitions give crs eq} shows that a coarse space is coarsely locally finite (resp.\ has bounded geometry) if and only if it is coarsely equivalent to a coarse space where all bounded sets are finite (resp.\ uniformly finite).

\begin{remark}
  Using \cref{rmk:controlled partitions give crs eq}, when working with a coarsely locally finite coarse space $\crse X$ one may, in principle, replace it with a coarsely equivalent locally finite space. However, we do not particularly recommend following this procedure, as it may hide important features that may otherwise be apparent. Besides, replacing $\crse X$ with a coarsely equivalent model would not work well with, for instance, uniform Roe algebras.
\end{remark}

When dealing with technical details, particularly with separated sets, it is often useful to consider \emph{symmetric} entourages that contain the diagonal. We thus give them a name.
\begin{notation}
  A \emph{gauge} on $\crse X$ is a symmetric controlled entourage $\gauge\in \CE$ containing the diagonal, that is, $\Delta_X \subseteq \gauge$.
\end{notation}

Coarse geometric properties can often be rephrased using gauges (and it will often be convenient to do so).

\section{Coarse geometric modules} \label{sec:modules}
In this section we start the path towards the construction of the \emph{Roe algebra} of a coarse space. Henceforth, $\crse X$ will be a fixed coarse space, and $\fkA \leq \CP(X)$ will be a unital Boolean algebra of subsets of $X$.

\subsection{Definition and general observations}
As in the classical setting, we will make use of representations of the relevant algebras to give an appropriate definition of modules. We thus start with the following.
\begin{definition}\label{def:non-degenerate}
  A unital representation $\chf{\bullet}\colon\fkA\to\CB(\CH)$ is \emph{non-degenerate} if there is a gauge $\gauge\in\CE$ such that the images $\chf{A}(\CH)$ with $\gauge$-bounded $A\in\fkA$ generate a dense subspace of $\CH$, \emph{i.e.}\
  \[
    \CH = \overline{\angles{\chf{A}(\CH) \mid A\in\fkA ,\ \gauge\text{-bounded}}}^{\norm{\cdot}}.
  \]
  If $\gauge$ is as above, then we call it a \emph{non-degeneracy gauge}.
\end{definition}

Recall that the image of $\fkA$ is a commuting family of projections in $\CB(\CH)$.
\begin{example} \label{ex:degenerate-rep}
  The prototypical example of a \emph{degenerate} representation may be obtained by considering ``natural representations of $X$'' when the coarse structure is too small. For example, if $(X,\mu)$ is a Borel probability space, the standard representation ${\rm{Borel}}(X)\to \CB(L^2(X,\mu))$ is certainly degenerate if all the $\CE$\=/bounded subsets of $X$ have measure $0$.
\end{example}

The following simple observations are very useful.

\begin{lemma}\label{lem:supports of vectors almost contained in bounded}
  Let $\chf{\bullet}$ be non-degenerate. For every $v\in \CH$ and $\varepsilon > 0$, there is an $A\in\fkA$ which is a finite union of disjoint $\gauge$-bounded sets such that
  \[
    \norm{v - \chf{A}(v)}<\varepsilon.
  \]
\end{lemma}
\begin{proof}
  By non-degeneracy, there exist \gauge-bounded $A_1,\ldots,A_n\in \fkA$ and vectors $w_1,\ldots, w_n\in \CH$ such that $\norm{v-\sum_{i=1}^n\chf{A_i}(w_i)}\leq \varepsilon$. We wish to make the $A_i$ disjoint. This is easily done by observing that
  \begin{align*}
    \chf{B}(w)+\chf{B'}(w')&=\chf{B}(w)+\chf{B\cap B'}(w') + \chf{B'\smallsetminus B}(w') \\
    &=\chf{B}(w)+\chf{B}\chf{B'}(w')+\chf{B'\smallsetminus B}(w') \\
    &=\chf{B}\paren{w+\chf{B'}(w')}+\chf{B'\smallsetminus B}(w') = \chf{B}\paren{w''}+\chf{B'\smallsetminus B}(w'),
  \end{align*}
  and so it is always possible to replace $A_i$ with $A_i'\coloneqq A_i\smallsetminus (\bigcup_{j<i}A_j)$ if needed.
  Thus, $A\coloneqq \bigsqcup_{i=1}^n A_i'$ is a disjoint union of \gauge-bounded elements of $\fkA$ so that $w\coloneqq \sum_{i=1}^n\chf{A_i}(w_i)\in \chf{A}(\CH)$. It follows that $\norm{v - \chf{A}(v)}\leq \varepsilon$, as desired.
\end{proof}

\begin{lemma}\label{lem:non-orthogonality is local}
  Let $\chf{\bullet}$ be non-degenerate. If $v,w\in \CH$ are not orthogonal, there is an $\gauge$-bounded $A\in\fkA$ such that $\angles{\chf{A}(v),\chf{A}(w)}\neq 0$.
\end{lemma}
\begin{proof}
  By \cref{lem:supports of vectors almost contained in bounded}, we may find $B=\bigsqcup_{i=1}^n B_i$ and $C=\bigsqcup_{j=1}^m C_j$ disjoint unions of \gauge-controlled elements of $\fkA$ such that $\norm{v-\chf{B}(v)} + \norm{w-\chf{C}(w)}$ is much smaller than $\abs{\scal{v}{w}}$.
  By the triangle inequality, we obtain that
  \[
    0<\abs{\scal{\chf{B}(v)}{\chf{C}(w)}}=\abs{\sum_{i,j}\scal{\chf{B_i}(v)}{\chf{C_j}(w)}}.
  \]
  In particular, there are $i$ and $j$ such that $\scal{\chf{B_i}(v)}{\chf{C_j}(w)}\neq 0$, and $A\coloneqq B_i\cap C_j$ is as desired.
\end{proof}

\cref{lem:supports of vectors almost contained in bounded,lem:non-orthogonality is local} show that the non-degeneracy condition is enough to imply that the representation can---to some extent---be studied locally. In other words, ``$\chf{\bullet}$ knows about the geometry of $\crse X$''. We make this into a definition.

\begin{definition} \label{def:coarse-module}
  A (\emph{finitely additive}) \emph{coarse geometric module for $\crse X$} (or an \emph{$\crse X$-module} for short) is a boolean algebra $\fkA\leq \CP(X)$ together with a Hilbert space $\CHx$ and a non-degenerate unital representation $\chf{\bullet}\colon \fkA \to \CB(\CHx)$.
\end{definition}

We will denote an $\crse X$-module simply by $\CH$. However, we will use $(\fkA,\CH)$ or $(\fkA,\chf{\bullet},\CH)$ if confusion may arise.
\begin{convention}\label{conv:sigma additive module}
  We drop the \emph{finitely additive} label, as these are the only modules we will work with. 
  Nevertheless, it is conceivable that, in some applications where $\fkA \subseteq \CP(X)$ is a $\sigma$-algebra, it may be useful to work with \emph{$\sigma$-additive coarse geometric modules} (which are defined in the natural way). This will however not be the case in this text.
\end{convention}

It is important to note that our definition of coarse geometric module extends the classical definition of geometric modules for metric spaces (recall \cref{subsec:classical-roe-alg}).
\begin{example}
  Let $(X,d)$ be a locally compact metric space and $\rho\colon C_0(X)\to\CB(\CH)$ an geometric module as explained in \cref{subsec:classical-roe-alg}, and let $\bar\rho$ be the extension to a unital representation of the algebra of bounded Borel functions ${\rm Borel}(X)$.
  For every Borel subset $A\subseteq X$ we may thus define $\chf A$ by applying $\bar\rho$ to the indicator function of $A$ (whence our choice of notation). The resulting map $\chf\bullet$ is a unital representation of the Borel $\sigma$\=/algebra, which is easily seen to be non-degenerate in the sense of \cref{def:non-degenerate}. This shows that $\CH$ naturally defines a ($\sigma$-additive) coarse geometric module for $\crse X=(X,\CE_d)$.
\end{example}

\begin{notation}
  Let $(\fkA,\chf{\bullet},\CH)$ be an $\crse X$-module. With a light abuse of notation, we call the elements of $\fkA$ \emph{measurable subsets} of $X$. For every $A\in\fkA$, we denote by $\CH_A$ the Hilbert space $\chf{A}(\CH) \subseteq \CH$.
\end{notation}

\begin{remark}\label{rmk:partitions extend the modules}
  Suppose that $(\fkA,\CH)$ is an $\crse X$-module and $X=\bigsqcup_{i\in I}A_i$ is a partition with $A_i\in \fkA$ for every $i\in I$. Let $\fkI$ be the boolean algebra consisting of arbitrary unions of the $A_i$, \emph{i.e.}\ $\fkI=\braces{\bigsqcup_{j\in J}A_j\mid J\subseteq I}$.
  Since arbitrary SOT-sums of orthogonal projections are orthogonal projections,
  $\chf{\bullet}$ naturally gives rise to a representation $\fkI\to \CB(\CH)$.
  Even more is true. Namely, one may canonically extend $\chf{\bullet}$ to a representation of the algebra $\widehat\fkA$ consisting of all the sets $B$ such that $B\cap A_i\in\fkA$ for every $i\in I$. Of course $(\widehat \fkA,\CH)$ still is an $\crse X$-module.
\end{remark}

\subsection{Additional properties for coarse geometric modules}
With no further assumptions, coarse geometric modules do not carry enough information about the coarse space. In this subsection we introduce various compatibility properties which turn them into much more useful constructs.
\begin{definition}\label{def:ample}
  An $\crse X$-module is \emph{faithful} if there is a gauge $\gauge\in\CE$ such that the family of non-trivially represented, \gauge-controlled elements of $\fkA$ is coarsely dense, that is,
  \[
    \bigcup\braces{A\in\fkA\mid A\text{ is \gauge-controlled and }\, \chf{A}\neq 0}\asymp X.
  \]
  Moreover, given a cardinal $\kappa$, we say $\CH$ is \emph{$\kappa$-ample} if it is faithful and $\chf{A}$ has rank at least $\kappa$ for every non-trivially represented, \gauge-controlled $A \in \fkA$.
\end{definition}

If $\gauge$ is as in \cref{def:ample}, then we say $\gauge$ is a \emph{faithful} (resp.\ \emph{ample}) gauge. If $\kappa=\aleph_0$, we drop it from the notation.
\begin{remark}
  Henceforth, we shall use the same symbol for the various gauges. Indeed, in most cases one would fix one gauge that is large enough to witness all the required properties.
\end{remark}

\begin{example}\label{exmp: measure on metric spaces}
  Let $\crse X =(X,\CE_d)$ be a metric coarse space (recall \cref{ex:metric-as-coarse-1}), and let $\fkB(X)$ be the Borel $\sigma$-algebra.
  Let $\mu$ be a measure on $X$. Then $L^2(X,\mu)$ with multiplication by indicator functions is a coarse geometric module for $\crse X$. If $\mu$ is fully supported (\emph{i.e.}\ $\mu(\Omega)>0$ for every open set $\Omega \subseteq X$) then $L^2(X,\mu)$ is faithful. It is non-faithful if there exist balls $B(x_n;r_n)\subset X$ of diverging radius such that $\mu(B(x_n;r_n))=0$.
  If $\mu$ is fully supported and non-atomic then $L^2(X,\mu)$ is also ample. On the contrary, if $X$ is discrete then $L^2(X,\mu)$ can never be ample.
\end{example}

\begin{remark}\label{rmk: rank vs local rank}
  In most cases, for instance for locally compact metric spaces, one would only work with separable Hilbert spaces. However, if one wishes to consider spaces such as $\ell^2(X;\CH)$ for some uncountable space $X$, then the notion of ampleness must be adapted to take this into account, whence the definition of $\kappa$\=/ampleness.
  It is generally desirable that the degree of ampleness coincides with the rank of the coarse geometric module under consideration (see \emph{e.g.} \cref{sec:covering iso}).

  If one does not wish to restrict the rank of the whole module, it may suffice to restrict its rank \emph{locally}. Namely, we say that $\CHx$ has \emph{local rank at most $\kappa$} if the rank of $\chf A$ is at most $\kappa$ for every bounded measurable $A\subseteq X$. This condition is mostly useful if $\crse X$ is a space of coarse local cardinality at most $\kappa$, \emph{e.g.}\ if $\kappa=\aleph_0$ and $\crse X$ is coarsely locally finite. 
  The local bound on the rank of modules interacts well with \emph{proper} coarse maps among coarse spaces, see the ``moreover'' statement in \cref{thm:covering isometries exist}.
  A setup where these conditions are rather natural is that of locally compact extended metric spaces with uncountably many coarsely connected components.
\end{remark}

In order to actually do coarse geometry, it is highly desirable to assume that sets admit measurable controlled thickenings. We make this into a definition.

\begin{definition}\label{def:admissible}
  A $\crse X$-module $(\fkA, \CH)$ is \emph{admissible} (resp.\ \emph{locally admissible}) if there is a gauge $\gauge\in\CE$ such that for every $B \subseteq X$ (resp.\ bounded $B \subseteq X$) there is an $A\in \fkA$ with $B \subseteq A \subseteq \gauge(B)$.
\end{definition}

As usual, if $\gauge$ is as above, we say that $\gauge$ is a \emph{(local) admissibility gauge}.
\begin{example}\label{exmp:non-admissible}
  One may construct modules that are not (locally) admissible by choosing a suitably small $\fkA$.
  For instance, consider $\NN$ with the standard metric coarse structure.
  If $\fkA$ consists of the sets that are either finite or co-finite in $\NN$, then any infinite non coarsely-dense subset of $\NN$ does not have a measurable controlled thickening, hence $\ell^2(\NN)$ with this boolean algebra is not admissible.
  Moreover, if we let $\fkA$ consist of finite subsets that do not contain $0$ or cofinite sets that do contain it, the usual representation on $\ell^2(\NN_{>0})$ is a (faithful) module, but the only measurable sets containing $0$ are cofinite, and hence unbounded.
\end{example}

\begin{remark}\label{rmk:modules extend}
  Let $(\fkA, \CH)$ be an $\crse X$-module. Since the projections $\chf{A}\in\CB(\CH)$ commute by assumption, and the set of projections in a commutative von Neumann algebra is a complete boolean algebra, it follows from Sikorski's Extension Theorem that it is possible to extend $\chf{\bullet}$ to a representation
  \[
   \chf{\bullet}\colon \CP(X)\to \overline{\angles{\chf{A}\mid A\in \fkA}}^{w},
  \]
  where $\CP(X)$ is the power set of $X$ and $\overline{(\ldots)}^{w}$ denotes the weak closure. In particular, every $\crse X$-module can be extended to an admissible module, and one could in principle always assume that $\chf{\bullet}$ is defined on every subset of $X$. On the other hand, it does not seem advisable to take this point of view, as taking such an extension may spoil other desirable properties of that module.

  For instance, for $X=\RR$, it is natural to consider the module structure given by ${\rm Borel}(\RR)\to \CB(L^2(\RR))$. Extending this module to all of $\CP(\RR)$ would spoil the $\sigma$-additivity property.
\end{remark}

\begin{remark}
  If a module is admissible, \cref{def:ample} is equivalent to saying that---up to picking a larger gauge if necessary---for every $x\in X$ there is an \gauge-bounded $A\in\fkA$ with $x\in A$ and $\chf{A} \neq 0$ (resp.\ with rank $\kappa$).
\end{remark}

The modules we are most interested in are those that can be ``reconstructed from their local structure'' in the following way.
\begin{definition}\label{def:discrete}
  A $\crse X$-module $(\fkA, \CH)$ is \emph{discrete} if there is a \emph{discreteness gauge} $\gauge\in\CE$ and a partition $X=\bigsqcup_{i\in I}A_i$ such that
  \begin{enumerate}[label=(\roman*)]
    \item \label{def:discrete:contr} $A_i$ is $\gauge$\=/controlled for every $i \in I$;
    \item \label{def:discrete:meas} $A_i\in \fkA$ for every $i\in I$;
    \item \label{def:discrete:sum} $\sum_{i\in I}\chf{A_i} = 1 \in \CB(\CH)$ (the sum is in the strong operator topology);
    \item \label{def:discrete:extension} if $C \subseteq X$ is such that $C \cap A_i \in \fkA$ for all $i \in I$ then $C \in \fkA$.
  \end{enumerate}
  We call such a partition $(A_i)_{i \in I}$ a \emph{discrete partition}.
\end{definition}

\begin{remark}\label{rmk:discrete module}
  \begin{enumerate}[label=(\roman*)]
    \item Due to \cref{def:discrete}~\cref{def:discrete:extension}, if $X=\bigsqcup_{i\in I}A_i$ is a discrete partition then the block diagonal $E\coloneqq \bigsqcup_{i\in I}A_i\times A_i$ is an admissibility gauge for $\CH$. This shows that every discrete module is admissible.
    \item If $\CH$ is an $\crse X$-module admitting a measurable partition $X=\bigsqcup_{i\in I}A_i$ that satisfies \cref{def:discrete:contr,def:discrete:meas,def:discrete:sum}, then \cref{rmk:partitions extend the modules} shows that we may canonically extend $\fkA$ and $\chf{\bullet}$ so that \cref{def:discrete:extension} holds as well.
  \end{enumerate}
\end{remark}

As expected, the most important modules are discrete.
\begin{example}\label{exmp:uniform module}
  Given any coarse space $\crse X= (X,\CE)$, the \emph{uniform module $\CH_{u,\crse X}$} is defined by the representation $\chf \bullet\colon\CP(X)\to\CB(\ell^2(X))$ given by multiplication by the indicator function.
  The uniform module is clearly discrete.

  Similarly, for any cardinal $\kappa$, we may also fix a Hilbert space $\CH$ of rank $\kappa$ and define the \emph{uniform rank-$\kappa$ module $\CH_{u,\crse X}^{\kappa}$} as the tensor product $\CH_{u,\crse X}^{\kappa}\coloneqq \CH_{u,\crse X}\otimes\CH = \ell^2(X;\CH)$.
  Then $\CH_{u, \crse X}^\kappa$ is discrete as well.
\end{example}

\begin{example} \label{ex:ell-2-representations}
  As in \cref{exmp: measure on metric spaces}, let $\mu$ be any measure on the Borel $\sigma$-algebra of a metric space $(X,d)$, and consider the coarse geometric module $L^2(X,\mu)$. If $(X,d)$ is separable, this module is discrete. This is easily shown by considering a countable controlled measurable partition $X=\bigsqcup_{n\in\NN}A_n$ and observing that $L^2(X, \mu)=\bigoplus_{n\in\NN}L^2(A_n, \mu|_{A_n})$.
\end{example}

\begin{remark}\label{rmk:discrete modules are discrete}
  If $(\fkA,\CH)$ is a discrete $\crse X$-module, we may make $I$ into a coarse space $\crse I=(I,\CF)$ coarsely equivalent to $\crse X$ as explained in \cref{rmk:controlled partitions give crs eq}.
  Furthermore, we may also use $(\fkA,\CH)$ to naturally define a coarse geometric module $(\CP(I),\CH)$ on $\crse I$.
  We observe that $\CH$ is isomorphic to the Hilbert sum $\bigoplus_{i\in I}\CH_{A_i}$. We can thus make the discreteness of this module particularly explicit by letting $\CH_i\coloneqq \CH_{A_i}$ and rewriting $(\CP(I),\chf{\bullet},\CH)$ as $(\CP(I),\,\delta,\, \bigoplus_{i\in I}\CH_i)$, where we let $\delta_i\coloneqq\chf{A_i}$ and $\delta(J)\coloneqq\sum_{j\in J}\delta_j$.

  Also observe that if every $\CH_{i}$ is isomorphic to some fixed Hilbert space $\CH'$ of rank $\kappa$, we may fix a choice of such isomorphisms and write the above module as $(\CP(I), \ell^2(I;\CH'))$, which is just the uniform $\crse I$-module of rank $\kappa$ (see \cref{exmp:uniform module}). These rewritings are purely formal, but they do bridge with the notation that is most often used in the literature when working with (uniform) Roe algebras of coarse spaces~\cites{bbfvw_2023_embeddings_vna,braga_rigid_unif_roe_2022,braga_farah_rig_2021,braga_gelfand_duality_2022}.
\end{remark}

The following shows that most of the geometric modules that one may reasonably expect to encounter in applications are indeed discrete (up to extension).
\begin{proposition}\label{prop:admissible is discrete}
  If $\crse X = (X, \CE)$ is coarsely locally finite and $(\fkA,\chf{\bullet}, \CH)$ is a locally admissible $\crse X$-module, then $\fkA$ and $\chf{\bullet}$ can be extended to a define a discrete module.
\end{proposition}
\begin{proof}
  As explained in \cref{rmk:discrete module}, it is enough to find a measurable partition such that $(\fkA,\chf{\bullet}, \CH)$ satisfies \cref{def:discrete:contr,def:discrete:meas,def:discrete:sum} of \cref{def:discrete}: condition \cref{def:discrete:extension} can then be imposed passing to a canonical extension.

  Let $X=\bigsqcup_{i\in I}B_i$ be a locally finite controlled partition. By local admissibility, we may fix for every $i\in I$ a controlled thickening $A_i\in\fkA$ with $B_i\subseteq A_i\subseteq \gauge(B_i)$. Thus $X=\bigcup_{i\in I}A_i$.

  Observe that the family $\paren{A_i}_{i\in I}$ is still locally finite. Indeed, given a bounded $B\subseteq X$ and $i \in I$ with $A_i\cap B\neq\emptyset$, the set $\gauge(B)$ is still bounded and $\gauge(B)\cap B_i\neq \emptyset$. Local finiteness of $\paren{B_i}_{i\in I}$ hence shows that this only happens for finitely many $i\in I$.

  Arbitrarily choose a total ordering $\leq$ on $I$, and let
  \begin{equation}\label{eq:take the difference}
    C_i\coloneqq A_i\smallsetminus\Bigpar{\bigcup_{j< i}A_j}.
  \end{equation}
  The sets $C_i$ are pairwise disjoint. Moreover, since $\paren{A_i}_{i\in I}$ is locally finite, $C_i$ is equal to the difference between $A_i$ and finitely many other $A_j$, whence $C_i\in\fkA$. Moreover, for every $x\in X$ there are only finitely many $j$ so that $x\in A_j$, so we see that $x\in C_{i_0}$ for $i_0 \coloneqq \min\{j\in I\mid x\in A_j\}$. Hence
  \begin{equation}\label{eq:partition}
    X=\bigsqcup_{i\in I}C_i
  \end{equation}
  is a locally finite controlled partition.

  It remains to show that the sum of the $\chf{C_i}$ converges to the identity in the strong operator topology. Since the projections $\chf{C_i}$ are orthogonal we already know that their sum strongly converges to some projection, call it $p$.
  By non-degeneracy of $\chf{\bullet}$, for every vector $v\in \CH$ and $\varepsilon>0$ there is an $A\in \fkA$ that is a finite union of $\gauge$-bounded sets and $\norm{v-\chf{A}(v)}<\varepsilon$ (see \cref{lem:supports of vectors almost contained in bounded}).
  By local finiteness of $(C_i)_{i \in I}$ there is a finite $J\subseteq I$ so that $A \subseteq \widetilde A\coloneqq \bigsqcup_{j\in J}C_j$, whence
  \begin{equation}\label{eq:convergence to identity}
    \norm{v-\chf{\widetilde A}(v)}\leq \norm{v-\chf{A}(v)}<\varepsilon.
  \end{equation}
  Since $\chf{\widetilde A}\leq p\leq 1$, it follows that $\norm{v-p(v)}\leq \norm{v-\chf{\widetilde A}(v)} \leq \varepsilon$. The claim follows letting $\varepsilon$ tend to zero.
\end{proof}

\begin{remark}
  The proof of \cref{prop:admissible is discrete} can be adapted to show that if $(\fkA, \CH)$ is an admissible $\sigma$-additive $\crse X$-module (recall \cref{conv:sigma additive module}) and $\crse X$ is locally countable (this is defined as one would expect), then $(\fkA, \CH)$ is discrete.
  In fact, \cref{eq:take the difference} defines an element in $\fkA$, because $\fkA$ is a $\sigma$-algebra. In order to get \cref{eq:partition} one has to make sure that the $C_i$ cover $X$, so one should assume that $(I,\leq)$ is a well-ordered set.
  Finally, \cref{eq:convergence to identity} holds because local countability implies that $J$ can be taken to be countable, so that one can use the $\sigma$\=/additivity of $\chf{\bullet}$.
\end{remark}

\begin{remark}
  Our notion of coarse geometric module extends that of \cite{winkel2021geometric}. 
  In that paper, Winkel considers Hilbert spaces of the form $L^2(X,\mu)$ where $\mu$ is a measure defined on a (arbitrary) $\sigma$-algebra containing all singletons. Bounded Borel functions are then represented by multiplication. 
  Without further assumptions, this data needs not define a coarse geometric module, as the representation needs not be non-degenerate (in the sense of \cref{def:non-degenerate}).
  However, the assumptions that \cite{winkel2021geometric} makes from Section 4 onward do imply that $X$ is coarsely locally finite and $L^2(X,\mu)$ is an admissible coarse module. It follows from \cref{prop:admissible is discrete} that these modules are even discrete (see also \cite{winkel2021geometric}*{Lemma 4.2}).
\end{remark}

As in \cref{exmp:non-admissible}, one may easily construct examples of non-discrete geometric modules by choosing a suitably small $\fkA$. Nevertheless, these examples are somewhat unsatisfactory because every module can be extended to a module defined on $\CP(X)$ (recall \cref{rmk:modules extend}). We do not know whether every geometric module can be extended to a discrete one. On the other hand, it is not hard to come up with examples of a degenerate representation that cannot be extended to a discrete one. For instance, consider $\NN$ with the coarse structure $\varcrs{fin}$ consisting of entourages with finitely many off-diagonal points. Then the natural representation $\CP(\NN)\to\CB(\ell^2(\NN))$ is not discrete.

\subsection{Technical lemmas for future references}
We collect here several basic lemmas on coarse geometric modules that will be useful later. Of particular interest are \cref{cor:exists disjoint ample coarsely dense,cor:discrete and faithful}, which allow us to find partitions of $X$ witnessing several structural properties simultaneously. 
\begin{lemma}\label{lem:exists disjoint coarsely dense}
  Let $(A_l)_{l\in L}$ be a \gauge-controlled family of subsets of $X$ such that $\bigcup_{l\in L}A_l$ is coarsely dense, and let $E$ be any gauge. Then there exists
  $I\subseteq L$ such that $A_i$ and $A_{i'}$ are $E$-separated for every $i\neq i'\in I$ and $\bigsqcup_{i\in I}A_i$ is coarsely dense.
\end{lemma}
\begin{proof}
  Let $I\subseteq L$ be a maximal subset such that $A_i\cap E(A_{i'})=\emptyset$ for every $i\neq i'\in I$ (such $I$ exists by Zorn's Lemma).
  Let $E'$ be a gauge witnessing the coarse denseness of $\bigcup_{l\in L}A_l$. Therefore,
  for every $x\in X$ there is some $l\in L$ with $x\in E'(A_l)$. By maximality of $I$, there also is $i\in I$ such that $A_l\cap E(A_i)\neq \emptyset$. It follows that $x\in E'\circ\gauge\circ E(A_i)$.
\end{proof}

\begin{corollary}\label{cor:exists disjoint ample coarsely dense}
  Let $\CH$ be a faithful (resp.\ $\kappa$-ample) $\crse X$-module. Then there is an \gauge-controlled family $(A_i)_{i\in I}$ of measurable pairwise disjoint subsets of $X$ such that $\chf{A_i}$ is non-zero (resp.\ has rank $\kappa$) for every $i\in I$ and $\bigsqcup_{i\in I}A_i$ is coarsely dense.
\end{corollary}

\begin{lemma}\label{lem:discrete have nice partitions}
  Let $\CH$ be a discrete module and let $(B_l)_{l\in L}$ be a controlled family of subsets such that $\bigcup_{l\in L}B_l$ is coarsely dense. Then we may find a discrete partition $X=\bigsqcup_{i\in I}A_i$ such that for every $i\in I$ there exists an $l\in L$ with $B_l\subseteq A_i$.
\end{lemma}
\begin{proof}
  Fix gauges $E_{\rm disc},\ E_{\rm dense}$ and $E_{\rm ctrl}$ such that there is an $E_{\rm disc}$-controlled discrete partition $X=\bigsqcup_{j\in J}C_j$; the family $(B_l)_{l \in L}$ is $E_{\rm ctrl}$\=/controlled; and $X = E_{\rm dense}(\bigcup_{l \in L}B_l)$.
  Consider:
  \[
    E_1 \coloneqq E_{\rm dense}\circ E_{\rm ctrl}\circ E_{\rm disc}\circ E_{\rm ctrl}\circ E_{\rm dense} \in \CE.
  \]
  By \cref{lem:exists disjoint coarsely dense} there is some $I\subseteq J$ such that $C_i$ and $C_{i'}$ are $E_1$-separated for every $i\neq i'\in I$ and $\bigsqcup_{i\in I}C_i$ is coarsely dense.
  For every $i\in I$, let
  \[
    J_i\coloneqq\{j\mid C_j\cap E_{\rm ctrl}\circ E_{\rm dense}(C_i)\neq \emptyset\},
  \]
  and define $A'_i\coloneqq\bigsqcup_{j\in J_i}C_j$. Observe that each $A_i$ is $E_2$-controlled for
  \[
    E_2 \coloneqq
    E_{\rm disc} \circ
    \paren{ E_{\rm ctrl}\circ E_{\rm dense}}
    \circ E_{\rm disc}\circ
    \paren{ E_{\rm dense}\circ E_{\rm ctrl}}
    \circ E_{\rm disc}.
  \]
  Since $(C_j)_{j \in J}$ partitions $X$, it follows that $E_{\rm ctrl}\circ E_{\rm dense}(C_i)\subseteq A_i'$.
  By the choice of $E_1$ it follows that $A'_i\cap A'_{i'}=\emptyset$ for every $i\neq i'$, 

  Let $E_3$ be a gauge witnessing coarse denseness of the $\paren{C_i}_{i \in I}$. Then for every $j \in J\smallsetminus\bigcup_{i\in I}J_i$ we may arbitrarily choose an $i\in I$ with $C_j\cap E_3(C_i)\neq \emptyset$ and attach the set $C_j$ to the previously defined $A_i'$.
  After all attachments are done, we obtain new sets $\paren{A_i}_{i \in I}$ which form a measurable partition of $X$ that is controlled by
  \[
    E_2\cup(E_3\circ E_{\rm disc}\circ E_3).
  \]
  Since the $A_i$ are unions of elements of a discrete partition, the partition $\bigsqcup_{i\in I}A_i$ is itself discrete.

  Observe that for every $i\in I$ there is an $l\in L$ such that $C_i\cap E_{\rm dense}(B_l)\neq\emptyset$, hence $B_l$ is contained in $E_{\rm ctrl}\circ E_{\rm dense}(C_i)\subseteq A_i$.
\end{proof}

\begin{corollary}\label{cor:discrete and faithful}
  Let $\CH$ be a discrete and faithful (resp.\ $\kappa$\=/ample) module. Then there is a discrete partition $X=\bigsqcup_{i\in I}A_i$ such that $\chf{A_i}$ is non-zero (resp.\ has rank $\kappa$) for every $i\in I$.
\end{corollary}

\begin{corollary}\label{cor:discrete and bounded geometry}
  Let $\CH$ be a discrete module and suppose that $\crse X$ has bounded geometry (in the sense of \cref{def:locally finite space - bounded geo}). Then there is a uniformly locally finite discrete partition $X=\bigsqcup_{i\in I}A_i$.
\end{corollary}
\begin{proof}
  Apply \cref{lem:discrete have nice partitions} to a uniformly locally finite controlled partition $X=\bigsqcup_{l\in L}B_l$. Since the resulting partition is controlled, it is also easy to verify that it is uniformly locally finite.
\end{proof}

\section{Operators among coarse geometric modules}
We now shift the study, and start investigating operators on and among coarse geometric modules. 
The most important notion that we introduce here is that of \emph{coarse support} of an operator in \cref{def:coarse support}. This will then be used in the definition of most other coarse-geometric properties of operators.
\begin{convention}
  Unless otherwise specified, for the rest of the paper $\crse X= (X,\CE)$ and $\crse Y=(Y,\CF)$ are fixed coarse spaces equipped with coarse geometric modules $\CHx$ and $\CHy$ respectively.
\end{convention}

\subsection{Supports of operators}\label{sec:support}
Recall that if $R\subseteq Y\times X$ is a relation, we say that $A\subseteq X$ and $B\subseteq Y$ are \emph{$R$-separated} (written as \emph{$B\sep{R}A$}) if $B\cap R(A)=\emptyset$.
\begin{definition} \label{def:supp-operator}
  Given $t\colon \CHx\to\CHy$ and $S\subseteq Y\times X$, we write \emph{$\supp(t)\subseteq S$} if $\chf B t\chf A = 0$ for any pair of $S$-separated measurable subsets $B\subseteq X$,  $A\subseteq Y$.
\end{definition}

Naturally, if $\supp(t) \subseteq S$ we say \emph{the support of $t$ is contained in $S$} (see \cref{rem:coarse support} below). Containments of supports are fairly well behaved with respect to operations of operators, as the following shows.
\begin{lemma}\label{lem:supports and operations}
  Given $r,s\colon\CHx\to\CHy$ and $t\colon\CHy\to \CHz$ with $\supp(r)\subseteq S_r$, $\supp(s)\subseteq S_s$ and $\supp(t)\subseteq S_t$, then:
  \begin{enumerate} [label=(\roman*)]
    \item \label{lem:supports and operations:sum} $\supp(r+s)\subseteq S_r\cup S_s$;
    \item \label{lem:supports and operations:star} $\supp(r^*)\subseteq \op{S_r}$;
    \item \label{lem:supports and operations:comp} $\supp(ts)\subseteq S_t\circ\gauge_{\crse Y}\circ S_s$, where $\gauge_{\crse Y}$ is any non-degeneracy gauge for $\CHy$.
  \end{enumerate}
\end{lemma}
\begin{proof}
  \cref{lem:supports and operations:sum} and \cref{lem:supports and operations:star} are straightforward, so we will only prove \cref{lem:supports and operations:comp}.
  Assume $A_X\subseteq X$ and $A_Z\subseteq Z$ are measurable sets with $\chf{A_Z}ts\chf{A_X}\neq 0$. Fix $v\in \CHx$ and $w\in\CHz$ with
  \[
    0\neq \angles{w,\chf{A_Z}ts\chf{A_X}(v)}=\angles{t^*\chf{A_Z}(w),s\chf{A_X}(v)}.
  \]
  By \cref{lem:non-orthogonality is local}, there is a  $\gauge_{\crse Y}$-bounded measurable set $B\subseteq Y$ such that
  \[
     0\neq \angles{\chf{B}t^*\chf{A_Z}(w),\chf{B}s\chf{A_X}(v)}.
  \]
  In particular, both $B\cap S_s(A_X)$ and $A_Z\cap \op{S_t}(B)$ are non-empty (we used \cref{lem:supports and operations:star} for the second one). It follows that $A_Z\cap \paren{S_t\circ\gauge_{\crse Y}\circ S_s}(A_X)\neq\emptyset$.
\end{proof}

\begin{remark}
  It is tempting to drop the non-degeneracy gauge and state part \cref{lem:supports and operations:comp} of \cref{lem:supports and operations} as $\supp(ts)\subseteq S_t\circ S_s$. The tentative proof would proceed by observing that for every $A\subseteq X$
  \[
    s\chf{A_X}=(\chf{S_s(A_X)}+\chfcX{S_s(A_X)})s\chf{A_X} =\chf{S_s(A_X)} s\chf{A_X},
  \]
  as $\chfcX{S_s(A_X)} s\chf{A_X}=0$ because $s$ is supported on $S_s$. Similarly, one would then have
  that $ts\chf{A_X}= \chf{S_t(S_s(A_X))}t\chf{S_s(A_X)} s\chf{A_X}$, which would imply that $\chf{B}ts\chf{A_X}$ vanishes whenever $B$ and $A$ are $(S_t\circ S_s)$-separated.
  The issue with this argument is that it may well be the case that $A_X$ is measurable but $S_s(A_X)$ is \emph{not}, so the above formal writing does not make sense.

  This, in fact, is a fundamental issue. It is essential to be able to control the support of the composition of operators and, for this, it is crucial to include the non-degeneracy assumption in the definition of $\crse X$-module. In \cite{winkel2021geometric} an even stronger countermeasure is taken, namely the class of modules there considered is further restricted so that one can slightly enlarge any relation $S$ to some relation $S'$ with the property that $S'(A)$ is measurable whenever $A$ is measurable. This can always be assumed for instance when $\fkA$ is a Borel $\sigma$-algebra on a metric space, as it is then enough to pick a small open neighborhood of $S$.
\end{remark}

Observe that \cref{def:supp-operator} does \emph{not} define what the support of an operator actually is. This is due to the fact that there is no natural way to choose a subset of $Y\times X$ that supports $t$. However, we will be able to define the support ``coarsely'' (see \cref{def:coarse support} below).
\begin{remark} \label{rem:coarse support}
  If $X$ and $Y$ are metric spaces and $\CHx$, $\CHy$ are geometric modules in the classical sense (see~\cite{willett_higher_2020} or \cref{subsec:classical-roe-alg} above), one usually defines the support of $t$ as the smallest closed relation $R$ with $\supp(t)\subseteq R$.
  Such an approach is unreasonable for general coarse spaces, as it relies on topological data.
  On the contrary, the approach here taken is closer to that in measure theory: when one cannot use the topology to define the support of a measurable function $f$, the second best thing to consider is the \emph{essential support} of $f$. The key point here is that, although the essential support is not a well defined set, it is well defined up to measure 0. In the same way, we have the following.
\end{remark}

\begin{definition}\label{def:coarse support}
  The \emph{coarse support} of an operator $t\colon \CHx\to \CHy$ is the minimal coarse subspace $\csupp(t)\crse{\subseteq Y\times X}$ containing the support of $t$.
\end{definition}

Explicitly, $\csupp(t)=[S]$ if $S$ is such that $\supp(t) \subseteq S$ and whenever $\supp(t) \subseteq R$ then $S\csub R$. Since the coarse support is defined by minimality, it is clear that if it exists it is unique.
The following shows it always exists (compare with \cref{rem:coarse-intersection}).
\begin{proposition}\label{prop:coarse_support_exists}
  Let $\gaugex$ and $\gaugey$ be non-degeneracy gauges for $\CHx$ and $\CHy$ respectively, and let $t\colon\CHx\to\CHy$ be a bounded operator. Let
  \[
    S\coloneqq \bigcup\left\{B\times A \;\middle|
    \begin{array}{c}A\ \text{measurable, \gaugex-bounded}\\ B\text{ measurable, \gaugey-bounded}\end{array},\ \chf{B}t\chf{A}\neq 0\right\} \subseteq Y \times X.
  \]
  Then $[S]$ is the coarse support of $t$.
\end{proposition}
\begin{proof}
  We first show that $\supp(t)\subseteq S$. Suppose $\chf{B_0}t\chf{A_0}\neq 0$. Then there are $v\in\chf{A_0}(\CHx)$ and $w\in\chf{B_0}(\CHy)$ such that $\scal{w}{t(v)}=\scal{t^*(w)}{v}\neq 0$. We may apply \cref{lem:non-orthogonality is local} twice to find a measurable \gaugex-bounded $A$ with
  \[
    0\neq\scal{\chf{A}t^*(w)}{\chf{A}(v)}=\scal{w}{t\chf{A}(v)}
  \]
  and a measurable \gaugey-bounded $B$ with
  \[
    0\neq\scal{\chf{B}(w)}{t\chf{A}(v)}.
  \]
  By construction, $B\times A\subseteq S$. Since $v = \chf{A_0} (v)$ and $w = \chf{B_0} (w)$, the intersections $A_0\cap A$ and $B_0\cap B$ cannot be empty, and therefore $B_0$ and $A_0$ are not $S$-separated.

  Suppose now that $\supp(t)\subseteq R$ for some $R$. Let $A,B$ be measurable and \gaugex-bounded (resp.\ \gaugey-bounded) such that $\chf{B}t\chf{A}\neq 0$. Then $B$ and $A$ cannot be $R$\=/separated. It follows that $B\times A\subseteq \gaugey\circ R\circ \gaugex$. This shows that $S\subseteq \gaugey\circ R\circ \gaugex$. Thus $S\csub R$, as desired.
\end{proof}

Having assessed that the coarse support of an operator is well defined, we may define when an operator is \emph{controlled}.
\begin{definition}\label{def:controlled operator}
  An operator $t\colon\CHx\to\CHy$ is \emph{controlled} if its coarse support is a partial coarse map $\csupp(t)\crse{\colon X\to Y}$.
\end{definition}

By definition, if $\csupp(t)=[S]$ then $t$ is controlled if and only if $S \subseteq Y\times X$ is a controlled relation (recall \cref{def:controlled_function}). Explicitly, this is the case if and only if  $S\circ E\circ \op S\in\CF$ for every $E\in\CE$. This explicit characterisation works well together with the explicit description of the coarse support obtained in \cref{prop:coarse_support_exists}. We observe that \cref{lem:coarse composition is well-defined} implies the following (compare with \cref{lem:supports and operations}).
\begin{corollary}\label{cor:coarse support of composition}
  Let $r,s\colon\CHx \to \CHy$ and $t\colon \CHy \to \CHz$. Then:
  \begin{enumerate} [label=(\roman*)]
    \item \label{cor:coarse support of composition:sum} $\csupp(r+s) \crse\subseteq \csupp(s) \crse\cup \csupp(r)$;
    \item \label{cor:coarse support of composition:star} $\csupp(r^*) = \op{\csupp(r)}$;
    \item \label{cor:coarse support of composition:comp} if $t$ is controlled and $\crse\pi_{\crse Y}(\csupp(s)) \crse \subseteq \crse\pi_{\crse Y}(\csupp(t))$, then 
    \[
      \csupp(ts) \crse\subseteq \csupp(t)\crse\circ\csupp(s).
    \]
  \end{enumerate}
\end{corollary}
\begin{proof}
  As was the case in the proof of \cref{lem:supports and operations}, \cref{cor:coarse support of composition:sum,cor:coarse support of composition:star} are routine to show. For \cref{cor:coarse support of composition:comp},
  since $\csupp(t)$ is a partial coarse map, \cref{lem:coarse composition is well-defined} implies that the composition $\csupp(t)\crse\circ\csupp(s)$ is a coarsely well-defined coarse subspace of $\crse{Z\times X}$.
  On the other hand, if we choose representatives $S_s$ and $S_t$ for the supports of $s$ and $t$ respectively, then \cref{lem:supports and operations} shows that $\supp(ts)\subseteq S_t\circ\gauge_{\crse Y}\circ S_s$, where $\gauge_{\crse Y}$ is a non-degeneracy gauge for $\CHy$.
  Since $\gauge_{\crse Y}\circ S_s$ is just another representative for $[S_s]=\csupp(s)$, it follows from the definition of the coarse composition (cf.\ \cref{def: coarse composition}) that
  \(
    [S_t\circ\gauge_{\crse Y}\circ S_s] = [S_t\circ\paren{\gauge_{\crse Y}\circ S_s}]
    \crse \subseteq 
    \csupp(t)\crse\circ\csupp(s)
  \).
\end{proof}

Continuing with our philosophy of confounding operators with their coarse supports, we also give the following.
\begin{definition} \label{def: proper operator}
  $s\colon \CHx\to\CHy$ is \emph{proper} if $\csupp(s)\subseteq \crse{Y\times X}$ is a proper coarse subspace (in the sense of \cref{def:proper}).
\end{definition}

It is worthwhile observing that, in most cases, properness also has a more operator-algebraic characterisation.
\begin{lemma}\label{lem: properness equivalent defs}
  Given $t\colon\CHx\to\CHy$, consider the following:
  \begin{enumerate}[label=(\roman*)]
    \item\label{lem: properness equivalent defs_mods} for every measurable bounded $B\subseteq Y$ there is a measurable bounded $A\subseteq X$ such that $\chf{B}t=\chf{B}t\chf{A}$;
    \item\label{lem: properness equivalent defs_supp} $t$ is proper.
  \end{enumerate}
  If $\CHy$ is locally admissible then \cref{lem: properness equivalent defs_mods} implies \cref{lem: properness equivalent defs_supp}. If $\CHx$ is locally admissible, the converse implication is true.
\end{lemma}
\begin{proof}
  To show that \cref{lem: properness equivalent defs_mods} implies \cref{lem: properness equivalent defs_supp} it is enough to consider the concrete construction of coarse support given in \cref{prop:coarse_support_exists}. Namely, we have to show that if $S$ is the set there defined and $B\subseteq Y$ is any bounded subset then $\op{S}(B)\subseteq X$ is bounded.
  Since $\CHy$ is locally admissible, up to passing to a controlled thickening we may assume that $B$ is measurable. Let $B''$ be a measurable controlled thickening of $\gaugey(B)$, and let $A\subseteq X$ denote the bounded set obtained applying \cref{lem: properness equivalent defs_mods} to $B''$. Suppose that $B'\times A'\subseteq S$ is as in the definition of $S$, with $B'\cap B\neq 0$. It follows that $B'\subseteq B''$. Hence we see that
  \[
    0<\norm{\chf{B'}t\chf{A'}}
    \leq \norm{\chf{B''}t\chf{A'}}
    = \norm{\chf{B''}t\chf{A}\chf{A'}},
  \]
  whence $A'\cap A$ is non-empty. This shows that $\op{S}(B)\subseteq \gaugex(A)$ is bounded.

  The converse implication is simpler. Let $S$ be a representative for $\csupp(t)$ with $\supp(t)\subseteq S$, and let $B\subseteq Y$ be measurable and bounded. By properness, $\op{S}(B)$ is bounded. Since $\CHx$ is locally admissible, we may thus pick a measurable controlled thickening $A$ thereof. In particular, $B\sep{S}(X\smallsetminus A)$, and we then obtain:
  \[
    \chf{B}t = \chf{B}t\chf{A} + \chf{B}t\chfcX{A} = \chf{B}t\chf{A}
  \]
  as was desired.
\end{proof}

\subsection{More properties of operators: propagation and local compactness}\label{sec:operators and propagation}
In the sequel, we want to consider algebras of operators ``with bounded displacement''.
The following is a key notion to define algebras of operators of coarse-geometric significance.
\begin{definition} \label{def: controlled propagation operator}
  An operator $t \in \CB(\CHx)$ has \emph{controlled propagation} if $\csupp(t)\crse\subseteq \cid_{\crse X}=[\Delta_X]$. Explicitly, this is the case if and only if its support is controlled by $E$ for some $E \in \CE$ (in which case we say that $t$ has \emph{$E$-controlled propagation}).
\end{definition}

An alternative name for the operators of controlled propagation would be ``translation operators'', as it is common to refer to functions with bounded displacement as ``translations''.
Since coarse subspaces of (partial) coarse maps are themselves (partial) coarse maps, an operator with controlled propagation is certainly controlled (in the sense of \cref{def:controlled operator}). The opposite is clearly false, since there are coarse subspaces that are \emph{not} contained in $[\Delta_X]$.

There are weakenings of \cref{def: controlled propagation operator} that arise naturally when entering into more complex investigations (\emph{e.g.}\ \cites{braga_rigid_unif_roe_2022,bbfvw_2023_embeddings_vna,rigid}). In the following definitions, $E$ is always a controlled entourage in $\CE$ and $\varepsilon$ is some positive constant.
\begin{definition} \label{def:approx-ope}
  An operator $t \in \CB(\CHx)$ is
  \begin{enumerate} [label=(\roman*)]
    \item \emph{$\varepsilon$-$E$-approximable} if there exists $s \in \CB(\CHx)$ with $E$-controlled propagation such that $\norm{s - t} \leq \varepsilon$;
    \item \emph{$\varepsilon$-approximable} if there is $E \in \CE$ such that $t$ is $\varepsilon$-$E$-approximable;
    \item \emph{approximable} if $t$ is $\varepsilon$-approximable for all $\varepsilon > 0$.
\end{enumerate}
\end{definition}

\begin{definition} \label{def:ql-ope}
  An operator $t \in \CB(\CHx)$ is
  \begin{enumerate} [label=(\roman*)]
    \item \emph{$\varepsilon$-$E$-quasi-local} if $\norm{\chf {A'} t\chf A} \leq \varepsilon$ for every choice of measurable $E$-separated subseteq $A, A' \subseteq X$ (that is, $A'\sep{E}A$);
    \item \emph{$\varepsilon$-quasi-local} if there is $E \in \CE$ such that $t$ is $\varepsilon$-$E$-quasi-local;
    \item \emph{quasi-local} if $t$ is $\varepsilon$-quasi-local for all $\varepsilon > 0$.
  \end{enumerate}
\end{definition}

\begin{remark}\label{rmk:approximable is ql}
  It is immediate to verify that $\varepsilon$-$E$-approximable operators are $\varepsilon$-$E$-quasi-local. In particular, approximable operators are always quasi-local.
  If $X$ is a bounded geometry metric space with Yu's property A (see \cite{yu_coarse_2000}*{Definition~2.1}) then every quasi-local operator is approximable (see \cite{spakula-zhang-2020-quasi-loc-prop-a}*{Theorem~3.3}). Quite recently, Ozawa~\cite{ozawa-2023} proved that there are quasi-local operators that are \emph{not} approximable, answering a long-standing open problem. The existence of these operators, of course, relies on $X$ being far from having Yu's property A, and in fact $X$ consists of a family of expander graphs.
\end{remark}

It is sometimes desirable to test quasi-locality ``in a local manner''. The next lemma shows that the non-degeneracy assumption on $\CHx$ allows this.
\begin{lemma}\label{lem:quasi-locality locally}
  An operator $t \in \CB(\CHx)$ is quasi-local if and only if for every $\varepsilon>0$ there is $E\in\CE$ such that $\norm{\chf{B'}t\chf B} \leq \varepsilon$ for every $B,B'\subseteq X$ bounded, measurable, and $E$-separated.
\end{lemma}
\begin{proof}
  The \emph{only if} direction is immediate from the definition. For the \emph{if} direction. Given $\varepsilon > 0$ fix $E$ as in the hypothesis, and let $A',A\subseteq X$ be two $E$-separated measurable subsets.
  Observe that
  \[
    \norm{\chf{A'}t\chf A}=\sup\{\abs{\scal{v'}{t(v)}}\mid v'\in\CH_{A'} \; \text{and} \; v\in\CH_A\text{ unit vectors}\},
  \]
  where, as usual, $\CH_{A'} = \chf{A'}(\CHx)$ and $\CH_A=\chf{A}(\CHx)$.
  Fix $v' \in \CH_{A'}$ and $v \in \CH_A$ realizing the supremum up to small $\delta>0$. By \cref{lem:supports of vectors almost contained in bounded}, there are $C'=\bigsqcup_{j=1}^{n'} B_{j'}'$ and $C=\bigsqcup_{j=1}^n B_j$,
  which are finite disjoint unions of bounded measurable subsets such that both $\norm{v'-\chf {C'}(v)}$ and $\norm{v-\chf{C} (v)}$ are smaller than $\delta$. Intersecting with $A'$ and $A$ if necessary we may further assume that $C'$ and $C$ are also $E$-separated.

  Let $\crse{X} = \bigsqcup_{i\in I} \crse X_i$ be the decomposition of $\crse X$ in coarsely connected components (recall \cref{def: coarsely connected}). The union of the sets $B_j$ (resp.\ $B_{j'}'$) that belong to the same coarsely connected component $X_i$ is still bounded. We may thus reorganize the partitions of $C'$ and $C$ in such a way that $B_i, B_i'\subseteq X_i$ for every $i\in I$ (these may also be empty).
  By the triangle inequality,
  \begin{equation}\label{eq:squasi-locality locally}
    \abs{\scal{v'}{t(v)}} \leq 2\delta +\abs{\scal{\chf{C'}(v')}{t\chf C(v)}}\leq 2\delta + \sum_{i',i\in I}\abs{\scal{\chf{B_{i'}'}(v')}{t\chf {B_i}(v)}}.
  \end{equation}
  Observe that if $i'\neq i$ then $B_{i'}$ and $B_i$ are $E'$-separated for every $E'\in \CE$ (as they are in different components of $\crse X$), and hence $\scal{\chf{B_{i'}'}(v')}{t\chf {B_i}(v)}=0$ from the assumption on $t$. We may thus rewrite the sum on the RHS of \eqref{eq:squasi-locality locally} as
  \[
    \sum_{i\in I}\abs{\scal{\chf{B_{i}'}(v')}{t\chf {B_i}(v)}} \leq \sum_{i\in I}\varepsilon\norm{\chf{B_{i}'}(v')}\norm{\chf{B_{i}}(v)}\leq \varepsilon \norm{\chf{C'}(v')}\norm{\chf{C}(v)}\leq \varepsilon,
  \]
  where the second-to-last inequality is an application of Cauchy--Schwarz. The statement follows by letting $\delta$ go to zero.
\end{proof}

\begin{remark}
  The proof of \cref{lem:quasi-locality locally} actually shows that if $\crse X$ is coarsely connected then $t\in\CB(\CHx)$ is $\varepsilon$\=/$E$\=/quasi-local if and only if $\norm{\chf{B'}t\chf B}\leq\varepsilon$ for every $B,B'\subseteq X$ bounded measurable $E$-separated subsets. That is, when $\crse X$ is coarsely connected an effective version of \cref{lem:quasi-locality locally} is true.
  This effective statement is not true in general. Namely, knowing the ``local'' $\varepsilon$\=/$E$\=/quasi-locality for some fixed $E$ does not imply that the operator is ``globally'' $\varepsilon$\=/$E$\=/quasi-local for the same $E$.

  Indeed, in order to obtain it in the disconnected case it is necessary to know \emph{a priori} that $\scal{\chf{B_{i'}'}(v')}{t\chf {B_i}(v)}=0$ whenever $B_{i'}'$ and $B_i$ belong to different connected components.
  This failure of locality will have as a consequence that in the implication \cref{prop:coarse functions preserve controlled propagation:1} $\Rightarrow$ \cref{prop:coarse functions preserve controlled propagation:4} in \cref{prop:coarse functions preserve controlled propagation} below one cannot easily replace admissibility with local admissibility.
\end{remark}

For later reference, it is worthwhile adapting \cref{lem:supports and operations} to the setting of quasi-local operators.
\begin{lemma}\label{lem:supports and operations quasi-local}
  Let $\CHx$ be an admissible $\crse X$-module.
  Given $s,t\in\CB(\CHx)$ that are respectively $\varepsilon_s$\=/$E_s$\=/quasi-local and $\varepsilon_t$\=/$E_t$\=/quasi-local, then 
  \begin{enumerate} [label=(\roman*)]
    \item  $s+t$ is $(\varepsilon_s+\varepsilon_t)$\=/$(E_s\cup E_t)$\=/quasi-local;
    \item  $\supp(s^*)$ is $\varepsilon_s$\=/$\op{E_s}$\=/quasi-local;
    \item  $st$ is $(\varepsilon_s\norm{t}+\varepsilon_t\norm{s})$\=/$(E_s\circ\gauge\circ E_t)$\=/quasi-local, where $\gauge$ is any admissibility gauge for $\CHx$.
  \end{enumerate}
\end{lemma}
\begin{proof}  
  The only non-trivial assertion is the third one.
  By assumption, $E_s, E_t \in \CE$ are such that $\norm{\chf{A'}s\chf A} \leq \varepsilon_s$ and $\norm{\chf{B'}t\chf B} \leq \varepsilon_t$ whenever $A',A$ and $B',B$ are measurable and $E_s$ (resp.\ $E_t$) separated.

  Suppose now that $A'$ and $B$ are measurable and $(E_s\circ\gauge\circ E_t)$\=/separated, and let $C$ be a measurable \gauge-controlled thickening of $E_t(B)$. Since $A'$ and $C$ are still $E_s$\=/separated, and $X\smallsetminus C$ and $B$ are $E_t$\=/separated, it follows from the triangle inequality that
  \begin{align*}
    \norm{\chf{A'}st\chf{B}} 
    \leq \norm{\chf{A'}s(\chfcX{C}t\chf{B})} + \norm{(\chf{A'}s\chf{C})t\chf{B}}
    \leq \varepsilon_t\norm{s}+\varepsilon_s\norm{t},
  \end{align*}
  as we needed to prove.
\end{proof}

\smallskip

For $K$-theoretic reasons (see \cites{higson-2000-analytic-k-hom,roe_index_1996,roe_lectures_2003} and \cref{thm:k-theory-functor} below), it is fundamental to study operators $t \colon \CHx \to \CHy$ that are \emph{locally compact}.
\begin{definition}\label{def:locally-compact-ope}
  An operator $t\colon\CHx\to\CHy$ is \emph{locally compact} (resp.\ has \emph{locally finite rank}) if $\chf Bt$ and $t\chf A$ are compact (resp.\ finite rank) operators for every bounded measurable $A\subseteq X$ and $B\subseteq Y$.
\end{definition}

The motivation behind the following example is deeply rooted in differential geometry and the study of operators on manifolds~\cite{roe_lectures_2003}.
\begin{example}
  Let $\fkA$ be a $\sigma$-algebra and $\mu$ a measure on $(X,\fkA)$, so that $\CH \coloneqq L^2(X,\mu)$ is a $\sigma$-additive $\crse X$-module (recall \cref{conv:sigma additive module}). Let $k\colon X\times X\to \CCC$ be a measurable kernel such that the essential supremum
  \[
    n_k \coloneqq \esssup_{x\in X}\int_X \abs{k(x,y)}^2 d\mu(y) < \infty.
  \]
  It then follows from Cauchy--Schwarz that the integral operator
  \[
    I_k \colon L^2\left(X,\mu\right) \to L^2\left(X,\mu\right), \;\; \text{where} \;\, I_kf\left(x\right) \coloneqq \int_X k\left(x,y\right)f\left(y\right) d\mu\left(y\right)
  \]
  is well-defined and of norm bounded by $n_k$. Observe that $\supp(I_k)$ is contained in the essential support $\esssupp(k)$. In particular, $I_k$ has $E$-controlled propagation if and only if $k$ is essentially supported on $E$ (as a function).
  Given a measurable $E \in \CE$ and $\varepsilon>0$, one may also consider the kernel $\chf E(k)$ (which gives rise to an integral operator of $E$-controlled propagation), and the remainder $k-\chf E(k)$. If $\norm{k-\chf E(k)} \leq \varepsilon$, it follows that $I_k$ is $\varepsilon$\=/$E$\=/approximable.
  Furthermore, if the kernel $k$ is so that the integrals
  \[
    \int_{X\times A}\abs{k\left(x,y\right)}^2d\mu\left(x\right)d\mu\left(y\right)
    \quad\text{and}\quad
    \int_{A\times X}\abs{k\left(x,y\right)}^2d\mu\left(x\right)d\mu\left(y\right)
  \]
  are finite for every bounded measurable $A \subseteq X$, then the operators $\chf AI_k$ and $I_k\chf A$ are Hilbert--Schmidt integral operators, and hence compact. By definition, this means that for such a $k$ the operator $I_k$ is locally compact.
  These integral operators thus yield natural examples of locally compact operators of controlled propagation/approximable operators.
\end{example}

\cref{lem: properness equivalent defs} has a few corollaries that will be useful later on.
\begin{corollary}\label{cor:local compactness and compositions by proper}
  Let $t\colon\CHx\to\CHy$ be a locally compact operator, and let $s\colon\CHy\to\CHz$ and $r\colon\CHx\to\CH_{\crse W}$ be proper.
  \begin{enumerate}[label=(\roman*)]
    \item \label{cor:local compactness and compositions by proper:y} If $\CHy$ is locally admissible then $st$ is locally compact.
    \item \label{cor:local compactness and compositions by proper:x} If $\CHx$ is locally admissible then $tr^*$ is locally compact.
  \end{enumerate}
\end{corollary}
\begin{proof}
  Let us prove \cref{cor:local compactness and compositions by proper:y}, \emph{i.e.}\ $st$ is locally compact. It is trivial that $st\chf A$ is compact when $A\subseteq X$ is measurable and bounded, for $t \chf{A}$ is then compact.
  Now, let $C\subseteq Z$ be bounded and measurable. Since $\CHy$ is locally admissible, \cref{lem: properness equivalent defs} implies that there is a bounded and measurable $B\subseteq Y$ such that $\chf Cs=\chf Cs\chf B$. 
  Therefore $\chf C st= \chf C s\chf Bt$, and the latter is compact because so is $\chf Bt$.

  Assertion~\cref{cor:local compactness and compositions by proper:x} follows from~\cref{cor:local compactness and compositions by proper:y}. In fact, taking the adjoint preserves local compactness and $(tr^*)^* =rt^*$ is locally compact.
\end{proof}

Note that if $s\colon\CHz\to\CHx$ is controlled, then $s^*$ is proper. Observing again that $(ts)^*=s^*t^*$ and adjoints preserve local compactness, we deduce the following.
\begin{corollary}\label{cor:local compactness and compositions by function}
  If $t\colon\CHx\to\CHy$ a locally compact operator, $s\colon\CHz\to\CHx$ is controlled and $\CHx$ is locally admissible, then $ts$ is locally compact.
\end{corollary}

\section{Roe-like algebras of modules} \label{sec:roe-algs-modules}
In this section we use the \emph{controlled propagation} and \emph{locally compact} notions for operators in $\CB(\CHx)$ introduced in \cref{def: controlled propagation operator,def:locally-compact-ope} respectively in order to construct \emph{Roe-like} (\cstar{})algebras associated to a coarse geometric module $\CHx$.
We then prove some general structural results for these \cstar{}algebras.
Heuristically, a coarse geometric module over $\crse X$ lets us define algebras of operators that ``know about the coarse geometry of $\crse X$''. 

\subsection{Definitions and basic properties}
We begin with the following.
\begin{definition}\label{def:C*cp}
  For a fixed $\crse X$-module $\CHx$, the \emph{\Star{}algebra of operators of controlled propagation of $\CHx$} is
  \[
    \cpstar{\CHx} \coloneqq \{t\in \CB\left(\CHx\right)\mid t\text{ has controlled propagation}\}.
  \]
  The norm-closure of $\cpstar{\CHx}$, which we denote by $\cpcstar{\CHx}$, is called the \emph{\cstar{}algebra of operators of controlled propagation of $\CHx$}.
\end{definition}

\begin{remark} \label{rem:fp-alg}
 Braga and Vignati~\cite{braga_gelfand_duality_2022} use the notation $\text{BD}(X)$ for the algebra $\cpcstar{\CHx}$ in the particular case when $X$ is a uniformly locally finite metric space and $\CHx = \ell^2(X) \otimes \ell^2(\NN)$ (\emph{i.e.}~the uniform rank-$\aleph_0$ module of \cref{exmp:uniform module}).
 In such context, operators $t \in \CB(\ell^2(X) \otimes \ell^2(\NN))$ can be understood as $X \times X$-matrices, and they have \emph{finite} (\emph{i.e.}\ controlled) propagation if their non-trivial entries are contained in a band of finite diameter around the diagonal. Norm limits of such operators are classically called \emph{band dominated}, hence the naming $\text{BD}(X)$.
\end{remark}
\begin{remark}
  Observe that \cref{lem:supports and operations} implies that $\cpstar{\CHx}$ and $\cpcstar{\CHx}$ are, respectively, a \Star{}algebra and a \cstar{}algebra. Moreover, note
  \[
    \cpcstar{\CHx} = \{t\in \CB(\CHx)\mid t\text{ is approximable}\}.
  \]
\end{remark}

\begin{exmp}\label{exmp:uniform roe algebra}
  Recall from \cref{exmp:uniform module} that the uniform coarse geometric module of $\crse X$ is simply defined as $\CH_{u,\crse X} \coloneqq \ell^2(X)$. The \cstar{}algebra of operators of controlled propagation $\cpcstar{\CH_{u,\crse X}}$ is classically known as the \emph{uniform Roe algebra} of $\crse X$~\cites{braga_rigid_unif_roe_2022,bbfvw_2023_embeddings_vna,higson-2000-analytic-k-hom,braga_farah_rig_2021}, and denoted by $\uroecstar{\crse X}$. This notion is most often considered when the coarse structure $\CE$ on $X$ is uniformly locally finite, recall \cref{rem:coarse-bdd-vs-ulf}.
\end{exmp}

Similarly to \cref{def:C*cp}, we also give the following.
\begin{definition} \label{def:C*ql-lc}
For a fixed $\crse X$-module $\CHx$, we define:
\begin{enumerate}[label=(\roman*)]
  \item \label{def:C*ql-lc:ql} The \emph{\cstar{}algebra of quasi-local operators} is
    \[
      \qlcstar{\CHx} \coloneqq \{t\in \CB\left(\CHx\right)\mid t\text{ is quasi-local}\}.
    \]
  \item \label{def:C*ql-lc:lc} Likewise, the \emph{\cstar{}algebra of locally compact operators} is
    \[
      \lccstar \CHx\coloneqq\{t\in\CB(\CHx)\mid t\text{ is locally compact}\}.
    \]
\end{enumerate}
\end{definition}

We observe the following.
\begin{lemma} \label{lem:ql-lc-c-star-algs}
  $\lccstar{\CHx}$ is a \cstar{}algebra. If $\CHx$ is admissible then $\qlcstar{\CHx}$ is a \cstar{}algebra as well.
\end{lemma}
\begin{proof}
  The fact that $\lccstar \CHx$ is a \cstar{}algebra follows from the fact that the compact operators $\CK(\CHx)$ are a closed bi-lateral ideal in $\CB(\CHx)$ and some routine computations. For instance, if $\paren{t_\lambda}_{\lambda \in \Lambda}$ is a net of locally compact operators converging to $t \in \CB(\CHx)$, for any bounded $A \subseteq X$ we have that $t_\lambda \chf A \to t \chf A$. Since $t_\lambda \chf A$ are compact operators, so is $t \chf A$. Likewise, it also follows that $\chf A t$ is compact, and thus $t$ is locally compact.

  The proof for $\qlcstar\CHx$ follows immediately from \cref{lem:supports and operations quasi-local}.
\end{proof}

\begin{remark} \label{rmk:approximable is ql 2}
  As noted in \cref{rmk:approximable is ql} it is always the case that $\cpcstar{\CHx}\subseteq\qlcstar{\CHx}$.
  Moreover, quite recently Ozawa~\cite{ozawa-2023} proved this inclusion can be strict. Namely, there are examples where $\qlcstar{\CHx}$ is strictly larger than $\cpcstar{\CHx}$. Thus, we will continue working with both \cstar{}algebras, as they have different advantages. 
  Indeed, $\cpcstar{\CHx}$ is better behaved with respect to coarse-geometric properties, whereas it is easier to verify if an operator is in $\qlcstar{\CHx}$.
\end{remark}

The following combines \cref{def:C*cp,def:C*ql-lc}.
\begin{definition}\label{def:roestar}
  The \emph{Roe \Star{}algebra of $\CHx$} is
  \[
    \roestar{\CHx}\coloneqq \{t\in \CB\left(\CHx\right)\mid t\text{ is locally compact of controlled propagation}\}.
  \]
  The norm-closure of $\roestar{\CHx}$, denoted by $\roecstar{\CHx}$, is called the \emph{Roe \cstar{}algebra of $\CHx$}, or \emph{Roe algebra} for short.
\end{definition}

As $\roestar{\CHx} = \cpstar\CHx\cap\lccstar\CHx$ is an intersection of \Star{}algebras, it is itself a \Star{}algebra. Moreover, note that
\[
 \roecstar\CHx\subseteq\cpcstar\CHx\cap\lccstar\CHx
\]
for all $\crse X$-modules $\CHx$. In fact, later on we will prove this containment is an equality when $\crse X$ is coarsely locally finite and $\CHx$ is discrete (see \cref{thm:roe-alg:intersection}). For now, we observe that $\cpcstar{\CHx}$ idealizes $\roecstar{\CHx}$.
\begin{corollary}\label{cor:cpc is multiplier of roec}
  Let $\crse X $ be a coarse space. If $\CHx$ is a locally admissible $\crse X$-module, then $\cpcstar\CHx $ is contained in the multiplier algebra $\multiplieralg{\roecstar\CHx}$.
\end{corollary}
\begin{proof}
  \cref{cor:local compactness and compositions by function} implies that if $\CHx$ is locally admissible, $t\in \roestar\CHx$ and $s\in\cpstar\CHx$, then the compositions $st$ and $ts$ are again in $\roestar\CHx$ (for the latter consider the adjoint $s^*t^*$).
  The claim follows.
\end{proof}

\begin{notation} \label{notation:roe-like-algs}
  For convenience, we henceforth write
  \begin{enumerate}[label=(\roman*)]
    \item $\roelike\CHx$ to denote either $\roestar{\CHx}$ or $\cpstar{\CHx}$, which we call \emph{Roe-like \Star{}algebras};
    \item $\roeclike \CHx$ to denote either $\roecstar{\CHx}$, $\cpcstar{\CHx}$, or $\qlcstar{\CHx}$, which we call \emph{Roe-like \cstar{}algebras}.
  \end{enumerate}
\end{notation}

\begin{remark}
  One may similarly define other Roe-like (\cstar{})algebras of operators, such as \emph{stable Roe algebras}~\cite{braga_rigid_unif_roe_2022}*{Section~4}, or \emph{uniformly locally-finite-rank Roe} algebras. See also \cite{spakula_rigidity_2013} and references therein for definitions and applications of a few such algebras.
  The crucial features of such algebras are that they consist of (limits of) operators of (quasi) controlled propagation, and yet are sufficiently rich to capture meaningful geometric information of the module. 
  In this work we will only deal with the Roe algebras and the algebras of operators of controlled propagation as in \cref{notation:roe-like-algs}.
\end{remark}

\subsection{Roe-like algebras over (non-connected) spaces} \label{subsec:roe-algs-non-connected}
Here we detail the structure of Roe-like algebras over non (coarsely) connected coarse spaces.
If $\crse X = \bigsqcup_{i \in I} \crse X_i$ is the decomposition of $\crse X$ into its coarsely connected components, it is highly desirable for $X_i$ to be measurable. 
We observe the following, which immediately follows from \cref{def:discrete}~\cref{def:discrete:extension}.

\begin{lemma} \label{lemma:measurable-connected-comps}
  Let $\crse X $ be a coarse space, and let $\crse X=\bigsqcup_{i\in I} \crse X_i$ be the decomposition in coarsely connected components. If $\CHx$ is discrete, then $X_i$ is measurable for all $i \in I$.
\end{lemma}

We now aim to compare the Roe-like (\cstar{})algebras of the $\crse X$-module $\CHx$ with the Roe-like (\cstar{})algebras of the restriction of $\CHx$ to the connected components $\crse X_i$ of $\crse X$ (see \cref{cor:roe algs disconnected}), seen as $\crse X_i$-modules. However, we first introduce the following to ease notation.
\begin{definition} \label{def:unif-contr-family}
Let $\crse X$ be a coarse space, and let $\CHx$ be an admissible $\crse X$-module. We say that a family $\paren{t_\ell}_{\ell \in L} \subseteq \CB(\CHx)$ is
\begin{enumerate}[label=(\roman*)]
  \item \label{def:unif-contr-family:uc} \emph{equi controlled} if $\bigcup_{\ell \in L} \supp(t_\ell) \in \CE$;
  \item \label{def:unif-contr-family:uql} \emph{equi quasi-local} if for all $\varepsilon > 0$ there is some $E \in \CE$ such that $t_\ell$ is $\varepsilon$\=/$E$\=/quasi-local for all $\ell \in L$.
\end{enumerate}
\end{definition}

When the coarsely connected components $(X_i)_{i \in I}$ of $\crse X$ are measurable, we let $\CH_i \coloneqq \CH_{X_i}=\chf{X_i}(\CHx)$. Since every bounded set only intersects one component of $\crse X$, an argument similar to \cref{prop:admissible is discrete} shows that the strong sum $\sum_{i\in I}\chf{X_i}=1$ is the identity operator.
Given an operator $t\in\CB(\CHx)$, let $t_i\coloneqq \chf{X_i}t\chf{X_i}$.
\begin{lemma}\label{lem:controlled iff equicontrolled}
  An operator $t\in\CB(\CHx)$ has controlled propagation (resp.\ is quasi-local) if and only if $\paren{t_i}_{i\in I}$ is equicontrolled (resp.\ equi quasi-local) and $t$ is equal to the strong sum $\sum_{i\in I}t_i$.
\end{lemma}
\begin{proof}
  Fix an operator $t$. If $X_i$ and $X_j$ are different components of $\crse X$ and $t$ is quasi-local (and, \emph{a fortiori}, if it has controlled propagation), then $\norm{\chf{X_i}t\chf{X_j}}$ must be smaller than every $\varepsilon>0$, so it is zero. The fact that $t$ is the strong sum $\sum_{i \in I} t_i$ follows from $\sum_{i\in I}\chf{X_i}=1$. 

  When $t=\sum_{i\in I}t_i$, it is clear that $\supp(t)\subseteq E$ if and only if $\supp(t_i)\subseteq E$ for every $i\in I$ (see \cref{lem:supports and operations}~\cref{lem:supports and operations:sum}). This proves the ``if and only if'' condition for controlled propagation.

  For quasi-locality, let $A',A\subseteq X$ be measurable subsets, and let $A'_i\coloneqq A'\cap X_i$ and $A_i\coloneqq A\cap X_i$. When $t=\sum_{i\in I}t_i$, the operator $\chf{ A'}t\chf A$ is equal to the orthogonal sum
  \(
    \sum_{i\in I}\chf{A'_i}t_i\chf{A_i}.
  \)
  The if and only if condition follows, because the norm of a sum of orthogonal operators is equal to the supremum of their norms, and $A'\cap E(A)=\emptyset$ if and only if $A'_i\cap E(A_i)=\emptyset$ for every $i\in I$.  
\end{proof}

Observe that $\CH_i$ can be seen as either an $\crse X$- or $\crse X_i$-module, and both Roe-like (\cstar{})algebras $\roeclike{\CH_i}$ and $\roeclike{\CH}$ can be seen as subalgebras of $\CB(\CHx)$ (an operator in $\CB(\CH_i)$ can be extended as zero on the orthogonal complement).
Since the spaces $\CH_i$ are pairwise orthogonal, also the product $\prod_{i\in I}\roeclike{\CH}$ is naturally contained in $\CB(\CHx)$. 
Moreover, as local compactness is tested locally, the following is an immediate consequence of \cref{lem:controlled iff equicontrolled}.
\begin{corollary}\label{cor:roe algs disconnected}
  Let $\crse X=\bigsqcup_{i\in I} \crse X_i$ be the decomposition in coarsely connected components of $\crse X$, and suppose $X_i\subseteq X$ is measurable for all $i \in I$. Then 
  \[
    \chf{X_i} \roelike{\CHx} \chf{X_i} = \roelike{\CH_i} 
    \;\; \text{and} \;\; 
    \chf{X_i} \roeclike{\CHx} \chf{X_i} = \roeclike{\CH_i}.
  \]
  Moreover,
  \begin{align*}
    \cpstar{\CHx} =& \left\{\left(t_i\right)_{i \in I} \bigmid\text{equi controlled and bounded}\right\}, \\
    \roestar{\CHx} =& \left\{\left(t_i\right)_{i \in I} \bigmid\text{equi controlled, locally compact, and bounded}\right\}, \\
    \qlcstar{\CHx} =& {\left\{\left(t_i\right)_{i \in I} \bigmid\text{equi quasi-local and bounded}\right\}}.
  \end{align*}
  Therefore,
  \begin{align*}    
    \cpcstar{\CHx} =& \overline{\left\{\left(t_i\right)_{i \in I} \bigmid\text{equi controlled and bounded}\right\}}^{\norm\cdot}, \\
    \roecstar{\CHx} =& \overline{\left\{\left(t_i\right)_{i \in I} \bigmid\text{equi controlled, locally compact, and bounded}\right\}}^{\norm\cdot},
  \end{align*}
  where the terms on the right-hand side are seen as subalgebras of $\prod_{i \in I} \CB\left(\CH_i\right)$.
\end{corollary}

\begin{corollary}\label{cor:roe algs disconnected vs sum and prod}
  With notation as in \cref{cor:roe algs disconnected},
  \[\arraycolsep=0.5ex \def\arraystretch{1.5}
    \begin{array}[]{rcl}
      \bigoplus_{i\in I}^{\rm alg}\roelike{\CH_i} \;\subseteq & \roelike{\CHx} & \subseteq\; \prod_{i\in I}\roelike{\CH_i}, \\
      \bigoplus_{i\in I}\roeclike{\CH_i} \;\subseteq & \roeclike{\CHx} & \subseteq\; \prod_{i\in I}\roeclike{\CH_i},
    \end{array}  
  \]
  where $\bigoplus_{i\in I}^{\rm alg}$ represents the algebraic direct sum (\emph{i.e.}\ consists of finite sums).
\end{corollary}

One key feature of classical Roe algebras is that they contain the ideal of compact operators $\CK(\CHx)$. This is, in fact, a crucial first step towards rigidity, see~\cite{spakula_rigidity_2013}*{Lemma~3.1}.
We observe below that this phenomenon extends to the general setting, although some extra care is needed to deal with coarsely disconnected spaces.
\begin{lemma} \label{lemma:roe-algs-contains-compacts-connected}
  Let $\crse X $ be a coarsely connected coarse space. Then $\CK(\CHx) \subseteq\roeclike\CHx$ for any $\crse X$-module $\CHx$.
\end{lemma}
\begin{proof}
  A rank-one operator on $\CHx$ can always be expressed as $\matrixunit{v'}{v} = \scal{v}{h} \, v'$ for some $v',v\in \CHx$ (cf.\ \cref{subsec:operator-boolalg}). By \cref{lem:supports of vectors almost contained in bounded}, for every $\varepsilon>0$ there are measurable and bounded $B',B\subseteq X$ such that $\norm{v'-\chf{B'}(v')} \leq \varepsilon$ and $\norm{v-\chf{B}(v)} \leq \varepsilon$ (here we are using the coarse connected assumption).
  Letting $w'\coloneqq\chf{B'}(v')$ and $w\coloneqq\chf{B}(v)$ it follows that $\matrixunit{w'}{w}$ is an operator of propagation controlled by $B'\times B$, which is a controlled entourage because $\crse X$ is coarsely connected.
  It is, moreover, routine to verify that $\norm{\matrixunit{v'}{v}-\matrixunit{w'}{w}} \leq \varepsilon(\norm v +\norm w)$. Thus, $\matrixunit{v'}{v} \in \roeclike{\CHx}$, \emph{i.e.}\ every rank-one operator is approximable. Since every compact operator is the norm-limit of finite sums of rank-one operators, the claim follows.
\end{proof}

\begin{corollary}\label{cor: roe algs intersect compacts}
  With notation as in \cref{cor:roe algs disconnected},
  \[
    \roeclike{\CHx}\cap \CK(\CHx)= \bigoplus_{i\in I}\CK(\CH_i).
  \]
\end{corollary}
\begin{proof}
  Observe that $\bigoplus_{i\in I}\CK(\CH_i)$ coincides with the set of compact operators in $\prod_{i\in I}\CB(\CH_i)$.
  Then \cref{cor:roe algs disconnected vs sum and prod} immediately implies that 
  \[
    \roeclike{\CHx}\cap \CK(\CHx)
    \subseteq \Bigparen{\prod_{i\in I}\CB(\CH_i)}\cap\CK(\CHx) 
    = \bigoplus_{i\in I}\CK(\CH_i).
  \]
  The other containment follows from \cref{cor:roe algs disconnected vs sum and prod} applying \cref{lemma:roe-algs-contains-compacts-connected} on each component.
\end{proof}

\begin{remark}
  From a functional-analytic point of view, the two inclusions $\roeclike{\CHx} \subseteq \prod_{i \in I} \roeclike{\CH_i}$ and $\CK\left(\CHx\right) \cap \roeclike{\CHx} \subseteq \prod_{i\in I}\CK(\CH_{\crse X_i})$ in \cref{cor:roe algs disconnected,cor: roe algs intersect compacts} respectively are similar in nature to the inclusion $\ell^\infty(\NN) \otimes \CK \subseteq \ell^\infty(\NN; \CK)$.
  Namely, in the left hand side we first choose an entourage $E \in \CE$ and \emph{then} a tuple $(t_i)_{i \in I}$ such that $\bigcup_{i \in I} \supp(t_i) \subseteq E$. This yields that $(t_i)_{i \in I}$ is equi controlled.
  In contrast, in the right hand sides we only need to choose controlled $(t_i)_{i \in I}$, which may fail to be equi controlled.
\end{remark}

\subsection{Approximate units, and Roe algebras as an intersections} \label{subsec:roe-alg-as-cap}
The goal of this subsection is to express the Roe algebra $\roecstar{\CHx}$ as the intersection of $\cpcstar{\CHx}$ and $\lccstar{\CHx}$ (see \cref{thm:roe-alg:intersection} below). 
This fact has also been observed by Braga and Vignati as a consequence of~\cite{braga_gelfand_duality_2022}*{Proposition~2.1}, which essentially contains its proof.
Nevertheless,~\cite{braga_gelfand_duality_2022}*{Proposition~2.1} only applies to bounded geometry metric spaces. The proof we give applies to arbitrary locally finite coarse spaces, and is therefore more involved.
\begin{theorem}\label{thm:approximate unit}\label{thm:roe-alg:intersection}
  Let $\crse X$ be a coarsely locally finite coarse space, and let $\CHx$ be a discrete $\crse X$-module.
  Then $\cpcstar{\CHx}\cap\lccstar{\CHx}$ has an approximate unit $\paren{p_\lambda}_{\lambda \in \Lambda}$ of projections such that
  \begin{enumerate}[label=(\roman*)]
    \item $(p_\lambda)_{\lambda \in \Lambda}$ is equi controlled (see \cref{def:unif-contr-family}~\cref{def:unif-contr-family:uc}) with propagation controlled by the discreteness gauge;
    \item $p_\lambda$ has locally finite rank for all $\lambda \in \Lambda$.
  \end{enumerate}
  In particular, $p_\lambda\in\roestar{\CHx}$ and hence $\cpcstar{\CHx}\cap\lccstar{\CHx} = \roecstar{\CHx}$.
\end{theorem}
\begin{proof}
  Let $X=\bigsqcup_{i\in I}A_i$ be a discrete partition of $\crse X$, and let $\CH_{A_i} \coloneqq \chf{A_i}(\CHx)$ for every $i\in I$. Let $\Lambda$ be the set of all functions $\lambda$ from $I$ to the set of finite rank vector subspaces of $\CHx$ subordinate to the condition that $\lambda(i)\leq \CH_{A_i}$ for every $i\in I$.
  Let $p_{\lambda(i)}$ be the projection onto $\lambda(i)$ and let $p_\lambda=\sum_{i\in I}p_{\lambda(i)}$ be their strong sum, \emph{i.e.}\ the projection onto the Hilbert space generated by all the $\lambda(i)$.
  We make $\Lambda$ into a directed set in the natural way, namely $\lambda \leq \lambda'$ if and only if $\lambda(i)\leq\lambda'(i)$ for every $i\in I$. Thus, $(p_\lambda)_{\lambda \in \Lambda}$ forms a net in $\CB(\CHx)$.

  Observe that $\supp(p_\lambda)\subseteq \bigsqcup_{i\in I}A_i\times A_i$, so the projections $p_\lambda$ have propagation controlled by the discreteness gauge $\gauge \in \CE$. Hence, they are equi controlled.
  We claim $p_\lambda$ has locally finite rank for every $\lambda \in \Lambda$. In fact, since $\crse X$ is locally finite, for every measurable bounded set $B\subseteq X$ there is a finite subset $J\subseteq I$ such that $B\subseteq \bigsqcup_{j\in J}A_j$. It follows that
  \[
    \chf{B}p_\lambda = \sum_{j\in J}\chf{B}\chf{A_j}p_\lambda =\sum_{j\in J}\chf{B}p_{\lambda(j)}
  \]
  and the latter is a finite sum of finite rank operators. The same argument applies to $p_\lambda\chf{B}$, and this shows that indeed $p_\lambda\in\roestar{\CHx}$.

  It remains to show that $(p_\lambda)_{\lambda \in \Lambda}$ forms an approximate unit of $\cpcstar{\CHx}\cap\lccstar{\CHx}$. Fix a contraction $t \in \cpcstar{\CHx}\cap\lccstar{\CHx}$ and $\varepsilon > 0$. Since $t$ is contained in $\cpcstar{\CHx}$, there is some controlled entourage $E \in \CE$ and a contraction $s \in \cpstar{\CHx}$ such that $\supp(s) \subseteq E$ and $\norm{s - t} \leq \varepsilon$. We may further assume that $E$ is a gauge with $\gauge \subseteq E$.

  Given $A_i, A_j \subseteq X$ (from the discrete decomposition above), we write $A_i \sim_E A_j$ if $A_i$ and $A_j$ are not $E^{\circ k}$\=/separated for any $k \in \NN$, where
  \[
    E^{\circ k} \coloneqq E \circ \dots \circ E \;\; \text{($k$ times)}.
  \]
  Observe that $\sim_E$ is an equivalence relation on $I$ (since $A_i \times A_i \subseteq \gauge \subseteq E$ and $E$ is symmetric).
  Therefore, we may decompose $X = \bigsqcup_{\ell \in L} X_\ell$ into its $\sim_E$\=/equivalence classes, which we may further decompose as $X_\ell = \bigsqcup_{i \in I_\ell} A_i$ for some $I_\ell \subseteq I$.
  
  For every $\ell\in L$, we arbitrarily fix an index $i_\ell \in I_\ell$ and we use it to subdivide each $X_\ell$ into concentric anuli. 
  Namely, we let  $C_0^\ell \coloneqq A_{i_\ell}$, and for every $n \geq 1$ we let $C_n^\ell$ be a measurable controlled thickening of $E^{\circ (2n)}(C_{n-1}^\ell)$. We then let
  \[
    B_n^\ell \coloneqq C_n^\ell \setminus C_{n-1}^\ell
    \quad\text{and}\quad 
    B_n \coloneqq \bigsqcup_{\ell \in L} B_n^\ell.
  \]
  Observe that the sets $B_n$ form a measurable (non-controlled) partition of $X$ with the property that $E(B_m) \cap E(B_n) = \emptyset$ when $\abs{n - m} \geq 2$.

  Since $B_n^\ell$ is bounded and $t$ is quasi-local, $\chf{E(B_n^\ell)}t \chf{B_n^\ell}$ is compact for all $\ell \in L$ and $n \in \NN$. In particular, we may choose a $\lambda^\ell_n\in\Lambda$ which is supported on the set indices $i\in I_\ell$ such that $A_i\cap E(B^\ell_n)\neq\emptyset$, and it satisfies
  \[
    \norm{\chf{E(B_n^\ell)}t \chf{B_n^\ell} - p_{\lambda_n^\ell} \chf{E(B_n^\ell)} t \chf{B_n^\ell}} \leq \varepsilon.
  \]
  Observe that for every $i\in I$ there are at most finitely many values $n\in\NN$ such that $A_i\cap E(B^\ell_n)$ is non-empty (they must be consecutive). It follows that the $\lambda^\ell_n$ can be joined to obtain some $\lambda\in\Lambda$ such that
  \[
    \norm{\chf{E(B_n^\ell)}t \chf{B_n^\ell} - p_{\lambda} \chf{E(B_n^\ell)} t \chf{B_n^\ell}} \leq \varepsilon.
  \]
  for every $\ell\in L$ and $n\in\NN$.

  Let now
  \[
    B_{\rm even}\coloneqq \bigsqcup_{n\in\NN}B_{2n}
    \quad\text{and}\quad
    B_{\rm odd}\coloneqq \bigsqcup_{n\in\NN}B_{2n+1}.
  \]
  Observe that $X = B_{\rm even}\sqcup B_{\rm odd}$, and therefore $t=t\chf{B_{\rm even}}+t\chf{B_{\rm odd}}$.  
  We will now focus on $B_{\rm even}$. Both 
  \[
    \norm{t\chf{B_{\rm even}} - s\chf{B_{\rm even}}}
    \quad\text{and}\quad
    \norm{p_\lambda t\chf{B_{\rm even}} - p_\lambda s\chf{B_{\rm even}}}
  \] are bounded by $\norm{t-s}$, which we assumed to be $\leq \varepsilon$.
  Since $\supp(s) \subseteq E$, we also have
  \[
    s\chf{B_{\rm even}} = \chf{E(B_{\rm even})} s\chf{B_{\rm even}}.
  \]
  It then follows that
  \begin{align*}
    p_\lambda t \chf{B_{\rm even}}
    &\approx_\varepsilon p_\lambda s \chf{B_{\rm even}}
    = p_\lambda \chf{E(B_{\rm even})} s \chf{B_{\rm even}}
    \approx_\varepsilon p_\lambda \chf{E(B_{\rm even})} t \chf{B_{\rm even}}; \\
    t \chf{B_{\rm even}}
    &\approx_\varepsilon s \chf{B_{\rm even}}
    = \chf{E(B_{\rm even})} s \chf{B_{\rm even}}
    \approx_\varepsilon \chf{E(B_{\rm even})} t \chf{B_{\rm even}}.
  \end{align*}
  It follows from the construction that
  \begin{align*}
    \paren{\chf{E(B_{\rm even})} t \chf{B_{\rm even}}} & - \paren{p_\lambda \chf{E(B_{\rm even})} t \chf{B_{\rm even}}} \\
    & = \sum_{\ell \in L}\sum_{n\in \NN}\paren{\chf{E(B_{2n}^\ell)}t \chf{B_{2n}^\ell}} - \paren{p_\lambda\chf{E(B_{2n}^\ell)}t \chf{B_{2n}^\ell}}
  \end{align*}
  is a sum of orthogonal operators, and therefore its norm is equal to
  \[
    \sup_{\ell \in L} \sup_{n \in \NN} \norm{\paren{\chf{E(B_{2n}^\ell)}t \chf{B_{2n}^\ell}} - \paren{p_\lambda\chf{E(B_{2n}^\ell)}t \chf{B_{2n}^\ell}}}\leq \varepsilon.
  \]
  Combining these inequalities yields that
  \[
   \norm{t \chf{B_{\rm even}} - p_\lambda t \chf{B_{\rm even}} }\leq 5\varepsilon,
  \]
  and a symmetric argument holds for the odd part as well. Since $\varepsilon > 0$ was arbitrary, this shows that $(p_\lambda)_{\lambda \in \Lambda}$ is a left approximate unit. Since each $p_\lambda$ is self-adjoint, it is also a right approximate unit.

  The ``in particular'' part of the statement is now trivial. Indeed, every operator $t \in \cpcstar\CHx\cap\lccstar{\CHx}$ is the norm-limit of the net $p_\lambda t$, which is contained in $\roecstar\CHx$ because 
  $p_\lambda \in \roecstar\CHx$ and $\cpcstar\CHx$ idealizes $\roecstar\CHx$ (see \cref{cor:cpc is multiplier of roec}).
\end{proof}

\cref{thm:roe-alg:intersection} is rather convenient, as there are instances where it is simpler to prove that an operator belongs to the Roe algebra by separately checking that it is the limit of operators of controlled propagation and that it is locally compact.
One instance where this is (implicitly) used is in the context of \emph{warped cones} (we refer, \emph{e.g.}\ to \cites{roe_warped_2005,li_markovian_2020} and references therein for notation and motivation).

In \cite{li_markovian_2020}*{Proposition 5.1} it is proved ``by hand'' that certain projections of the form $P\otimes 1_{\ell^2(\NN)}$ belong to the (classical) Roe algebra $\roecstar{\CO_\Gamma X}$ of the (\emph{unified}) \emph{warped cone} $\CO_\Gamma X$ by constructing approximations by locally finite rank operators of controlled propagation (this extends a previous result of Sawicki \cite{SawickiThesis}).\footnote{\,
The fact that the projections $P\otimes 1_{\ell^2(\NN)}$ belong to the Roe algebra has consequences regarding the failure of surjectivity of the coarse assembly map of the coarse Baum--Connes conjecture.}
This statement would have been an immediate consequence of \cref{thm:roe-alg:intersection}. In fact, the projections $P$ have rank-1, so $P\otimes 1_{\ell^2(\NN)}$ is obviously locally compact, and one of the main results of \cite{li_markovian_2020} directly implies that  $P\otimes 1_{\ell^2(\NN)}$ is a limit of certain operators of finite propagation (which are, however, never locally compact).

Using \cite{delaat2023dynamical}, \cref{thm:roe-alg:intersection} lets us prove a much more refined result than that of \cite{li_markovian_2020}. The \cstar{}algebra of operators of \emph{finite dynamical propagation} $C^*_{\rm fp}(\Gamma\curvearrowright X) \leq \CB(L^2(X))$, where $\Gamma\curvearrowright (X,\mu)$ is a measurable action on a measure space, is defined in \cite{delaat2023dynamical}.
It is shown in~\cite{delaat2023dynamical}*{Corollary~C} that if the $\Gamma$-action is ergodic then $C^*_{\rm fp}(\Gamma\curvearrowright X)\cap \CK(L^2(X))$ is either $\{0\}$ or the whole of $\CK(L^2(X))$ (and the latter case happens if and only if the action is strongly ergodic).
If $X$ is also a compact metric space so that $\mu$ is a $\sigma$-finite Borel measure of full support, taking tensor products defines a natural \Star{}embedding of the spatial (a.k.a.\ minimal) tensor product
\[
  C^*_{\rm fp}(\Gamma\curvearrowright X)\otimes \cpcstar{L^2(\RR_{\geq 0})}\to\cpcstar{\CO_\Gamma X}.
\]
It is clear by the definition of the warped metric that $\CK(L^2(X))\otimes \lccstar{L^2(\RR_{\geq 0})}$ is mapped into $\lccstar{\CO_\Gamma X}$. It then follows from \cref{thm:roe-alg:intersection} that the tensor product restricts to a \Star{}embedding
\[
  \paren{C^*_{\rm fp}(\Gamma\curvearrowright X)\cap\CK(L^2(X))}\otimes \roecstar{L^2(\RR_{\geq 0})}
  \to\roecstar{\CO_\Gamma X}.
\]
In particular, if the action is strongly ergodic this induces a \Star{}embedding of $\CK(L^2(X))\otimes \roecstar{L^2(\RR_{\geq 0})}$ into $\roecstar{\CO_\Gamma X}$. Observe that \cite{li_markovian_2020}*{Proposition 5.1} follows, as $P\in\CK(L^2(X))$.

We end this subsection by observing that in many cases the approximate unit constructed in \cref{thm:approximate unit} is an approximate unit for much larger \cstar{}algebras.
\begin{proposition}\label{prop: approx unit of lccstar}
  Let $\crse X$ be a coarsely locally finite and coarsely countable space, and let $\CHx$ be a discrete $\crse X$-module. Then the net $\paren{p_\lambda}_{l\in \Lambda}\subseteq \roestar{\CHx}$ constructed in \cref{thm:approximate unit} is an approximate unit of $\lccstar{\CHx}$.
\end{proposition}
\begin{proof}
  Let $X=\bigsqcup_{i\in \NN}A_i$ be a \emph{countable} discrete partition.
  Fix a contraction $t \in \lccstar\CHx$ and $\varepsilon>0$. For every $i\in I$, the operator $\chf{A_i}t$ is compact and has range contained in $\CH_{A_i}$. It follows that we may choose $\lambda(i)\leq\CH_{A_i}$ such that
  \[
    \norm{p_{\lambda(i)}t -\chf{A_i}t}= \norm{p_{\lambda\left(i\right)}\chf{A_i}t -\chf{A_i}t}\leq \frac{\varepsilon}{2^i}
  \]
  (here we are using in a key manner that $\crse X$ is coarsely countable). Hence,
  \begin{align*}
    \norm{p_\lambda t - t} = \norm{\sum_{i \in \NN} p_{\lambda\left(i\right)} t - \chf{A_i} t} \leq \sum_{i \in \NN} \norm{p_{\lambda\left(i\right)} t - \chf{A_i} t} \leq \varepsilon.
  \end{align*}
  Taking $\varepsilon \to 0$ yields that $(p_\lambda)_{\lambda \in \Lambda}$ is an approximate unit.
\end{proof}

\subsection{(Non-commutative) Cartan subalgebras of Roe-like algebras} \label{subsec:cartan subalgebras}
In this subsection we continue the study of some structural properties of $\roeclike{\CHx}$. Specifically, we construct non-commutative Cartan subalgebras in Roe-like algebras.
We refer the reader to~\cites{feldman-moore-1997-i1,feldman-moore-1997-ii} for the classical theory of Cartan subalgebras for von Neumann algebras,\footnote{\,
  Cartan subalgebras for von Neumann algebras are closer in spirit to \emph{\cstar{}diagonals} in the sense of~\cite{kumjian-diagonals-1986}. This is a special kind of (\cstar{})Cartan subalgebra, which were introduced in \cite{Renault2008CartanSI}.
  }
and to \cites{kumjian-diagonals-1986,Renault2008CartanSI} for its analogue in \cstar{}algebras. We also note that \cite{white_cartan_2018} discusses the existence and uniqueness of Cartan subalgebras of the classical uniform Roe algebras of metric spaces, and that these constructions have been extended to the non-uniform case in \cite{braga_gelfand_duality_2022}.
For now, let us recall the definition of a (\cstar{})\emph{Cartan subalgebra}.
\begin{definition} \label{def:cartan-subalg}
  Let $B \subseteq A$ be an inclusion of \cstar{}algebras. We say that $B$ is a \emph{(\cstar{})Cartan subalgebra} of $A$ if the following hold.
  \begin{enumerate}[label=(\alph*)]
    \item \label{def:cartan-subgal:max-ab} $B$ is maximal abelian in $A$;
    \item \label{def:cartan-subgal:app-unit} $B$ contains an approximate unit of $A$;
    \item \label{def:cartan-subgal:norm} $A$ is generated (as a \cstar{}algebra) by the \emph{normalizers} of $B$, that is, those $n \in A$ such that $nBn^*, n^*Bn \subseteq B$;
    \item \label{def:cartan-subgal:cond-exp} There is a faithful conditional expectation $E \colon A \to B$.
  \end{enumerate}
  Since we only work with \cstar{}algebras we shall drop the \cstar{} and simply call these pairs \emph{Cartan subalgebras}.
\end{definition}

Note that \cref{def:cartan-subalg}~\cref{def:cartan-subgal:max-ab} implies, in particular, that $B$ is abelian. 
Such a Cartan pair can be usually found for Roe algebras $\roecstar{\CHx}$ (see \cref{rmk:commutative Cartan for Roe}) or Roe-like algebras of the uniform coarse geometric module $\CH_{u,\crse X}$ of \cref{exmp:uniform module}.
However, this cannot be arranged in the setting of general Roe-like algebras (see \cref{rmk:commutative Cartan for Roe}).
In the general setup it is instead more natural to work with \emph{non-commutative} Cartan pairs, as introduced in~\cite{exel-nc-cartan-subalgs-2011}.
\begin{definition}[cf.~\cite{exel-nc-cartan-subalgs-2011}*{Definition~9.2}]
  Let $B \subseteq A$ be an inclusion of \cstar{}algebras. A \emph{virtual commutant of $B \subseteq A$} is a pair $(\varphi, J)$, where $J \subseteq B$ is a (closed $2$-sided) ideal, and $\varphi \colon J \to A$ is a linear map such that $\varphi(bx) = b \varphi(x)$ and $\varphi(xb) = \varphi(x) b$ for all $x \in J$ and $b \in B$.
\end{definition}

\begin{definition}[cf.~\cite{exel-nc-cartan-subalgs-2011}*{Definition~12.1}] \label{def:nc-cartan-subalg}
  Let $B \subseteq A$ be an inclusion of \cstar{}algebras. We say that $B$ is a \emph{non-commutative Cartan subalgebra} of $A$ if \cref{def:cartan-subgal:app-unit,def:cartan-subgal:norm,def:cartan-subgal:cond-exp} in \cref{def:cartan-subalg} hold, and it moreover satisfies
  \begin{enumerate}[label=(\alph*')]
    \item \label{def:cartan-subgal:virtual-comm} $\image(\varphi) \subseteq B$ for any virtual commutant $(\varphi, J)$ of $B \subseteq A$.
  \end{enumerate}
\end{definition}
It is proven in~\cite{exel-nc-cartan-subalgs-2011}*{Propositions~9.3 and~9.4} that every virtual commutant is bounded and its image is contained in $JAJ$. Likewise,~\cite{exel-nc-cartan-subalgs-2011}*{Proposition~9.7} shows that if~\cref{def:cartan-subgal:virtual-comm} holds then $B' \cap A \subseteq B$, \emph{i.e.}\ $B$ contains all the elements of $A$ that commute with it.
Moreover, \cite{exel-nc-cartan-subalgs-2011}*{Proposition~9.8} proves that if $B$ is assumed to be abelian then the converse also holds. It follows that conditions~\cref{def:cartan-subgal:max-ab,def:cartan-subgal:virtual-comm} are equivalent if $B$ is abelian, and hence \cref{def:nc-cartan-subalg} is a natural non-commutative generalization of \cref{def:cartan-subalg}. We record the following for future use.
\begin{theorem}[cf.~\cite{exel-nc-cartan-subalgs-2011}*{Theorem~9.5}] \label{thm:virtual-comm}
  Suppose $A \subseteq \CB(\CH)$ is concretely represented, and let $B \subseteq A$ be a non-commutative Cartan subalgebra.
  Then, for any virtual commutant $(\varphi, J)$ of $B \subseteq A$ there is some $t \in \CB(\CH) \cap B'$ such that $\varphi(x) = tx \in A$ for all $x \in J$.
\end{theorem}

We turn now to the construction of non-commutative Cartan subalgebras in Roe-like algebras (see \cref{thm:cartan-subalgs}). The following sets notation that will be useful in the sequel.
\begin{notation}
  In the following, we denote by $\discdec \coloneqq (A_i)_{i \in I}$ a discrete partition $X = \bigsqcup_{i \in I} A_i$ for a discrete $\crse X$-module $\CHx$.
\end{notation}

As in \cref{subsec:roe-algs-non-connected}, once the partition $\discdec$ is fixed, for every operator $t\in\CB(\CHx)$ we denote the section $\chf{A_i}t\chf{A_i}$ by $t_i$.
\begin{definition} \label{def:cond-exp}
  Let $\crse X$ be a coarse space, and let $\CHx$ be a discrete module. Let $\discdec = (A_i)_{i \in I}$ be a discrete partition of $\CHx$. We define
  \[
    \begin{tikzcd}[row sep =0]
      \condexp{\discdec} \colon \CB\left(\CHx\right) \arrow{r} &  \prod_{i \in I} \CB\left(\CH_{A_i}\right), \\
    \hspace{1.5 em} t \arrow[|->, r] & (t_i)_{i \in I}.
    \end{tikzcd}    
  \]
  Moreover, let $\condexpql{\discdec}, \condexpcp{\discdec}$ and $\condexproe{\discdec}$ be the restrictions of $\condexp{\discdec}$ to $\qlcstar{\CHx}, \cpcstar{\CHx}$ and $\roecstar{\CHx}$ respectively. 
\end{definition}

Henceforth, we will use the convention that $\condexproelike{\discdec}$ denotes one of $\condexpql{\discdec}, \condexpcp{\discdec}$ and $\condexproe{\discdec}$. We also define $\cartansubalgroelike{\discdec}{\CHx}$ to be the image of $\condexproelike{\discdec}$.
Observe that
\begin{equation}\label{eq:cartan algs I}
    \cartansubalgcp{\discdec}{\CHx} = \cartansubalgql{\discdec}{\CHx} = \prod_{i\in I}\CB(\CH_{A_i}).
\end{equation}
Moreover, if $\fka $ is locally finite we also have
\begin{equation}\label{eq:cartan algs II}
    \cartansubalgroe{\discdec}{\CHx} = \prod_{i\in I}\CK(\CH_{A_i}).
\end{equation}

\begin{lemma} \label{lemma:cond-exp}
  For any discrete partition $\discdec = (A_i)_{i \in I}$ of $\crse X$, the maps $\condexproelike{\discdec} \colon \roeclike{\CHx} \to \cartansubalgroelike{\discdec}{\CHx}$ are faithful conditional expectations.
\end{lemma}
\begin{proof}
  Observe that the map $\condexp{\discdec}$ itself is completely positive and contractive, and hence so is $\condexproelike{\discdec}$.
  Moreover, since $\cartansubalgroelike{\discdec}{\CHx} \subseteq \prod_{i \in I} \CB(\chf{A_i}(\CHx))$, it follows that $\condexproelike{\discdec}$ restricts to the identity on $\cartansubalgroelike{\discdec}{\CHx}$. Hence, by Tomiyama's Theorem, it follows that $\condexproelike{\discdec}$ is a conditional expectation.
  Alternatively, it is just as easy to directly verify that $\condexproelike{\discdec}\paren{xty}=x\condexproelike{\discdec}\paren{t}y$ for every $t\in\roeclike{\CHx}$ and $x,y\in\cartansubalgroelike{\discdec}{\CHx}$.

  Therefore, it suffices to prove that $\condexproelike{\discdec}$ is faithful. Let $x \in \roeclike{\CHx}$ be such that $\condexproelike{\discdec}(x^*x) = 0$. Then $0 = \chf{A_i} x^*x \chf{A_i} = (x \chf{A_i})^*(x \chf{A_i})$ for all $i \in I$. Thus $x \chf{A_i} = 0$. As $\sum_{i \in I} \chf{A_i} = 1$, the claim follows.
\end{proof}

\begin{lemma} \label{lemma:cartan:commutant}
  If $\CR$ is ${\rm cp}$ or ${\rm ql}$, or $\CR$ is ${\rm Roe}$ and $\fka$ is locally finite, then the commutant of $\cartansubalgroelike{\discdec}{\CHx}$ is $\prod_{i\in I}\paren{\CCC\cdot 1_{\CH_{A_i}}}$.
\end{lemma}
\begin{proof}
  For every $i\in I$, the compacts $\CK(\CH_{A_i})$ are naturally contained in $\cartansubalgroelike{\discdec}{\CHx}$. Therefore
  \[
    \paren{\cartansubalgroelike{\discdec}{\CHx}}'
    \subseteq (\CK(\CH_{A_i}))'
    = \CCC\cdot 1_{\CH_{A_i}}\times\CB(\CH_{X\smallsetminus A_i}).
  \]
  It follows that $\paren{\cartansubalgroelike{\discdec}{\CHx}}'$ is contained in
  \[
    \bigcap_{i\in I}\bigparen{\CCC\cdot 1_{\CH_{A_i}}\times\CB(\CH_{X\smallsetminus A_i})}
    = \prod_{i\in I}\paren{\CCC\cdot 1_{\CH_{A_i}}}.
  \]
  The converse containment is obvious.
\end{proof}

Recall that $\crse X $ has bounded geometry (in the sense of \cref{def:locally finite space - bounded geo}) if it admits a uniformly locally finite controlled partition. Moreover, if $\CHx$ is a discrete module, \cref{cor:discrete and bounded geometry} implies that $\crse X$ admits a uniformly locally finite partition.
The following is a rather useful standard trick.
\begin{lemma}\label{lem: bded geo partitioning controlled sets}
    Let $\discdec = (A_i)_{i\in I}$ be a uniformly locally finite partition. Then for every $E\in\CE$ one may find finitely many partially defined injective functions $f_0,\ldots,f_n\colon I\to I$ such that
    \[
      E\subseteq \bigsqcup_{k=0}^n\bigsqcup_{j\in J_k} \paren{A_{f(j)}\times A_j},
    \]
    where $J_k\subseteq I$ denotes the domain of definition of $f_k$.
\end{lemma}
\begin{proof}
  Consider the graph $\CG$ whose vertices are the pairs $(i,j)\in I^2$ such that $(A_i\times A_j)\cap E\neq \emptyset$, and whose edges are the pairs sharing one of the two coordinates. Since $\discdec$ is uniformly locally finite, this graph has bounded degree.
  Let $n$ be the degree of $\CG$. We may then choose a coloring of the vertices of $\CG$ using the $n+1$ colours $\{0,\ldots,n\}$, so that any two adjacent vertices have different colours.
  By construction, the set of vertices of colour $k$ is the graph of an injective partially defined function $f_k$. The functions thus defined satisfy the requirements of the lemma.
\end{proof}

\begin{corollary}\label{cor:bded geo decomposition of operators}
  If $\discdec = (A_i)_{i\in I}$ is a uniformly locally finite discrete partition then every controlled  propagation operator $t$ can be written as a finite sum $t= t^{(0)}+\cdots+t^{(n)}$ with the property that for every $i\in I$ there is at most one $j\in I$ and one $j'\in I$ with
  \[
    \chf{A_j}t^{(k)}\chf{A_i}\neq 0
    \quad{ and/or }\quad
    \chf{A_i}t^{(k)}\chf{A_{j'}}\neq 0.
  \]
\end{corollary}
\begin{proof}
  Say $\supp(t)\subseteq E$. Apply \cref{lem: bded geo partitioning controlled sets} to $E$, and let $f_0,\ldots,f_n$ be the resulting injections. For each $k$, let $J_k$ be the domain of $f_k$ and define
  \[
    t^{(k)}\coloneqq \sum_{j\in J_k}\chf{A_{f(j)}}t\chf{A_j}.
  \]
  Since $\supp(t)\subseteq E$, it follows that $t= t^{(0)}+\cdots+t^{(n)}$. Moreover, these $t^{(0)}, \dots, t^{(n)}$ are as claimed, since $f_0, \dots, f_n$ are injective.
\end{proof}

We are now ready to prove the main theorem of the subsection (note that the statement does not apply to the algebra of quasi-local operators, see \cref{rem:cartan-for-ql}).
\begin{theorem} \label{thm:cartan-subalgs}
  Let $\crse X $ be a coarse space of bounded geometry, $\CHx$ a discrete $\crse X$-module, and $\discdec = (A_i)_{i\in I}$ a uniformly locally finite discrete partition. If $\roeclike{\CHx}$ is either $\cpcstar\CHx$ or $\roecstar\CHx$, then the inclusion
  \[
    \cartansubalgroelike{\discdec}{\CHx}\subseteq \roeclike{\CHx}
  \]
  forms a non-commutative Cartan pair.
\end{theorem}
\begin{proof}
  We divide the proof into parts, each one containing one of the conditions defining a (non-commutative) Cartan pair in \cref{def:nc-cartan-subalg}.

  \smallskip

  {\cref{def:cartan-subgal:virtual-comm}}.
  By \cref{thm:virtual-comm}, for any virtual commutant $(\varphi, J)$ of $\cartansubalgroelike{\discdec}{\CHx}$ there is some $t$ in the commutant of $\cartansubalgroelike{\discdec}{\CHx}$ such that $\varphi(x) = tx \in \roeclike{\CHx}$ for all $x \in J$.
  \cref{lemma:cartan:commutant} implies in particular that $t$ belongs to $\cartansubalgroelike{\discdec}{\CHx}$.
  Since $J$ is an ideal, it follows that $\phi(x)=tx\in J$ for every $x\in J$.

  {\cref{def:cartan-subgal:app-unit}}. The inclusion $\cartansubalgcp{\discdec}{\CHx} \subseteq \cpcstar{\CHx}$ is unital, so there is nothing to prove. The proof of \cref{thm:approximate unit} shows that $\cartansubalgroe{\discdec}{\CHx}$ contains an approximate unit of $\roecstar{\CHx}$.

  {\cref{def:cartan-subgal:norm}}. By definition, every $t\in \roeclike{\CHx}$ is arbitrarily well approximated by an operator $s$ of controlled propagation. It is then enough to apply \cref{cor:bded geo decomposition of operators} to decompose $s$ as a finite sum $s^{(0)}+\cdots+s^{(n)}$ and observe that each $s^{(k)}$ is a normalizer.

  {\cref{def:cartan-subgal:cond-exp}}. The conditional expectations described in \cref{def:cond-exp} are faithful by \cref{lemma:cond-exp}.  
\end{proof}

\begin{remark} \label{rem:cartan-for-ql}
  The only part of the proof of \cref{thm:cartan-subalgs} that does not apply to $\qlcstar{\CHx}$ is \cref{def:cartan-subgal:norm}.
\end{remark}

We end the discussion on Cartan subalgebras with some consequences of \cref{thm:cartan-subalgs}, along with some remarks.
\begin{corollary} \label{cor:roe-models}
  Let $\crse X$ be a coarse space with bounded geometry, $\CHx$ a discrete $\crse X$-module, $\discdec = (A_i)_{i \in I}$ a uniformly locally finite discrete partition, and $\roeclike{\CHx}$ denote either $\cpcstar\CHx$ or $\roecstar{\CHx}$.
  Then there are:
  \begin{enumerate}[label=(\roman*)]
    \item an inverse semigroup $S_{\CR}(\CHx, \discdec)$;
    \item an action $\alpha_{\CR}$ of $S_{\CR}(\CHx, \discdec)$ on $\cartansubalgroelike{\discdec}{\CHx}$;
    \item an isomorphism $\Phi^\discdec \colon \roeclike{\CHx} \to \cartansubalgroelike{\discdec}{\CHx} \rtimes_{\rm red} S_{\CR}(\CHx, \discdec)$ such that $\Phi^\discdec$ restricted to $\cartansubalgroelike{\discdec}{\CHx}$ is the identity map.
  \end{enumerate}
  In particular, if the action $\alpha_{\CR}$ has the approximation property (in the sense of \cite{buss-martinez-approx-prop-23}*{Definition~4.4}) and $\paren{\chf{A_i}}_{i \in I}$ are of uniformly bounded rank, then $\roeclike{\CHx}$ is nuclear.
\end{corollary}
\begin{proof}
  This is an immediate consequence of \cref{thm:cartan-subalgs}, \cite{exel-nc-cartan-subalgs-2011}*{Theorem~14.5} and \cite{buss-martinez-approx-prop-23}*{Theorem~6.7}.
\end{proof}

On any group $G$ (equipped with the discrete topology) one may consider the \emph{minimal connected left-invariant coarse structure} $\varcrs[left]{fin}$, whose controlled entourages are contained in the entourages of the form
\[
  E_A \coloneqq \braces{\left(g,h\right)\mid h^{-1}g \in A},
\]
where $A$ ranges among all finite subsets of $G$ (see \cite{coarse_groups} and references therein for more information on left-invariant coarse structures). 
If $G=\Gamma$ is a discrete \emph{countable} group, $\varcrs[left]{fin}$ coincides with the metric coarse structure associated with any proper left-invariant metric on $\Gamma$.
Observe $(G,\varcrs[left]{fin})$ is always uniformly locally finite.
An immediate corollary of the last statement of \cref{cor:roe-models} is the following result, which is well-known for discrete countable groups.
\begin{corollary} \label{cor:group-exact}
  Let $\crse X = (G,\varcrs[left]{fin})$ and let $\CH_G$ be a discrete $\crse X$-module such that $\paren{\chf{g}}_{g \in G}$ are of uniformly bounded rank. Then, $\cpcstar{\CH_G}$ is nuclear if and only if $G$ is exact.
\end{corollary}

\begin{remark}
  The condition that $\paren{\chf{A_i}}_{i \in I}$ have uniformly bounded rank is rather restrictive.
  Indeed, when it holds one may always pass to a coarsely equivalent space $\crse X'$ so that $\CHx$ can be identified with the uniform coarse module $\CH_{u,\crse X'}$.
  In particular, $\cpcstar\CHx=\roecstar\CHx=\uroecstar{\crse X'}$ (see \cref{exmp:uniform roe algebra}).
    
  However, this is a necessary condition, as otherwise $\cartansubalgroelike{\discdec}{\CHx}$ would not be exact since it would contain a copy of $\prod_{n \in \NN} M_{k(n)}$ for some increasing sequence $\{k(n)\}_{n \in \NN}$.
  In such case, $\roeclike{\CHx}$ cannot be nuclear.\footnote{\, Recall that every nuclear \cstar{}algebra is automatically exact by the work of Kirchberg, and exactness passes to \cstar{}sub-algebras~\cite{brown_c*-algebras_2008}.}
  Note that this condition is, indeed, needed in \cite{buss-martinez-approx-prop-23}*{Theorem~6.7}.
\end{remark}

\begin{remark}
  With the same assumptions and notation as in \cref{cor:roe-models}, it is worth mentioning that $S_{\CR}(\CHx, \discdec)$ may be chosen to be the set of \emph{slices} of the inclusion $\cartansubalgroelike{\discdec}{\CHx} \subseteq \roeclike{\CHx}$~\cite{exel-nc-cartan-subalgs-2011}*{Definition~10.1}. That is, the closed linear subspaces $M$ of the set of normalizers of the inclusion such that $Mt, tM \subseteq M$ for all $t \in \cartansubalgroelike{\discdec}{\CHx}$. 
  It is then not hard to show that $S_{\rm cp}(\CHx, \discdec)$ can be identified with the set of injective controlled partial maps $I \to I$.
  In other words, if $\crse I$ is the coarse space coarsely equivalent to $\crse X$ constructed from $\discdec = (A_i)_{i \in I}$ in the usual manner, the semigroup $S_{\rm cp}(\CHx, \discdec)$ coincides with the pseudogroup $\CG(\crse I)$ as defined in~\cite{skandalis-tu-yu-coarse-gpds-02}*{Definition~3.1}. Observe that the \emph{groupoid of germs} of the natural action of $S_{\rm cp}(\CHx, \discdec)$ on $I$ yields the so-called \emph{coarse groupoid} introduced in~\cite{skandalis-tu-yu-coarse-gpds-02}*{Section~3.2}.
\end{remark}

\begin{remark}\label{rmk:commutative Cartan for Roe}
  If we choose for every $i\in I$ an orthonormal basis $\paren{v_{i,\ell}}_{\ell \in L_i} $ of $\CH_{A_i}$, then $\ell^\infty(\paren{v_{i,\ell}}_{i, \ell}) \cap \roecstar{\CHx}$ is a commutative Cartan subalgebra of $\roecstar{\CHx}$.
  In fact, note that $\ell^\infty(\paren{v_{i,\ell}}_{i, \ell}) \cap \roecstar{\CHx}$ is abstractly isomorphic to $\prod_{i\in I}c_0(L_i)$, where the $L_i$ are treated as discrete sets. 
  The conditional expectation
  \[
    E' \colon \roecstar{\CHx} \to \ell^\infty(\paren{v_{i,\ell}}_{i, \ell}) \cap \roecstar{\CHx}
  \]
  is the composition $E' \coloneqq E_2 \circ \condexp{\discdec}$, where $E_2$ is the coordinate-wise sum of the usual conditional expectations $\CK(\chf{A_i}(\CHx)) \to c_0(\paren{v_{i,\ell}}_{\ell \in L_i}) \cong c_0(L_i)$.
  This conditional expectation $E'$ meets the requirements of \cref{def:cartan-subalg}~\cref{def:cartan-subgal:cond-exp}, while \cref{def:cartan-subgal:max-ab,def:cartan-subgal:app-unit,def:cartan-subgal:norm} are proven in a similar manner to the previous discussion.
  
  In contrast, we also observe that, unless the projections $\chf{A_i}$ have uniformly bounded rank, the algebra $\ell^\infty(\paren{v_{i,\ell}}_{i,\ell}) = \ell^\infty(\paren{v_{i,\ell}}_{i,\ell})\cap \cpcstar\CHx$ is \emph{not} a Cartan subalgebra of $\cpcstar\CHx$.
  For instance, if $\CH_{A_i}$ is infinite dimensional then $c_0(\paren{v_{i,\ell}}_{\ell \in L_i})<\CB(\CH_{A_i})\subseteq \cpcstar\CHx$ is not a Cartan subalgebra, as the inclusion does not admit enough normalizers to generate $\cpcstar{\CHx}$.
\end{remark}

\section{Mappings of operators}
After having discussed spaces, modules on spaces, operators on modules, and algebras of operators, it is now the time to peruse mappings between such algebras.
\subsection{Adjoint operators}
We will be mostly interested in mappings defined by conjugation. A bounded operator $s\colon \CHx\to\CHy$ gives rise to a norm-continuous mapping $\Ad(s)\colon \CB(\CHx)\to\CB(\CHy)$ by $t \mapsto sts^*$. 
Observe that $\norm{\Ad(s)}=\norm{s}^2$. 
We shall momentarily prove that  if $s$ is a controlled operator, then $\Ad(s)$ is by and large compatible with the various Roe-like algebras (see \cref{cor:coarse functions preserve roe-like algebras} for a precise statement).
The following result is simple but important.
\begin{proposition}\label{prop:coarse functions preserve controlled propagation}
  Let $s\colon\CHx\to\CHy$ be bounded. Consider the assertions.
  \begin{enumerate} [label=(\roman*)]
    \item \label{prop:coarse functions preserve controlled propagation:1} $s$ is controlled (in the sense of \cref{def:controlled operator}).
    \item \label{prop:coarse functions preserve controlled propagation:2} For every $E\in\CE$ there is $F\in\CF$ with $\Ad(s)$ sending operators of $E$-controlled propagation to operators of $F$-controlled propagation.
    \item \label{prop:coarse functions preserve controlled propagation:3} For every $E\in\CE$ there is $F\in\CF$ with $\Ad(s)$ sending $\varepsilon$-$E$-approximable operators to $\varepsilon\norm{s}^2$-$F$-approximable operators.
    \item \label{prop:coarse functions preserve controlled propagation:4} For every $E \in \CE$ there is $F \in \CF$ with $\Ad(s)$ sending $\varepsilon$-$E$-quasi-local operators to $\varepsilon\norm{s}^2$-$F$-quasi-local operators.
  \end{enumerate}
  Then \cref{prop:coarse functions preserve controlled propagation:1} $\Leftrightarrow$ \cref{prop:coarse functions preserve controlled propagation:2} $\Rightarrow$ \cref{prop:coarse functions preserve controlled propagation:3}. Moreover, \cref{prop:coarse functions preserve controlled propagation:1} $\Rightarrow$ \cref{prop:coarse functions preserve controlled propagation:4} when $\CHx$ is admissible.
\end{proposition}
\begin{proof}
  \cref{prop:coarse functions preserve controlled propagation:1} $\implies$ \cref{prop:coarse functions preserve controlled propagation:2}. Fix a representative $S\subseteq Y\times X$ for $\csupp(s)$ such that $\supp(s)\subseteq S$ (\emph{e.g.}\ the one in \cref{prop:coarse_support_exists}).
  Let $t\in\CB(\CHx)$ and $E\in\CE$ be such that $\supp(t)\subseteq E$, and let $\gaugex$ be a non-degeneracy gauge for $\CHx$.
  It then follows from \cref{lem:supports and operations} that
  \begin{equation}\label{eq:Ad(s) is controlled}
    \supp(sts^*)\subseteq S\circ \gaugex\circ E\circ\gaugex\circ\op S\eqqcolon F.
  \end{equation}
  Then, by definition of (partial) coarse function, $F \in \CF$, as desired.

  \smallskip

  \cref{prop:coarse functions preserve controlled propagation:2} $\implies$ \cref{prop:coarse functions preserve controlled propagation:1}. As in \cref{prop:coarse_support_exists}, let
  \[
    S\coloneqq \bigcup\left\{B\times A \;\middle|
    \begin{array}{c}A\ \text{measurable \gaugex-bounded}\\ B\text{ measurable \gaugey-bounded}\end{array},\ \chf{B}t\chf{A}\neq 0\right\}
  \]
  be a representative for $\csupp(t)$ and fix $E\in \CE$. Suppose $y'(S\circ E\circ \op S)y$. That is, there are $x',x$ with $y'Sx'Ex\op Sy$. This in turn means that there are measurable \gaugex-controlled $A', A\subseteq X$, measurable \gaugey-controlled $B',B\subseteq Y$ and unit vectors $v'\in\CH_{A'}$ and $v\in\CH_{A}$ with $x'\in A'$, $x\in A$ and so that both $\chf{B'}s\chf{A'}(v')=\chf{B'}s(v')$ and $\chf{B}s\chf{A}(v)=\chf{B}s(v)$ are non-zero.
  Recall that $\matrixunit{v'}{v} (h) \coloneqq \scal{v}{h} \, v'$  (cf.\ \cref{subsec:operator-boolalg}). By construction, $\matrixunit{v'}{v}$ is supported in $A'\times A\subseteq \gaugex\circ E\circ\gaugex$. Let $F\in\CF$ be so that $\Ad(s)$ sends operators of $\gaugex\circ E\circ\gaugex$-controlled propagation to operators of $F$-controlled propagation.

  Since $\Ad(s)\matrixunit{v'}{v}=\matrixunit{s(v')}{s(v)}$, it follows
  \[
    \norm{\chf{B'}s\matrixunit{v'}{v}s^*\chf{B}s}
    =\norm{\chf{B'}\matrixunit{s(v')}{s(v)}\chf{B}}
    =\norm{\chf{B'}s(v)}\norm{\chf Bs(v)}\neq 0.
  \]
  Since $\Ad(s)\matrixunit{v'}{v}$ has $F$-controlled propagation, $B'$ and $B$ are not $F$-separated, hence $(y',y)\in B'\times B\subseteq\gaugey \circ F \circ\gaugey$. This shows that indeed $S\circ E\circ \op S\subseteq\gaugey \circ F \circ\gaugey$.

  \smallskip

  \cref{prop:coarse functions preserve controlled propagation:1} $\implies$ \cref{prop:coarse functions preserve controlled propagation:3}. This follows immediately from \cref{prop:coarse functions preserve controlled propagation:1} $\implies$ \cref{prop:coarse functions preserve controlled propagation:2}, because $\Ad(s)$ sends an operator of $E$-controlled propagation that $\varepsilon$-approximates $t$ to an operator of $F$-controlled propagation that $\norm{\Ad(s)}\varepsilon$-approximates $\Ad(s)(t)$.

  \smallskip

  \cref{prop:coarse functions preserve controlled propagation:1} $\implies$ \cref{prop:coarse functions preserve controlled propagation:4}.
  This goes along the same lines. Suppose that $E\in\CE$ and $\varepsilon \geq 0$ are so that $\norm{\chf{A'}t\chf A}\leq \varepsilon$ for every pair of $E$-separated measurable subsets of $X$.
  Let $B',B\subseteq Y$ be $F$-separated measurable subsets, where $F$ is as in \eqref{eq:Ad(s) is controlled} and $\gaugex$ is now also assumed to be an admissibility gauge (as well as a non-degeneracy one).
  Then there are \gaugex-controlled measurable thickenings $A'$ and $A$ of $\op R(B')$ and $\op R(B)$ respectively. Observe that $\chf{B'}s=\chf{B'}s\chf{A'}$ and $s^*\chf B=\chf As^*\chf B$. Since $A'$ and $A$ are $E$-separated by construction,
  \[
    \norm{\chf{B'} sts^*\chf B}
    =\norm{\chf{B'} s\chf{A'} t\chf A s^*\chf B}
    \leq\norm{s}\norm{\chf{A'} t\chf A}\norm{s^*}\leq \varepsilon\norm{s}^2
  \]
  as desired.
\end{proof}

\begin{remark}
  The equivalence between \cref{prop:coarse functions preserve controlled propagation:1} and \cref{prop:coarse functions preserve controlled propagation:2} opens a way for defining approximate versions of control for operators much in the spirit of \cref{def:approx-ope} and \cref{def:ql-ope}. These will play an important role in \cite{rigid}.
\end{remark}

Observe that, even if $s$ is controlled, $\Ad(s)$ needs not preserve local compactness of operators.
On the other hand, it is an immediate consequence of \cref{cor:local compactness and compositions by proper} that $\Ad(s)$ does preserve local compactness if $s$ is proper (cf.\ \cref{def: proper operator}). We record this as a corollary.
\begin{corollary}\label{cor:proper operators preserve local compactness}
  If $\CHx$ is locally admissible and $s \colon \CHx \to \CHy$ is proper, then $\Ad(s)$ sends locally compact operators to locally compact operators.
\end{corollary}
\begin{proof}
  Apply \cref{cor:local compactness and compositions by proper} with $\crse X = \crse Y$, $\crse Z = \crse W$ and $s=r$.
\end{proof}

We end this section with the previously announced result, which is a consequence of \cref{prop:coarse functions preserve controlled propagation,cor:proper operators preserve local compactness}.
\begin{corollary}\label{cor:coarse functions preserve roe-like algebras}
  Let $\CHx$ be admissible. If $s\colon\CHx\to\CHy$ is controlled, then $\Ad(s)$ defines (strongly) continuous maps
  \begin{align*}
    \cpstar{\CHx}&\to\cpstar{\CHy}, \\
    \cpcstar{\CHx}\to\cpcstar{\CHy}, &\text{ and} \;\;
    \qlcstar{\CHx} \to \qlcstar{\CHy}.
  \end{align*}
  Moreover, if $s$ is proper then $\Ad(s)$ gives maps
  \begin{align*}
    \roestar{\CHx}&\to\roestar{\CHy}, \\
    \roecstar{\CHx}&\to\roecstar{\CHy}.
  \end{align*}
  Lastly, if $s$ is an isometry then all of the above are \Star{}embeddings.
\end{corollary}

\begin{remark} \label{rmk:coarse functions preserve roe-like functorially}
  One way of rephrasing \cref{cor:coarse functions preserve roe-like algebras} is that $\cpstar{\variable}$, $\cpcstar{\variable}$ and $\qlcstar{\variable}$ are functors from the category of coarse geometric modules and controlled isometries between them, to the category of \cstar{}algebras and \Star{}homomorphisms. Similarly, $\roestar{\variable}$ and $\roecstar{\variable}$ are functors defined on the category of coarse geometric modules and \emph{proper} controlled isometries.
\end{remark}

\subsection{Covering isometries}\label{sec:covering iso}
One aspect we have not yet covered is how to construct controlled operators (and hence induce adjoint mappings of Roe-like algebras). The aim of this technical but important subsection is to show that discrete coarse geometric modules admit a great abundance of such operators. These can be constructed by ``covering'' coarse maps.
\begin{definition}\label{def:covering}
  An operator $t\colon \CHx\to \CHy$ \emph{covers} a partial coarse map $\crse{f\colon X\to Y}$ if $\csupp(t)=\crse f$.
\end{definition}

In the following discussion, $\crse{f\colon X\to Y}$ will be a fixed coarse map, and $t\colon \CHx\to\CHy$ an operator with $\csupp(t)\crse{\subseteq f}$. We wish to understand how large $\csupp(t)$ actually is.
\begin{lemma}\label{lem:injective operators have large domain}
   Fix $Z\subseteq X$ and a controlled entourage $E\in\CE$. Suppose there are $E$-bounded measurable $A_i\subseteq X$ and $v_i\in\CH_{A_i}$, where $i\in I$, such that
  \begin{itemize}
    \item $t(v_i)\neq 0$ for every $i\in I$;
    \item $Z\subseteq \bigcup_{i\in I}A_i$.
  \end{itemize}
  Then $\crse{Z\subseteq\pi_{X}}\paren{\csupp(t)}$.
\end{lemma}
\begin{proof}
  Let $S$ be as in \cref{prop:coarse_support_exists}, in particular $\csupp(t)=[S]$. Arguing as in the proof of \cref{prop:coarse_support_exists},
  we observe that for every $i\in I$ there are measurable $\gaugex$ and $\gaugey$-bounded sets $A$ and $B$ such that $\chf{B}t\chf{A}(v_i)\neq 0$.
  This implies that $A_i\cap A\neq\emptyset$. Since by construction $B\times A\subseteq S$, it follows that $A\cap\pi_X(S)\neq\emptyset$.
  Thus, $Z\subseteq E(\pi_X(S))$, as every $A_i$ is $E$-bounded.
\end{proof}

Considering the latter together with \cref{cor:restriction with same domain coincides with f} yields the following.

\begin{corollary}\label{cor:isometries cover}
  If $\CHx$ is a faithful module and $t$ is an isometry supported on $\crse f$, then it covers it.
\end{corollary}

\begin{remark}
  \cref{cor:isometries cover} implies that \cref{def:covering} is an extension of the usual definition of covering isometry.
  That is, if $X$ and $Y$ are metric spaces and $t\colon \CH_X\to\CH_Y$ is an isometry between (classical) geometric modules, then $t$ covers a controlled map $f\colon X\to Y$ in the sense of~\cite{higson-2000-analytic-k-hom}*{Definition~6.3.9} or~\cite{willett_higher_2012}*{Definition~4.3.3} if and only if it covers $\crse f$ according to \cref{def:covering}.
\end{remark}

The following technical lemma is a useful tool when constructing partial isometries that cover a given partial coarse map $\crse f \colon \crse{X \to Y}$.
\begin{lemma}\label{lem:partial isometries defined on bases are controlled}
  Let $(v_i)_{i\in I}$ and $(w_i)_{i\in I}$ be families of orthonormal vectors in $\CHx$ and $\CHy$ respectively. Suppose there are $E\in \CE$ and $F\in\CF$, measurable $E$-controlled $A_i \subseteq X$, and measurable $F$-controlled $B_i\subseteq Y$ such that for every $i\in I$
  \begin{enumerate} [label=(\roman*)]
    \item \label{lem:partial isometries defined on bases are controlled:vi} $v_i\in\CH_{A_i}$;
    \item \label{lem:partial isometries defined on bases are controlled:wi} $w_i\in\CH_{B_i}$;
    \item \label{lem:partial isometries defined on bases are controlled:f} $f(A_i)\cap B_i\neq\emptyset$.
  \end{enumerate}
  Then the partial isometry $t\colon \CHx\to\CHy$ defined by $t(v_i)=w_i$ and $0$ on the orthogonal complement of $\angles{v_i\mid i\in I}$ is supported on $\crse f$.
\end{lemma}
\begin{proof}
  Let $R\coloneqq \bigcup_{i\in I}B_i\times A_i$. Assumptions \cref{lem:partial isometries defined on bases are controlled:vi,lem:partial isometries defined on bases are controlled:wi,lem:partial isometries defined on bases are controlled:f} imply that
  \[
    R\subseteq F\circ \grop{f}\circ E\csub \grop{f}.
  \]
  It is therefore enough to verify that $\supp(t)\subseteq R$. Suppose that $\chf{B}t\chf{A}\neq 0$. Then there is some $i\in I$ such that neither $\chf{A}(v_i)$ nor $\chf{B}(w_i)$ are $0$. But this implies that the intersections $A\cap A_i$ and $B\cap B_i$ are not empty, and therefore $B$ and $A$ are not $R$-separated.
\end{proof}

We may now prove the existence of covering partial isometries.
\begin{theorem}\label{thm:covering isometries exist}
  Let $\crse{f \colon X\to Y}$ be a partial coarse map and let $\CHx$ be faithful and discrete.
  \begin{enumerate}[label=\arabic*.]
    \item \label{thm:covering isometries exist:inj} If $\CHx$ has rank at most $\kappa$ and $\CHy$ is $\kappa$-ample then
      \begin{enumerate}[label=(\cref{thm:covering isometries exist:inj}\alph*)]
        \item \label{thm:covering isometries exist:inj-pi} there is a partial isometry $t\colon\CHx\to\CHy$ covering $\crse f$;
        \item \label{thm:covering isometries exist:inj-s} if $\cdom(\crse f)=\crse X$ we may further assume $t$ is an isometry.
      \end{enumerate}
    \item \label{thm:covering isometries exist:sur} If $\crse f$ is coarsely surjective, $\CHy$ is discrete of rank at most $\kappa$ and $\CHx$ is $\kappa$-ample then:
      \begin{enumerate}[label=(\cref{thm:covering isometries exist:sur}\alph*)]
        \item \label{thm:covering isometries exist:sur:coi} there exists a coisometry $t\colon\CHx\to\CHy$ covering $\crse f$;
        \item \label{thm:covering isometries exist:sur:uni} if $\cdom(\crse f)=\crse X$ and both $\CHx$ and $\CHy$ are $\kappa$-ample of rank $\kappa$, then $t$ can be taken to be a unitary.
      \end{enumerate}
  \end{enumerate}
  Moreover, if  $\crse f$ is proper all the above hold with ``rank at most $\kappa$'' replaced with ``\emph{local} rank at most $\kappa$'' (see \cref{rmk: rank vs local rank}).
\end{theorem}

\begin{proof}
  We prove Part~\cref{thm:covering isometries exist:inj} first.
  By \cref{cor:discrete and faithful}, there is a discrete partition $X=\bigsqcup_{i\in I}A_i$ such that $\chf{A_i}$ is not zero for every $i\in I$.
  Likewise, by \cref{cor:exists disjoint ample coarsely dense} there is a controlled family $(B_j)_{j\in J}$ of disjoint measurable subsets of $Y$ such that $\chf{B_j}$ has rank at least $\kappa$ for every $j\in J$ and $\bigsqcup_{j\in J}B_j$ is $F$-dense, where $F \in \CF$.

  Choose a representative $f$ for $\crse f$, and let $I_0 \coloneqq \{i\in I\mid A_i\cap\dom(f)\neq \emptyset\}$. Then $f(A_i)\neq\emptyset$ for every $i\in I_0$. By coarse denseness of $(B_j)_{j \in J}$ there is some $j\in J$ with $f(A_i)\cap F(B_j)\neq\emptyset$. Arbitrarily choose one such $j$ and call it $\phi(i)$. This defines a function $\phi\colon I\to J$.
  For every $j\in J$, let $I_j\coloneqq \phi^{-1}(j)$, so that $I_0=\bigsqcup_{j\in J}I_j$.

  Fix $j\in J$. For every $i\in I_j$, choose a Hilbert basis $\braces{v_{i,k}}_{k\in K_i}$ of $\CH_{A_i}$. Similarly choose Hilbert bases $\braces{w_{j,l}}_{l\in L_j}$ for $\CH_{B_j}$.  By assumption, the cardinality of $L_j$ is at least $\kappa$, while 
  we have:
  \begin{equation}\label{eq: cardinality bound_1}
    \abs*{\bigsqcup\{K_i\mid i\in I_j\} }\leq \rank(\CHx)\leq \kappa.
  \end{equation}
  We may thus choose an injective mapping
  \[
    \braces{v_{i,k}\mid i\in I_j,\ k\in K_i}\hookrightarrow \braces{w_{j,l}\mid l\in L_j},
  \]
  which defines an isometry $t_j\colon \bigoplus_{i\in I_j}\CH_{A_i}\to\CH_{B_j}$. Doing this for every $j\in J$ and mapping the orthogonal complement of $\bigoplus_{i\in I_0}\CH_{A_i}$ to zero defines a partial isometry $t\colon \CHx\to\CHy$.

  Using \cref{lem:partial isometries defined on bases are controlled} yields $\csupp(t)\crse{\subseteq f}$. On the other hand, since $\dom(f)\subseteq \bigsqcup_{i\in I_0}A_i$, by \cref{lem:injective operators have large domain} it follows that $\cdom(\crse f)\crse\subseteq\crse{\pi_{X}}(\csupp(t))$. It then follows from \cref{cor:restriction with same domain coincides with f} that $\crse f=\csupp(t)$, establishing \cref{thm:covering isometries exist:inj-pi}.
  Moreover, if $\cdom(\crse f)=\crse X$ we may choose $f$ with $\dom(f)=X$, so that $I_0=I$. It then follows that $t$ is actually an isometry, yielding \cref{thm:covering isometries exist:inj-s}.

  \smallskip

  Part~\cref{thm:covering isometries exist:sur} is similar. Let $X=\bigsqcup_{i\in I}A_i$ be such that $\chf{A_i}$ has rank at least $\kappa$ for every $i\in I$. As before, choose a representative $f$ of $\crse f$ and let $I_0$ be the set of indices with $A_i\cap\dom(f)\neq\emptyset$.

  Since $f$ is controlled, the images $(f(A_i))_{i \in I}$ are uniformly controlled subsets of $Y$. They are, moreover, coarsely dense because $\crse f$ is coarsely surjective. By applying \cref{lem:discrete have nice partitions} to the discrete module $\CHy$ we obtain a discrete partition $Y=\bigsqcup_{j\in J}B_j$ with the property that for every $j\in J$ there is an $i\in I_0$ such that $f(A_i)\subseteq B_j$. This means that we can construct a function $\phi\colon I\to J$ as before, but this time we can also assume it is surjective.

  Again, choose Hilbert basis $\braces{v_{i,k}}_{k\in K_i}$ of $\CH_{A_i}$ and $\braces{w_{j,l}}_{l\in L_j}$ of $\CH_{B_j}$. Since $\abs{K_i}\geq \kappa$ for every $i\in I_0$ and 
  \begin{equation}\label{eq: cardinality bound_2}
    \abs{L_j}\leq\rank(\CHy)\leq\kappa,
  \end{equation}
  there is an embedding in the opposite direction
  \[
    \braces{v_{i,k}\mid i\in I_j,\ k\in K_i}\hookleftarrow \braces{w_{j,l}\mid l\in L_j}.
  \]
  Such a map defines an isometry $\CH_{B_j}\to \bigoplus_{i\in I_j}{\CH_{A_j}}$ whose adjoint is a co-isometry $t_j\colon\bigoplus_{i\in I_j}{\CH_{A_j}}\to \CH_{B_j}$.
  The rest of the argument now runs exactly as in Part~\cref{thm:covering isometries exist:inj}. Furthermore, observe that for \cref{thm:covering isometries exist:sur:uni}, all the basis have cardinality $\kappa$. Thus the injections may be replaced by bijections, and the resulting $t$ is a unitary.

  \smallskip

  For the moreover statement, since $\crse f$ is proper, the preimage of any bounded set $B\subseteq Y$ is bounded, and so is $C_B\coloneqq\bigsqcup\braces{A_i\mid A_i\cap f^{-1}(B)\neq \emptyset}$. We may then replace \cref{eq: cardinality bound_1} with 
  \[  
    \abs*{ \bigsqcup\{K_i\mid i\in I_j\} }= \rank\paren{\chf{C_{B_j}}} \leq \kappa.
  \]
  The condition of local rank at most $\kappa$ also suffices to deduce \cref{eq: cardinality bound_2}.
  These are the only adjustments needed to run the above arguments.
\end{proof}

Together with \cref{cor:coarse functions preserve roe-like algebras}, this implies the following (observe that if $t$ covers $\crse f$ then it is proper if and only if so is $\crse f$, cf.\ \cref{def: proper operator}).
\begin{corollary}\label{cor:existence of covering Ad}
  If $\CHx$ and $\CHy$ are both discrete, $\kappa$-ample and of rank $\kappa$, then every coarse map $\crse{f\colon X\to Y}$ is covered by an isometry $t$, which then induces a \Star{}embedding $\Ad(t)$ at the level of $\cpcstar{\variable}$ and $\qlcstar{\variable}$.
  If $\crse f$ is also proper, $\Ad(t)$ defines a \Star{}embedding of Roe algebras.
  If $\crse f$ is a coarse equivalence, $t$ can be taken to be unitary.
\end{corollary}

\begin{remark} \label{rem:covering isos}
  With moderate effort, one can prove minor improvements or modifications of \cref{thm:covering isometries exist,cor:existence of covering Ad}. Among these, we would like to highlight the following two.
  \begin{enumerate}[label=(\roman*)]
    \item Conclusion~\cref{thm:covering isometries exist:sur:coi} in \cref{thm:covering isometries exist} does not require $\CHx$ to be discrete. Likewise, if $\crse f$ is a coarse embedding then~\cref{thm:covering isometries exist:inj-pi} holds even if $\CHy$ is assumed to be faithful instead of $\kappa$-ample.
    \item \label{rem:covering isos:coe-covered-loc-fin} The ``moreover'' statement of \cref{thm:covering isometries exist} shows that if $\crse f$ is assumed to be proper then in \cref{cor:existence of covering Ad} it is enough to ask that $\CHx$ and $\CHy$ be $\kappa$-ample of local rank at most $\kappa$ (\emph{e.g.}\ ample and locally separable).
  \end{enumerate}
\end{remark}

\subsection{Coarse equivalences, controlled unitaries and outer automorphisms}
Observe that the homomorphisms constructed in \cref{cor:existence of covering Ad} are highly non-canonical. The main aim of this subsection is to show that the choices involved there simply result in conjugated isometries. In the classical setting, this was observed in \cite{braga_gelfand_duality_2022}, see also \cite{spakula_rigidity_2013}*{Theorem A.5} for the algebraic counterpart.
\begin{lemma}\label{lem: operators with same coarse support are close}
  Let $t_0,t_1\colon\CHx\to\CHy$ be controlled operators with $\csupp(t_0)=\csupp(t_1)$. Then $t_it_j^*$ for $i,j=0,1$ has controlled propagation, \emph{i.e.}\ it belongs to $\cpstar{\CHy}$.
\end{lemma}
\begin{proof}
  Throughout, $i,j=0,1$. Let $\crse f$ denote $\csupp(t_0)=\csupp(t_1)$. By assumption, $\crse{f\colon X\to Y}$ is a partially defined coarse map.
  It follows from \cref{cor:coarse support of composition}~\cref{cor:coarse support of composition:star} that $\csupp(t_j^*)=\op{\crse f}$. In particular, 
  \[
    \crse{\pi_{X}}(\csupp\paren{t_j^*})=\crse{\pi_{X}}(\csupp\paren{t_i})=\cdom(\crse f).
  \]
  Since $\crse f$ is controlled, we deduce from \cref{cor:coarse support of composition}~\cref{cor:coarse support of composition:comp} that
  \begin{equation*}
    \csupp(t_it_j^*)
    \crse\subseteq \csupp(t_i)\crse\circ \csupp(t_j^*)
    =\crse{f\circ}\op{\crse f}.
  \end{equation*}
  Since the coarse composition is well-defined (see \cref{lem:coarse composition is well-defined}), we may as well fix a function $f$ as representative for $\crse f$, so that $\crse f=[\grop f]$. Since $\pi_X(\grop f)=\pi_X(\gr f)=\dom(f)$, \cref{lem:coarse composition is well-defined} shows that
  \[
    \crse{f\circ}\op{\crse f} = [(\grop f)\circ(\gr f)] = [\Delta_{f(X)}].
  \]
  This shows that $\csupp(t_it_j^*)\crse\subseteq \cid_{\crse Y}$, as desired.
\end{proof}

\cref{lem: operators with same coarse support are close} shows, in particular, that if $u_0$ and $u_1$ are controlled unitaries with the same coarse support, then $u_0u_1^*$ is a unitary element in $\cpstar{\CHy}$. To better leverage this observation, we make some definitions.
\begin{definition} \label{def: CE and CUni}
  Let $\crse X$ be a coarse space, and let $\CHx$ be an $\crse{X}$-module. Then we denote by
  \begin{enumerate}[label=(\roman*)]
    \item $\coe{\crse X}$ the set of coarse equivalences $\crse{f \colon X \to X}$;
    \item $\cuni{\CHx}$ the set of unitaries $u \colon \CHx \to \CHx$ such that both $u$ and $u^*$ are controlled.
  \end{enumerate}
\end{definition}

Observe that both sets above are, in fact, groups under composition (in the definition of $\cuni{\CHx}$ it is necessary to impose that $u^*$ is controlled in order to have inverses).
The action by conjugation of $\cuni{\CHx}$ on itself is nothing but $\Ad$. It then follows from \cref{prop:coarse functions preserve controlled propagation} that the group of unitaries $U(\cpstar{\CHx})$ is a normal subgroup of $\cuni{\CHx}$. The following is now easy to prove.

\begin{theorem}\label{thm: coarse equivalences iso to cuni}
  Let $\CHx$ be a discrete module that is $\kappa$-ample of local rank at most $\kappa$, where $\kappa >0$ is arbitrary. Then there is a canonical isomorphism
  \[
    \rho\colon\coe{\crse X}\xrightarrow{\ \cong\ } \cuni{\CHx}/U(\cpstar{\CHx}).
  \]
\end{theorem}
\begin{proof}
  Given $\crse f\in\coe{\crse X}$, by \cref{thm:covering isometries exist} we may choose a unitary $u$ covering $\crse f$. Since $\crse f$ is a coarse equivalence, and $u^*$ covers $\op{\crse{f}}=\crse f^{-1}$, we see that both $u$ and $u^*$ are controlled. We may thus define $\rho(\crse f)\coloneqq [u]$. 
  
  If $u'\in\cuni{\CHx}$ is a different choice of covering unitary, \cref{lem: operators with same coarse support are close} shows that $u' = (u'u^*)u$ with $u'u^*\in U(\cpstar{\CHx})$, therefore $[u]=[u']$. This shows that $\rho$ is well-defined, and it is clearly a homomorphism, because $u_1u_2$ covers $\csupp(u_1)\circ\csupp(u_2)$ by \cref{cor:coarse support of composition,cor:isometries cover}.
  Injectivity is also clear. By definition, $[u]=[1_{\CHx}]$ if and only if $u\in U(\cpstar{\CHx})$, \emph{i.e.}\ $\csupp(u)\crse \subseteq \cid_{\crse X}$. This shows that $\ker(\rho)=\{\cid_{\crse X}\}$, as desired.

  It only remains to discuss surjectivity.
  For a given $u\in\cuni{\CHx}$, both $\csupp(u)$ and $\csupp(u^*)$ are partial coarse maps.
  Recall that by \cref{cor:coarse support of composition} we have $\csupp(u^*) = \op{\csupp(u)}$.
  Since $\CHx$ is faithful and $u$ is a unitary, \cref{lem:injective operators have large domain} implies that $\csupp(u)$ and $\csupp(u^*)$ are coarsely everywhere defined.
  It then follows from \cref{prop:transpose is coarse inverse} that $\csupp(u)$ is a coarse equivalence (with coarse inverse $\csupp(u^*)$).
  Since an operator covers its own support by definition, we deduce that $[u]=\rho(\csupp(u))$.
\end{proof}

It is now natural to attempt to compose this homomorphism with the action by conjugation on the Roe-like algebras. Before doing so, we prove a preliminary result.
\begin{lemma}\label{lem: controlled quasi local is in cp}
  Let $\CHx$ be an admissible $\crse X$-module.
  If $t\colon\CHx\to\CHx$ is a controlled isometry that is also quasi-local, then it has controlled propagation.
\end{lemma}
\begin{proof}
  Let $R\subseteq Y\times X$ be a representative for $\csupp(t)$ with $\supp(t)\subseteq R$, and consider the projection onto the second coordinate $\pi_2(R)\subseteq X$. Observe that there is some $E_0\in\CE$ such that for every $x\in \pi_2(R)$ we can find an $E_0$-controlled measurable thickening $x\in A_x$ such that $\CH_{A_x}\neq \{0\}$ (this is obvious \emph{e.g.} if $R$ is the support constructed in \cref{prop:coarse_support_exists}).
  
  Observe that $R(A_x)\times R(A_x)$ is contained in $E_1\coloneqq R\circ E_0\circ \op R$. That is, $R(A_x)$ is $E_1$-controlled. Let $\gauge$ be an admissibility gauge. We may then choose for every $x\in \pi_2(R)$ a measurable $\gauge$-controlled thickening $A'_x$ of $R(A_x)$. In particular, the sets $A'_x$ are $E_2$\=/controlled for $E_2\coloneqq \gauge\circ E_1\circ\gauge$.

  By definition, $t$ is \emph{not} controlled if and only if $R\notin \CE$. This implies that for every $E\in\CE$ there exist $x'Rx$ with $x'\sep{E_3}x$, where $E_3\coloneqq E_2\circ E \circ E_0$. Since $x'\in R(A_x)\subseteq A'_x$, the choice of $E_3$ is made so that $A'_x$ and $A_x$ must then be $E$-separated.
  Let now $v$ be a unit vector in $\CH_{A_x}$ (which is non-trivial by construction). Since $\supp(t)\subseteq R$ and $t$ is an isometry, the image $t(v)$ is a unit vector supported in $A_x'$. It follows that $\norm{\chf{A'_x}t\chf{A_x}} =1$. Since $A'_x$ and $A_x$ are $E$-separated by construction, and $E$ is arbitrary, this shows that $t$ is not quasi-local.
\end{proof}

\cref{lem: controlled quasi local is in cp} shows that $\cuni{\CHx}\cap\qlcstar{\CHx}=U(\cpstar{\CHx})$, which we see as \emph{trivial}, \emph{i.e.}\ \emph{inner}, unitaries. This will be needed in the proof of the following (recall that $\multiplieralg{\variable}$ denotes the multiplier algebra).
\begin{theorem}[cf.\ \cite{braga_gelfand_duality_2022}*{Section 2.2}]\label{thm: injection into outer}
  Let $\CHx$ be an admissible $\crse X$-module. Then $\Ad$ induces canonical homomorphisms
  \[
    \sigma_{\CR}\colon\cuni{\CHx}/U(\cpstar{\CHx})\rightarrow \out(\roeclike{\CHx}).
  \]
  Moreover, if $\multiplieralg{\roeclike{\CHx}} \subseteq \qlcstar{\CHx}$ then $\sigma_{\CR}$ is injective.
  In particular, the latter is the case when $\roeclike\CHx$ is $\cpcstar\CHx$ or $\qlcstar{\CHx}$.
\end{theorem}
\begin{proof}
  Recall that the group of outer automorphisms is the quotient of $\aut(\roeclike{\CHx})$ by the action by conjugation of the unitaries in the multiplier algebra of $\roeclike{\CHx}$.
  \cref{cor:cpc is multiplier of roec} shows that $\cpstar{\CHx}$ is contained in the multiplier algebra of the three Roe-like \cstar{}algebras $\roeclike{\CHx}$. It then follows that $\Ad$ does indeed descend to homomorphisms
  \[
    \sigma_{\CR} \colon \cuni{\CHx}/U(\cpstar{\CHx})\to \out(\roeclike{\CHx}).
  \]

  It remains to prove injectivity of the maps $\sigma_{\CR}$. Suppose that $u\in \cuni{\CHx}$ is a controlled unitary such that $\Ad(u)=\Ad(m)$ for some $m\in U(\multiplieralg{\roeclike{\CHx}})$. This implies that $m^*ut=tm^*u$ for every $t\in \roeclike{\CHx}$. That is, $m^*u$ is a unitary in the commutant $\roeclike{\CHx}'$.

  Let $\crse X=\bigsqcup_{i\in I}\crse X_i$ be the decomposition in coarsely connected components of $\crse X$. Arguing as in \cref{lemma:cartan:commutant}, we deduce that $\roeclike{\CHx}'=\prod_{i\in I}\CCC\cdot \chf{\crse X_i}$.
  In particular, the element $m^*u$ is contained in $U(\cpstar{\CHx})$.
  As a consequence, we see that the multiplier $m = u(u^*m)$ is also a controlled unitary.
  Since $m \in U(\multiplieralg{\roeclike{\CHx}}) \subseteq \qlcstar\CHx$ by assumption, it follows from \cref{lem: controlled quasi local is in cp} that $m$ actually belongs to $\cpstar{\CHx}$.
  Writing $u = m (m^*u)$ we see that $u\in\cpstar{\CHx}$ as well.

  For the ``in particular'' statement, simply note that the hypothesis that $\multiplieralg{\roeclike{\CHx}} \subseteq \qlcstar{\CHx}$ clearly holds if $\roeclike{\CHx}$ is $\cpcstar\variable$ or $\qlcstar\variable$, as both are unital \cstar{}algebras.
\end{proof}

\begin{remark}\label{rmk: Ccp is multiplier of Roe}
  The injectivity part of \cref{thm: injection into outer} even holds when $\roeclike{\variable}$ is $\roecstar{\variable}$, as long as $\crse{X}$ is a coarsely locally finite extended metric space and $\CHx$ is discrete. In fact, it can be shown that in this case $\multiplieralg{\roecstar\CHx}=\cpcstar\CHx$, see \cfRigidityMultiplier (cf.~\cite{braga_gelfand_duality_2022}*{Theorem 4.1}).
  However, it is also shown in \cite{rigid}*{Remark~9.16} that there are non countably generated coarse spaces for which the equality $\multiplieralg{\roecstar\CHx}=\cpcstar\CHx$ fails and $\sigma_{\rm Roe}$ is not injective.
\end{remark}

Together with \cref{thm: coarse equivalences iso to cuni}, we obtain the following.
\begin{corollary} \label{cor:injection-from-ce-to-out}
  Let $\CHx$ be a discrete $\kappa$-ample $\crse X$-module of local rank $\kappa$. Then there are canonical homomorphisms 
  \[
    \coe{\crse X}\rightarrow \Out(\roeclike{\CHx})
  \]
  that are injective whenever $\multiplieralg{\roeclike{\CHx}} \subseteq \qlcstar{\CHx}$.
\end{corollary}

\begin{remark}
  Remarkably, it can be proved that if $\crse X$ is a coarsely locally finite extended metric space and $\CHx$ is discrete, then all the all the above homomorphisms (both $\rho$ and $\sigma_\CR$) are actually isomorphisms \cites{braga_gelfand_duality_2022,rigid}.
\end{remark}

\begin{remark}
  The same proof of \cref{thm: injection into outer} implies that $\Ad$ embeds $\cuni{\CHx}/U(\cpstar{\CHx})$ into
  \begin{itemize}
    \item $\aut(\cpstar{\CHx})/\Ad\bigparen{U(\cpstar{\CHx})}$; and
    \item $\aut(\roestar{\CHx})/\Ad\bigparen{U(\cpstar{\CHx})}$.
  \end{itemize}
  A straightforward modification shows that \cref{lem: controlled quasi local is in cp} also holds when $t$ is an invertible operator (bounded below in fact suffices). In view of this, one may also ignore the \Star{}algebra structure of $\cpstar{\CHx}$ and obtain an embedding into the group of algebraic outer automorphisms of the algebra $\cpstar{\CHx}$ (\emph{i.e.} $\aut(\cpstar{\CHx})$ modulo $t\sim ata^{-1}$ for every invertible $t\in \cpstar{\CHx}$). 
\end{remark}

\section{Roe-like algebras of coarse spaces} \label{sec:roe-algs-spaces}
In \cref{sec:roe-algs-modules} we showed how coarse geometric modules naturally give rise to \Star{}algebras and \cstar{}algebras of interest (see \cref{cor:coarse functions preserve roe-like algebras,rmk:coarse functions preserve roe-like functorially}).
We may now shift the point of view and study the coarse spaces themselves.
The goal is to construct algebras of operators that only depend on the coarse geometry of the space (\emph{i.e.}\ are invariant under coarse equivalence). After such algebras are defined, their invariants---like the K-theory---become coarse geometric invariants for the space.

\subsection{Definition of rank-\texorpdfstring{$\kappa$}{k} Roe-like algebras}
The main difficulty in defining Roe-like algebras for a coarse space $\crse X$ is that it is generally not possible to \emph{canonically} choose an $\crse X$-module in a satisfactory manner.
On one hand, it is natural to consider the uniform module $\ell^2(X)$ and the associated uniform Roe algebra $\uroecstar{\crse X} = \roecstar{\ell^2(X)}$ (recall \cref{exmp:uniform roe algebra}). However, from a coarse geometric perspective this may not be a particularly good idea, as this module is too heavily dependent on the underlying set $X$.
Indeed, if $\crse X$ and $\crse X'$ are coarsely equivalent, it is highly desirable that the associated algebras be isomorphic. This, unfortunately, is just not the case for the uniform Roe algebras.\footnote{%
\, In fact, one can show that if uniformly locally finite metric spaces $X$ and $X'$ are not amenable (in the sense of Block and Weinberger, see \cite{block_weinberger_1992}*{Section~3}) then the uniform Roe algebras $\uroecstar X$ and $\uroecstar{X'}$ are isomorphic if and only if there is a bijective coarse equivalence between $X$ and $X'$ (see \cite{braga_rigid_unif_roe_2022}*{Corollary~3.9}, \cite{white_cartan_2018}*{Theorem~5.1} and \cite{whyte_amenability_1999}*{Theorem~4.1}).
Actually, it is an open question whether the amenability assumption is even necessary.
Nevertheless, there are numerous examples of spaces that are coarsely equivalent but not bijectively so.
For instance, the easiest obstruction is given by the (local) cardinality of the spaces themselves.
}
The approach we take in this text is closer to the original one, due to Higson and Roe~\cite{higson-2000-analytic-k-hom} (see also \cite{willett_higher_2012}).
\begin{definition}\label{def:roe-algebras}
  Let $\kappa$ be an (infinite) cardinal. If $\crse X$ admits a $\kappa$-ample discrete module $\CH_{\crse X}$ of rank $\kappa$ (in the sense of \cref{def:ample}), we let
  \begin{itemize}
    \item $\roecstar[\kappa]{\crse X} \coloneqq \roecstar{\CH_{\crse X}}$;
    \item $\cpcstar[\kappa]{\crse X} \coloneqq \cpcstar{\CH_{\crse X}}$;
    \item $\qlcstar[\kappa]{\crse X} \coloneqq \qlcstar{\CH_{\crse X}}$,
  \end{itemize}
  be the \emph{$\kappa$-ample Roe algebra}, the \emph{$\kappa$-ample \cstar{}algebra of operators of controlled propagation}, and the \emph{$\kappa$-ample quasi-local algebra of $\crse X$}.
  If $\kappa=\aleph_0$, we omit it from the notation and call the resulting algebras \emph{Roe-like \cstar{}algebras of $\crse X$}.
\end{definition}

\begin{remark}
  \cref{def:roe-algebras} is only interesting for infinite cardinals. If $\kappa$ is finite, it is simple to show that $\crse X$ admits a $\kappa$-ample discrete module of rank $\kappa$ if and only if $\crse{X}$ is bounded.
\end{remark}

The notation of \cref{def:roe-algebras} is somewhat abusive, as it hides the dependence of $\roeclike{\crse X}$ from the choice of ample module. We shall momentarily see that this causes no issue, as different choices yield isomorphic algebras. Before doing so, we introduce a convenient piece of nomenclature.
\begin{definition}
  A homomorphism $\phi\colon \roeclike{\crse X}\to \roeclike{\crse Y}$ is \emph{spatially implemented} if it is of the form $\Ad(t)$ for some isometry $t\colon\CHx\to\CHy$ at the level of the defining coarse geometric modules.
  We define the \emph{coarse support} of a spatially implemented homomorphism $\phi=\Ad(t)$ as $\csupp(\phi)\coloneqq\csupp(t)$, and similarly say that $\phi$ \emph{covers} a coarse map $\crse{f\colon X\to Y}$ if so does $t$.
\end{definition}

\begin{remark}
  As observed in \cite{spakula_rigidity_2013}, the good understanding we have on the set of compact operators belonging to Roe-like algebras (recall \cref{cor: roe algs intersect compacts}) implies that most of the ``meaningful'' homomorphisms among Roe-like algebras are spatially implemented. This applies, for instance, to all isomorphisms. We refer to \cite{rigid}*{Proposition~7.2} for a thorough discussion of these aspects.
\end{remark}

It is now a simple matter to verify that \cref{thm:covering isometries exist} together with \cref{cor:coarse functions preserve roe-like algebras} implies that $\roecstar[\kappa]{\crse X}, \cpcstar[\kappa]{\crse X},$ and $\qlcstar[\kappa]{\crse X}$ are well defined up to \Star{}isomorphism, justifying \cref{def:roe-algebras}. More precisely, those results imply the following.
\begin{theorem}\label{thm:maps-among-roe-algs}
  Suppose that $\crse X$ and $\crse Y$ admit discrete $\kappa$\=/ample modules of rank $\kappa$. Given a coarse map $\crse{f\colon X\to Y}$, there are spatially implemented \Star{}homomorphisms
  \begin{align*}
    \psi_{\kappa,\crse f}^{\rm cp} & \colon \cpcstar[\kappa]{\crse X} \to \cpcstar[\kappa]{\crse Y}, \\
    \psi_{\kappa,\crse f}^{\rm ql} & \colon \qlcstar[\kappa]{\crse X} \to \qlcstar[\kappa]{\crse Y}
  \end{align*}
  covering $\crse f$. If $\crse f$ is proper, there is a spatially implemented \Star{}homomorphism
  \[
    \phi_{\kappa,\crse f}\colon\roecstar[\kappa]{\crse X} \to \roecstar[\kappa]{\crse Y}.
  \]
  Moreover, if $\crse f$ is a coarse equivalence then all of the above can be taken to be \Star{}isomorphisms.
\end{theorem}
\begin{proof}
  It is enough to consider $\Ad(t)$, where $t$ is an isometry obtained from \cref{thm:covering isometries exist} (the same isometry can be used to define all the claimed \Star{}homomorphisms). 
  The only thing that is not immediately clear from \cref{thm:covering isometries exist} is the `moreover' statement (this is due to the lack of naturality in the construction of $t$).
  This is however simple to verify: if $\crse f$ is a coarse equivalence, \cref{thm:covering isometries exist} can be used to obtain a unitary $u$ covering it. Since $u^*$ covers $\op{\crse{f}}=\crse f^{-1}$, $u^*$ is also controlled (and proper). Therefore $\Ad(u^*)\colon\roeclike[\kappa]{\crse Y}\to \roeclike[\kappa]{\crse X}$ is the inverse of $\Ad(u)\colon\roeclike[\kappa]{\crse X}\to \roeclike[\kappa]{\crse Y}$.
\end{proof}

\begin{remark}\label{rmk:roe alg are not functorial}
  Note that the \Star{}homomorphisms constructed in \cref{thm:maps-among-roe-algs} are non-canonical (although they only differ by conjugation in $\cpstar{\CHx}$, see \cref{lem: operators with same coarse support are close}).
  In particular, the choices made to construct them are not well-behaved under composition, \emph{i.e.}\ the induced maps from \Cat{Coarse} to \Cat{\cstar{}Alg} need not be functorial.
\end{remark}

\begin{example} \label{ex:roe-algs-extend-classical-case}
  Recall \cref{ex:metric-as-coarse-1}. If $(X,d)$ is a separable metric space then $\crse X=(X,\CE_d)$ admits separable ample $\crse X$-modules. For instance, if $Z\subseteq X$ is a countable dense subset then $\ell^2(Z) \otimes \ell^2(\NN)$ is such a module. In this case, $\roecstar{\crse X}$ coincides with the classical Roe algebra of $\crse X$.
  Thus, \cref{def:roe-algebras} extends the usual definition of the Roe algebra~\cites{roe_lectures_2003,higson-2000-analytic-k-hom,spakula_rigidity_2013}.
\end{example}

\begin{example}\label{exmp:uniform kappa-ample roe alg}
  Given $\crse X=(X,\CE)$ with $\abs{X}=\kappa$, let $\CH$ be a fixed Hilbert space of rank $\kappa$. Then the $\crse X$-module $\CH_{\crse X} \coloneqq \ell^2(X;\CH)$ is $\kappa$-ample of rank $\kappa$ (cf.\ \cref{exmp:uniform module}),
  and it may thus be used to construct the rank-$\kappa$ Roe-like algebras of $\crse X$.
\end{example}

The above examples show that every coarse space admits a rank-$\kappa$ Roe-like \cstar{}algebra for some $\kappa$ large enough. However, it is generally preferable to keep the rank as small as possible, and ideally even restrict to separable Hilbert spaces.
We observe that not every coarse space admits countable-rank Roe-like \cstar{}algebras. More specifically, we have the following.
\begin{lemma}\label{lem:existence of kappa-ample modules}
  Let $\crse X$ be a coarse space. The following are equivalent:
  \begin{enumerate}[label=(\roman*)]
    \item \label{lem:existence of kappa-ample modules:1} $\crse X$ admits a $\kappa$\=/ample discrete module of rank $\kappa$;
    \item \label{lem:existence of kappa-ample modules:2} $\crse X$ is coarsely equivalent to a coarse space over a set of cardinality at most $\kappa$.
  \end{enumerate}
\end{lemma}
\begin{proof}
  Suppose there is an $\crse X$-module $\CHx$ as in \cref{lem:existence of kappa-ample modules:1}. Choosing a gauge large enough, \cref{cor:discrete and faithful} yields a discrete partition $X=\bigsqcup_{i\in I}A_i$ so that $\chf{A_i}$ has rank $\kappa$ for every $i\in I$. As explained in \cref{rmk:discrete modules are discrete}, we may equip $I$ with a coarse structure so that $\crse{X}$ and $\crse{ I}$ are coarsely equivalent. Since $\CHx=\bigoplus_{i\in I}\CH_{A_i}$ has rank $\kappa$, the set $I$ has, at most, $\kappa$ elements.

  Suppose we are given a coarse equivalence $\crse{f\colon X\to I}$ for some set with $\abs{I}=\kappa' \leq \kappa$ as in \cref{lem:existence of kappa-ample modules:2}. Fix a representative $f\colon X\to I$ of $\crse f$. The sets $(f^{-1}(i))_{i \in I}$ form a controlled partition of $\crse X$. Fix a Hilbert space $\CH$ of rank $\kappa$ to obtain a $\kappa$\=/ample separable $\crse I$-module $\ell^2(I;\CH)$.
  Putting $\CHx \coloneqq \ell^2(I;\CH)$ and $\chf{\bullet}$ being defined on the Boolean algebra consisting of arbitrary preimages $f^{-1}(J)$, where $J\subseteq I$, yields a discrete $\kappa$\=/ample $\crse X$-module of rank $\kappa$.
\end{proof}

In light of \cref{lem:existence of kappa-ample modules}, it is convenient to make the following definition.
\begin{definition} \label{def:coarse cardinality}
  The \emph{coarse cardinality of $\crse X$} is the least cardinal $\kappa$ such that $\crse X$ is coarsely equivalent to a coarse space $\crse Y$ with $\abs{Y}=\kappa$.
\end{definition}

We can then restate \cref{lem:existence of kappa-ample modules} as the following.
\begin{corollary}\label{cor: roe-like algebras vs coarse cardinality}
  The rank-$\kappa$ Roe-like \cstar{}algebras of $\crse X$ are defined if and only if $\crse X$ has coarse cardinality at most $\kappa$.
\end{corollary}

\begin{remark}
  If one only wishes to work with \emph{proper} coarse maps (as was for instance the case in \cite{roe_lectures_2003}), one may of course weaken the ``rank $\kappa$'' condition in \cref{thm:maps-among-roe-algs} to ``local rank $\kappa$'' (see \cref{rem:covering isos}~\cref{rem:covering isos:coe-covered-loc-fin}). In particular, it would also make sense consider rank-$\kappa$ Roe-like algebras of spaces that have coarsely \emph{local cardinality} at most $\kappa$.
\end{remark}

\begin{remark}
  It should be acknowledged that \cref{def:coarse cardinality} only makes sense under the Axiom of Choice, which we always assumed.
\end{remark}

\cref{cor: roe-like algebras vs coarse cardinality} shows that if one wishes to restrict to work with separable Hilbert spaces (\emph{e.g.}\ for K-theoretic considerations), then one must restrict to the setting of coarsely countable coarse spaces.
For the record, we observe that there are countable connected coarse spaces that are not metrizable.
\begin{example}
  Consider the countable set $\ZZ^2$. For every line $l\subset \RR^2$ passing through the origin, let $A_{l} \subset \ZZ^2$ be the set of integer points within distance $1$ from $l$. Observe that if $l$ and $l'$ are distinct lines then $A_{l}\cap A_{l'}$ is finite. Let $\CE$ be the coarse structure generated by all the sets $A_{l}\times A_{l}$. Clearly, $(X,\CE)$ is countable and coarsely connected, and we claim that it is not metrizable.

  By \cref{prop:metrizable coarse structure iff ctably gen}, it suffices to show that $\CE$ does not have a countable cofinal set of controlled entourages.
  Assume it does, and let $\{E_n\}_{n \in \NN}$ be such a family. 
  Observe that the composition $(A_i\times A_i)\circ(A_j\times A_j) = A_i\times A_j$ is contained in $(A_i\cup A_j)\times(A_i\cup A_j)$.
  It follows that for each $E_n$ there are $l_{1,n},\ldots ,l_{k_n,n}$ such that
  \[
    E_n\subseteq (A_{l_{1,n}}\cup\cdots\cup A_{l_{n_k,n}})\times (A_{l_{1,n}}\cup\cdots\cup A_{l_{n_k,n}}).
  \]
  Since there are uncountably many lines, we can choose $l$ that is different from every $\{l_{i,j}\}_{i,j}$. It follows that $A_l\times A_l$ has finite intersection with $E_n$ for every $n \in \NN$. Since $A_l$ is infinite, this contradicts the cofinality of $\{E_n\}_{n \in \NN}$.
\end{example}

\subsection{K-theory of Roe algebras} \label{sec:k-theory}
Roe algebras were originally defined for their K-theory~\cites{roe_index_1996,roe_lectures_2003}, as the K-theory of the Roe algebra of a non-compact Riemannian manifold is a natural place to define higher indexes of certain pseudo-differential operators. It is therefore due to verify that the K-theory of the Roe algebras of general coarse spaces is well behaved as well. As we shall sketch below, the classical proof (see \emph{e.g.} \cite{willett_higher_2020}*{Theorem 5.1.15}) goes through without difficulty.
\begin{proposition}\label{prop:covering isometries are unique in K-theory}
  Let $\crse{f\colon X\to Y}$ be a proper coarse map, $\CHx$ an $\crse X$-module and $\CHy$ a locally admissible $\crse Y$-module. If $t_0,t_1\colon\CHx\to\CHy$ are isometries covering $\crse f$, then the maps $\Ad(t_0)_*$ and $\Ad(t_1)_*$ induced in $K$-theory coincide.
\end{proposition}
\begin{proof}
  Define $\alpha_0$ and $\alpha_1$ by
  \[
  \begin{tikzcd}[row sep =0]
    \roecstar \CHx \arrow{r}{\alpha_0} & M_2(\roecstar \CHy), \\
    x  \arrow[|->, r] & \begin{pmatrix}
               t_0xt_0^* & 0 \\
               0 & 0
             \end{pmatrix},
  \end{tikzcd}
  \quad
  \begin{tikzcd}[row sep =0]
    \roecstar \CHx \arrow{r}{\alpha_1} & M_2(\roecstar \CHy), \\
    x  \arrow[|->, r] & \begin{pmatrix}
               0 & 0 \\
               0 & t_1xt_1^*
             \end{pmatrix}.
  \end{tikzcd}
  \]
  Identifying $K_*(M_2(\roecstar \CHy)) \cong K_*(\roecstar\CHy)$, it is enough to verify that $(\alpha_0)_*=(\alpha_1)_*$.
  Since $t_0$ and $t_1$ are isometries, the operator
  \begin{equation}\label{eq:unitary K-theory}
    u\coloneqq
    \begin{pmatrix}
      1-t_0t_0^* & t_0t_1^* \\
      t_1t_0^*  & 1- t_1t_1^*
    \end{pmatrix}
    \in M_2(\CB(\CHy))
  \end{equation}
  is a unitary involution such that $\alpha_1(x) = \Ad(u)(\alpha_0(x))$ for every $x \in \roecstar \CHx$. By \cref{lem: operators with same coarse support are close},  the operator $t_it_j^*$ has controlled propagation for $i,j=0,1$.
  Since $\CHy$ is locally admissible, we know by \cref{cor:cpc is multiplier of roec} that $\cpcstar \CHy$ is contained in the multiplier algebra of $\roecstar\CHy$. It follows that $u$ is a multiplier of $M_2(\roecstar\CHy)$.
  It is then well-known that $\Ad(u)_*$ defines the identity map in $K$-theory (this is seen by passing once more to $2$-by-$2$ matrices and observing that
  \[
    u_\theta\coloneqq
    \begin{pmatrix}
      u & 0 \\
      0 & 1
    \end{pmatrix}
    \begin{pmatrix}
      \cos(\theta) & -\sin(\theta) \\
      \sin(\theta) & \cos(\theta)
    \end{pmatrix}
    \begin{pmatrix}
      1 &0\\
      0 & u^*
    \end{pmatrix}
    \begin{pmatrix}
      \cos(\theta) & \sin(\theta) \\
      -\sin(\theta) & \cos(\theta)
    \end{pmatrix}
  \]
  is a homotopy between the identity and the conjugation by the matrix having $u$ and $u^*$ on the diagonal).
\end{proof}

\begin{remark}
  The proof of \cref{prop:covering isometries are unique in K-theory} works, in fact, with weaker hypotheses. For the operator $u$ given in \eqref{eq:unitary K-theory} to be a unitary it is enough to ask that $t_0^*t_0=t_1^*t_1$. The last condition can be further relaxed \emph{e.g.} seeking for unitaries of the form
  \[
    u\coloneqq
    \begin{pmatrix}
      1-t_0t_0^* & t_0wt_1^* \\
      t_1w^*t_0^*  & 1- t_1t_1^*
    \end{pmatrix}
  \]
  where $w$ is some partial isometry with controlled propagation.
\end{remark}

An important consequence of \cref{prop:covering isometries are unique in K-theory} is that the choices involved in the construction of the \Star{}homomorphisms covering a coarse map $\crse{f\colon X\to Y}$ in \cref{cor:coarse functions preserve roe-like algebras} do not influence the resulting homomorphism at the level of $K$-theory. In turn, this fixes the lack of functoriality of the Roe algebra construction (see \cref{rmk:roe alg are not functorial}). The formal statement is as follows.
\begin{theorem} \label{thm:k-theory-functor}
  Assigning to each coarse space $\crse X$ of coarse cardinality at most $\kappa$ the $K$-theory groups $K_*(\roecstar[\kappa]{\crse X})$ and to each proper coarse map $\crse{f\colon X\to Y}$ the homomorphism
  \[
    (\phi_{\kappa,\crse f})_*\colon K_*(\roecstar[\kappa]{\crse X})\to K_*(\roecstar[\kappa']{\crse Y})
  \]
  gives a functor from the category of coarse spaces of cardinality at most $\kappa$ and proper coarse maps to the category of $\ZZ/2\ZZ$-graded abelian groups. Moreover, this functor is well-defined up to natural isomorphism.
\end{theorem}
\begin{proof}
  \cref{lem:existence of kappa-ample modules} shows that for any such coarse space $\crse X$ we may choose a $\kappa$\=/ample discrete coarse geometric Hilbert $\CHx$ module of rank $\kappa$, so $K_*(\roecstar[\kappa]{\crse X})$ is well defined. Fix such choices.
  \cref{thm:maps-among-roe-algs} applied to the $\kappa$\=/ample discrete coarse geometric modules $\CHx$ and $\CHy$ yields a spatially implemented \Star{}homomorphism $\phi_{\kappa,\crse f}\colon\roecstar[\kappa]{\crse X}\to\roecstar[\kappa']{\crse Y}$ covering the proper coarse map $\crse{f\colon X\to Y}$. In turn, this defines the required map in $K$-theory.
  Functoriality is now routine to check. Indeed, let $\crse{g\colon Y\to Z}$ be a second proper coarse map. The composition $\phi_{\kappa,\crse g}\circ\phi_{\kappa,\crse f}$ is a spatially implemented embedding that covers $\crse{g\circ f}$. At the level of $K$-theory, \cref{prop:covering isometries are unique in K-theory} shows that
  \[
    \paren{\phi_{\kappa,\crse g}}_\ast\circ \paren{\phi_{\kappa,\crse f}}_* = \paren{\phi_{\kappa,\crse{g\circ f}}}_*
  \]
  regardless of the choice of isometries.

  The naturality is also clear: suppose we make other choices of module $\CH'_{\crse X}$. Applying \cref{thm:covering isometries exist} to the identity map $\cid_{\crse X}$ one obtains a \Star{}isomorphism $\Phi_{\crse X}\colon\roecstar{\CHx}\to\roecstar{\CH'_{\crse X}}$ covering $\cid_{\crse X}$, and hence an isomorphism
  \[
    (\Phi_{\crse X})_*\colon K_*(\roecstar{\CHx})\to K_*(\roecstar{\CH'_{\crse X}}).
  \]
  As above, if $\phi'_{\kappa,\crse f}\colon\roecstar{\CH'_{\crse X}}\to \roecstar{\CH'_{\crse Y}}$ is an embedding covering the proper coarse map $\crse{f\colon X\to Y}$, then \cref{prop:covering isometries are unique in K-theory} shows that
  \[
    \paren{\phi'_{\kappa,\crse{f}}}_* = \paren{\Phi_{\crse Y}}_*\circ \paren{\phi_{\kappa,\crse f}}_\ast\circ \paren{\Phi_{\crse X}^{-1}}_*,
  \]
  which shows that $\Phi$ defines a natural isomorphism.
\end{proof}

\appendix

\section{Comparison with the Roe algebras of Bunke--Engel} \label{sec:differences with Bunke-Engel}
In the appendix we overview some of the main differences between our approach to Roe-like algebras and that of Bunke--Engel \cite{bunke2020homotopy}.

\subsection{Working category}
Some very important stylistic differences follow from the choice of category of interest. In our definition of the coarse category \Cat{Coarse} we do identify close maps (\emph{i.e.}\ morphisms are equivalence classes of controlled functions). On the contrary, for Bunke--Engel morphisms are controlled functions, without further identifications.\footnote{\, In \cite{coarse_groups}, this is called the pre-coarse category \Cat{PreCoarse}.}
In a sense, the \emph{moral} difference is that their category allows to distinguish between different points in a coarse space.
This is a double-edged sword, for, on the one side, it makes it easier to give certain definitions in terms of fixed representatives, while, on the other side, one is then forced to work with those fixed representatives, and changing them requires extra effort. Working in \Cat{Coarse} allows for slicker arguments.

Additionally, \cite{bunke2020homotopy} actually works on a category of \emph{bornological} coarse spaces \cite{bunke2020homotopy}*{Definition 2.11}. That is, the theory they develop gives some freedom to enlarge the family of bounded sets by including a bornology in the definitions. Doing this changes the definition of locality and, most importantly, properness (for this latter reason something similar is also done in \cite{coarse_groups}*{Appendix~C.2}). This is important in their setup, because only proper maps are sufficiently well-behaved for the topological investigations they wish to undertake. However, the notion of properness directly arising from the coarse structure is too stringent, as it precludes, for instance, the existence of product objects.

On the contrary, in our study of Roe-like algebras we generally do not need the properness assumption. We therefore decided against developing this additional aspect, even though we expect that everything here developed goes through without difficulty to bornological coarse spaces as well.

\subsection{Broad goals}
A further reason for the discrepancies between this work and~\cite{bunke2020homotopy} is to be found in the broad goals of the research. Both works are interested in the Roe algebras of proper metric spaces, but we are just as much concerned with the \cstar{}algebras $\cpcstar\variable, \qlcstar\variable$ and $\uroecstar\variable$ of general coarse spaces as studied in~\cites{braga_rigid_unif_roe_2022,bbfvw_2023_embeddings_vna,braga_farah_rig_2021,spakula_maximal_2013}, among many others.
One main goal of this work was to build a framework to deal with all these constructions simultaneously.
Moreover, as explained in the introduction, we are rather interested in various operator-algebraic properties of these \cstar{}algebras, and we developed this language with an eye towards the theory of rigidity of Roe-like algebras~\cite{rigid}.

On the contrary,~\cite{bunke2020homotopy} focuses more on the $K$-theoretical aspects of Roe algebras. Its main goal is to axiomatize the coarse homology theories and construct a spectrum-valued $K$-homology functor in the most general setting possible.
Their main applications go in the direction of the coarse Baum--Connes conjecture \cites{bunke-engel-2020-ass-maps,bunke-engel-2020-coarse-cohom}.
The technical differences this entails, at the level of the modules and algebras, will be discussed shortly.

\subsection{Modules} \label{sec:app-modules}
A large part of this work was finding a reasonable ``minimal'' set of properties needed to define a useful notion of $\crse X$-module.
As explained, the hard constraint was to find a common framework that allows to easily work with uniform Roe algebras of discrete (coarse) spaces or Roe algebras of metric spaces (say, Riemannian manifolds) \emph{at the same time}.

In \cite{bunke2020homotopy} these problems have been sidestepped by defining \emph{$X$-controlled modules} as non-degenerate representations of the algebra of \emph{all} $\CCC$-valued bounded functions of $X$ (see~\cite{bunke2020homotopy}*{Definition~8.1}). It makes sense to do as much in their context, as in their category one is always allowed to distinguish points.
One drawback of this approach is that classical geometric modules~\cite{willett_higher_2020} are generally not $X$-controlled modules. This requires taking some extra steps when dealing with manifolds (see \cite{bunke2020homotopy}*{Example~8.23}).

Besides the above philosophical differences, there are also a few technical/terminological discrepancies between this text and~\cite{bunke2020homotopy}.
\begin{itemize}
  \item An $X$-controlled module is \emph{determined on points} \cite{bunke2020homotopy}*{Definition~8.3} if and only if it is a discrete $\crse X$-module with respect to the trivial partition $X=\bigsqcup_{x\in X}\{x\}$.
  \item An $X$-controlled module is \emph{ample} \cite{bunke2020homotopy}*{Definition~8.13} if and only if it is discrete and $\aleph_0$\=/ample in our terminology.
\end{itemize}

In fact, Bunke--Engel~\cite{bunke2020homotopy} only work with locally-separable modules, and often even locally finite rank ones (cf.~\cite{bunke2020homotopy}*{Definition~8.2}). This is due to the aforementioned explicit goal of constructing a spectrum valued coarse $K$-homology theory (see \cite{bunke2020homotopy}*{Remark~8.36 and Corollary~8.63}). For instance, when they introduce \emph{Roe categories} in~\cite{bunke2020homotopy}*{Definitions~8.74 and 8.76}, they require the modules to be locally finite rank precisely because of the ``unitary'' hypothesis of \cite{bunke2020homotopy}*{Corollary~8.63}.
This constraint is explained in~\cite{bunke2020homotopy}*{Remark~8.77}. There it is also stated that ``\textit{For non-unital C*-algebras it would be necessary to redefine the notion of equivalence using a generalization to C*-categories of the notion of a multiplier
algebra of a C*-algebra.
But we were not able to fix all details of an argument which extends Corollary 8.63 from the unital to the non-unital case}''.
In view of this, it would be interesting to consider non-locally finite rank Roe categories as in \cite{bunke2020homotopy} knowing that the multiplier algebra of $\roecstar{\variable}$ is, in fact, often equal to $\cpcstar{\variable}$, see \cite{rigid}.

\subsection{Operators}
Given two $X$-controlled modules $\CH$ and $\CH'$, Bunke--Engel define the support $\supp(t)$ of an operator $t\colon \CH\to \CH'$ \cite{bunke2020homotopy}*{Definition~8.8} as
\begin{equation}\label{eq:appendix:support}
  \bigcap\bigbraces{R\bigmid R\subseteq X\times X,\ \text{s.t. }\forall Y\subseteq X,\  t(\chf Y(\CH))\subseteq \chf{R(Y)}(\CH')}.
\end{equation}
Such a definition is only useful for their notion of $X$-controlled modules, and one cannot easily use it in modules that are defined as $L^2$-spaces with respect to non-atomic measures: in such a setting $\chf Y$ is only defined for measurable sets, and even if we were to restrict to measurable sets so that \eqref{eq:appendix:support} makes sense, the resulting intersection would always end up being empty. Moreover, without further non-degeneracy assumptions, the support defined by \eqref{eq:appendix:support} may behave rather erratically. In most cases, one would only use this definition for $X$-modules that are \emph{determined on points} (see above).

One substantial notational difference is that our quasi-local operators are called \emph{weakly quasi-local} in \cite{bunke2020homotopy}*{Definition~8.29}. Their definition of quasi-locality is a slightly stronger notion, which does not agree with the classical definition of quasi-locality for non-geodesic metric spaces.
Moreover,~\cite{bunke2020homotopy} also has a notion of \emph{locally finite operator} \cite{bunke2020homotopy}*{Definition~8.24}, which is---at least formally---a stronger version than what we call locally finite rank.

\subsection{Roe-like algebras of modules}
Here our notation diverges quite dramatically. Given an $X$\=/module $H$, they define \cite{bunke2020homotopy}*{Definition~8.34}
\begin{itemize}
  \item $C^*(\CH)$ is generated by locally finite operators of controlled propagation;
  \item $C_{\rm lc}^*(\CH)$ is generated by locally compact operators of controlled propagation;
  \item $C^*_{\rm ql}(\CH)$ is generated by locally finite operators that are quasi-local;
  \item $C^*_{\rm lc,ql}(\CH)$ is generated by locally compact operators that are quasi-local.
\end{itemize}
Thus, their $C_{\rm lc}^*(\CH)$ is our $\roecstar\CH$ and none of the others appear in this work. For technical reasons, Bunke--Engel prefer to work with $C^*(\CH)$: this is not a great limitation, as it often coincides with the Roe algebra~\cite{bunke2020homotopy}*{Remark~8.36}.

It is an interesting feature that, even if a classical geometric module $\CH$ is \emph{not} an $X$-controlled module in the sense of~\cite{bunke2020homotopy}*{Definition~8.1}, one can usually construct a companion $X$-controlled module using the same underlying Hilbert space $\CH$. The Roe algebra defined using this new $X$-controlled module structure coincides with the classical Roe algebra \cite{bunke2020homotopy}*{Remark 8.43}.

\subsection{Roe algebras of spaces}
In this regard, the approach of~\cite{bunke2020homotopy} is completely different from ours. We content ourselves with fixing a cardinal $\kappa$ and defining the Roe algebra of rank $\kappa$ for those coarse spaces where this can be done. Instead,~\cite{bunke2020homotopy} constructs a \cstar{}category by considering the class of all locally finite $X$-controlled modules at the same time. Morphisms of this \cstar{}category are obtained by bounded operators of controlled propagation~\cite{bunke2020homotopy}*{Definition~8.74}.
If $f\colon X\to Y$ is a proper controlled map and $\CH$ is a (locally finite) $X$-controlled module, then one may define a (locally finite) $Y$-controlled module $f_*(\CH)$ (using the same Hilbert space). Since they are considering \emph{all} controlled modules at the same time, this allows~\cite{bunke2020homotopy} to show that this construction is functorial.
One may then use $K$-theory of \cstar{}categories to define the coarse $K$-homologies of coarse spaces $K\CX_*(X)$ and $K\CX_{\rm ql,{*}}(X)$. \cite{bunke2020homotopy}*{Theorem 8.88} shows that if $\CH$ is an ample $X$-controlled module then  $K\CX_*(X)$ and $K\CX_{\rm ql,{*}}(X)$ agree with $K_*(C^*(\CH))$ and $K_*(C_{\rm ql}^*(\CH))$ respectively.

\bibliographystyle{amsalpha}
\bibliography{BibRigidity.bib}

\end{document}